\documentclass[11pt]{article}
\usepackage[utf8]{inputenc}
\usepackage[T1]{fontenc}
\usepackage[a4paper, margin=1in, nohead]{geometry}
\usepackage{amssymb, amsmath, amsthm, amsfonts, bbm, dsfont, mathrsfs}
\usepackage{graphicx, graphics, enumerate}
\usepackage{booktabs, caption, subcaption, multirow, float, url, yfonts}
\usepackage{comment}
\usepackage{color}
\usepackage{xcolor}
\usepackage{indentfirst}
\usepackage[hidelinks]{hyperref}

\newtheorem{prop}{\bf Proposition}[section]
\newtheorem{cor}[prop]{\bf Corollary}

\newtheorem{rem}[prop]{\bf Remark}
\newtheorem{thm}[prop]{\bf Theorem}
\newtheorem{lemma}[prop]{\bf Lemma}

\newtheorem{hyp}[prop]{\bf Assumption}

\newcommand{\norme}[1]{\left\|#1\right\|}
\newcommand{\normesmall}[1]{\|#1\|}
\newcommand{\normebig}[1]{\big\|#1\big\|}

\newcommand{\sumiunn}{\sum_{i=1}^n}

\newcommand{\operatorvec}{{\textnormal{vec}}}
\newcommand{\minvp}{\lambda_{\min}}
\newcommand{\maxvp}{\lambda_{\max}}
\newcommand{\trace}{{ \mathrm{tr} }}

\newcommand{\Id}{\mathbb{I}}

\newcommand{\INTST}{\mathbb{N}^{*}}
\newcommand{\Rb}{\mathbb{R}}

\newcommand{\Rstarplus}{\mathbb{R}_+^*}

\newcommand*{\support}{\mathrm{support}}
\newcommand{\Indicator}{\mathds{1}}

\newcommand{\convD}{
\underset{n\to +\infty}{\overset{d}   {\longrightarrow}}}
\newcommand{\convP}{
\underset{n\to +\infty}{\overset{P}   {\longrightarrow}}}

\newcommand*{\eventA}{{\mathcal{A}}}
\newcommand*{\eventB}{{\mathcal{B}}}
\newcommand*{\eventC}{{\mathcal{C}}}

\newcommand{\iid}{\text{i.i.d.}}

\newcommand{\Normale}{\mathcal{N}}
\newcommand*{\QuantileGauss}{{ q_{\Normale(0,1)} }}
\newcommand*{\QuantileGaussAtsmall}[1]{\QuantileGauss(#1)}
\newcommand*{\QuantileGaussAt}[1]{\QuantileGauss \big( #1 \big)}
\newcommand*{\QuantileGaussAtBig}[1]{\QuantileGauss \Big( #1 \Big)}

\newcommand{\alphamin}{{ \alpha_{\min_{}} }}

\newcommand{\Distribution}{P}
\newcommand{\Distributionde}[1]{\Distribution_{#1}}

\newcommand{\Thetatilde}{\widetilde{\Theta}}
\newcommand*{\ThetaBEE}{\Theta_{\mathrm{EE}}}
\newcommand*{\ThetaBEEsigmaknown}{\Theta_{\mathrm{EE}}^{\textnormal{known}}}
\newcommand*{\Thetasigmaknown}{\Theta^{\textnormal{known}}}
\newcommand*{\ThetaBEEstrict}{\Theta_{\mathrm{EE}}^{\mathrm{strict}}}
\newcommand{\ThetaGauss}{\Theta_{\mathrm{Gauss}}}
\newcommand{\ThetaKasy}{\Theta_{\mathrm{Kasy}}}
\newcommand{\ThetaBC}{\Theta_{\mathrm{BC}}}

\newcommand*{\setDistribXeps}{\mathcal{P}_{X,\eps}}

\newcommand{\Probsousthetan}{{\Prob_{\theta}}}
\newcommand{\Expsousthetan}{{\E_{\theta}}}

\newcommand{\Expsoustheta}{{\E_{\theta}}}
\newcommand{\Varsoustheta}{{\V_{\theta}}}

\newcommand{\Prob}{\mathbb{P}}
\newcommand{\E}{\mathbb{E}}

\newcommand{\V}{\mathbb{V}}

\newcommand{\MeanXi}{\overline{\xi}_{n}}

\newcommand{\lambdatroisn}{\lambda_{3,n}}

\newcommand*{\majorantKurtxi}{K_{\xi}}
\newcommand*{\majorantEnormeXepsquatre}{K_{\eps}}
\newcommand*{\majorantEXXprimeconcentration}{K_{\mathrm{reg}}}
\newcommand*{\minlambdaEXXt}{\lambda_{\mathrm{reg}}}
\newcommand*{\minlambdaEXXtepssquare}{\lambda_{\eps}}
\newcommand*{\majorantEnormeXquatreunplusrho}{K_{X}}

\newcommand{\uu}{{u}}
\newcommand{\CIEdgu}{  \textnormal{CI}_\uu^{\textnormal{Edg}}(1-\alpha,n)}

\newcommand{\CIAsympu}{\textnormal{CI}_\uu^{\textnormal{Asymp}}(1-\alpha,n)}

\newcommand{\CS}{\textnormal{CS}}
\newcommand{\CSalphan}{\CS(1-\alpha,n)}
\newcommand{\CSalphannumero}[1]{{\CS_{\mathrm{#1}}(1-\alpha,n)}}

\newcommand{\CSHoeff}{\CSalphannumero{Hoeff}}

\newcommand{\CIalphannumero}[1]{{\mathrm{CI}_{\mathrm{#1}}(1-\alpha,n)}}
\newcommand{\CIsigmaknown}{\CIalphannumero{\sigmaknown}}
\newcommand{\CIsigmanhat}{\CIalphannumero{\sigmanhat}}

\newcommand{\Edg}{{\mathrm{Edg}_n}}

\newcommand{\DeltanE}{{\Delta_{n,\mathrm{E}}}}
\newcommand{\DeltanB}{{\Delta_{n,\mathrm{B}}}}

\newcommand*{\nuEdg}{\nu_n^{\mathrm{Edg}}}

\newcommand*{\nuApprox}{\nu_n^{\mathrm{Approx}}}
\newcommand*{\nuVar}{\nu_n^{\textrm{Var}}}
\newcommand*{\RnLin}{R_{n,\mathrm{lin}}}
\newcommand*{\RnVar}{R_{n,\mathrm{var}}}

\newcommand*{\QnEdg}{Q_n^{\mathrm{Edg}}}

\newcommand*{\Cn}{C_n}

\newcommand*{\betazero}{\beta_{0}}

\newcommand{\betazeroindice}[1]{\beta_{0,{#1}}}

\newcommand{\betahat}{\widehat{\beta}}

\newcommand*{\AsymptoticVarbetahat}{V}
\newcommand{\Vhat}{{   \widehat{V}}}

\newcommand{\Voracle}{{\mathscr{V}}}

\newcommand{\sigmaknown}{\sigma_{\textnormal{known}}}
\newcommand{\sigmadeux}{\sigma^2}
\newcommand*{\sigmanhat}{ \widehat{\sigma}}
\newcommand*{\sigmadeuxnhat}{ \widehat{\sigma}^2}
\newcommand*{\sigmazeronhat}{ \widehat{\sigma}_0}
\newcommand*{\sigmazerocarrenhat}{\widehat{\sigma}^2_0}

\newcommand{\eps}{\varepsilon}
\newcommand{\epshat}{\widehat{\eps}}

\usepackage{setspace}
\title{Can we have it all? \texorpdfstring{\\}{} Non-asymptotically valid and asymptotically exact
\texorpdfstring{\\}{}confidence intervals for expectations and linear regressions
}

\author{Alexis Derumigny\thanks{Delft University of Technology, Mourik Broekmanweg 6,
2628 XE  Delft, Netherlands.
\newline E-mail address: a.f.f.derumigny@tudelft.nl},
Lucas Girard\thanks{Centre de Recherche en \'Economie et de Statistiques (CREST), CNRS, \'Ecole polytechnique, GENES, ENSAE Paris, Institut Polytechnique de Paris, 91120 Palaiseau, France.
\newline E-mail address: lucas.girard@ensae.fr},
Yannick Guyonvarch\thanks{PSAE-INRAE, 22 Place de l'Agronomie, 91120 Palaiseau, France.
\newline E-mail address:
yannick.guyonvarch@inrae.fr
\newline \indent We would like to thank Laurent Davezies, Xavier D'Haultf\oe{}uille, seminar or conference participants at CREST, University of Surrey, Université Paris-Sud, CIREQ Montreal Econometrics Conference 2022, CEPS Economics seminar, and Paris Econometrics Workshop 2025.
All possible errors remain ours.
Lucas Girard was supported by Project-ANR-23-CE26-0008 MLIVE, headed by Elia Lapenta, during the final period of the project (from September 2024).}
}

\date{\today{}
}

\begin{document}

\maketitle

\begin{abstract}
We contribute to bridging the gap between large- and finite-sample inference
by studying confidence sets (CSs) that are both \textit{non-asymptotically valid} and \textit{asymptotically exact uniformly} (NAVAE) over semi-parametric statistical models.
NAVAE CSs are not easily obtained; for instance, we show they do not exist over the set of Bernoulli distributions.
We first derive a generic sufficient condition: NAVAE CSs are available as soon as uniform asymptotically exact CSs are.
Second, building on that connection, we construct closed-form NAVAE confidence intervals (CIs) in two standard settings -- scalar expectations and linear combinations of OLS coefficients -- under moment conditions only.
For expectations, our sole requirement is a bounded kurtosis.
In the OLS case, our moment constraints accommodate heteroskedasticity and weak exogeneity of the regressors.
Under those conditions, we enlarge the Central Limit Theorem-based CIs, which are asymptotically exact, to ensure non-asymptotic guarantees. Those modifications vanish asymptotically so that our CIs coincide with the classical ones in the limit.
We illustrate the potential and limitations of our approach through a simulation study.

\medskip

\noindent
\textbf{Keywords}: non-asymptotic inference; efficient confidence intervals; heteroskedastic linear models.

\noindent
\textbf{MSC Classification}: 62G15; 62J05.
\end{abstract}

\section{Introduction}
\label{CIchap:sec:introduction}

Mathematical statistics results are often divided into two branches: non-asymptotic (also called finite-sample) and asymptotic results.
In this way, most confidence sets are typically designed to have either good asymptotic properties (based on limiting results, usually the Central Limit Theorem) or good non-asymptotic properties (based on concentration inequalities).
In this paper, we try to bridge this gap and construct confidence sets -- more precisely, intervals -- that are NAVAE, i.e., both \textit{Non-Asymptotically Valid} (their coverage probability is at least their nominal level for any sample size \(n\)) and \textit{Asymptotically Exact}
\textit{uniformly} over the statistical model (their coverage probability uniformly tends to the nominal level as $n \to \infty$).

\medskip

Non-asymptotically exact confidence sets (defined by their coverage probability being exactly equal to their nominal level for any sample size) are automatically NAVAE uniformly on the corresponding parameter set.
However, they can be constructed only in very particular cases, such as the classical Student confidence interval for the expectation of a normal distribution.
Even in seemingly simple cases, constructing uniform NAVAE confidence sets can prove difficult or sometimes even impossible, as displayed in the following result on Bernoulli distributions.
\begin{prop}
    Let \(\mathfrak{Ber}\) be the class of all Bernoulli distributions with parameter \({p \in (0,1)}\).
    For any~\(\alpha \in (0,1)\), there is no confidence set for $p$ that is both~non-asymptotically valid and~asymptotically exact uniformly over $\mathfrak{Ber}$ at nominal level~\(1 - \alpha\).
    \label{prop:impossibility_NAVAE_Bernoulli}
\end{prop}
Confidence sets and their properties 
are formally defined in Section~\ref{sec:not_qual_measure} and a proof of this proposition is given in Appendix~\ref{proof:prop:impossibility_NAVAE_Bernoulli}.
Note that it is not a consequence of the impossibility results by \cite{bahadur1956} or by \cite{bertanha2016impossible}.
It is also unrelated to \cite{dufour1997some}'s impossibility results, which apply to locally almost unidentified parameters; this is not the case here since there is no identification issue.

\medskip

Despite the negative result of Proposition~\ref{prop:impossibility_NAVAE_Bernoulli}, we construct in this paper NAVAE confidence intervals uniformly over infinite-dimensional classes of distributions for two specific parameters:
\begin{itemize}
    \item the expectation of a (potentially non-normal) distribution, discussed in Section~\ref{sec:expectation_example};
    
    \item
    a scalar parameter of the form $\uu'\beta_0$, where $\uu$ is a known vector and $\betazero$ corresponds to the coefficients of a linear regression, discussed in Section~\ref{sec:lin_mod_exo}.
    Individual coefficients fall in that framework.
\end{itemize}

To build such confidence intervals, we only impose moment restrictions on our class of distributions.
In particular, we do not use any parametric assumption, be it in the expectation or in the linear regression case.
In the latter setup, our moment conditions encompass very general regression designs: regressors are only assumed to be orthogonal to error terms, and heteroskedasticity is allowed.
Those moment conditions are necessary to avoid the impossibility results of \cite{bahadur1956} and enable us to leverage the asymptotic normality of the sample mean/OLS estimator.
That property is crucial in our construction.
Indeed, our confidence intervals (CIs) are built from asymptotically negligible modifications of the standard intervals based on the Central Limit Theorem (CLT).
Those modifications rely on bounds on Edgeworth expansions (see, e.g., \cite{esseen1945,cramer1962}, and the introduction in \cite{derumigny2022explicitBE} for additional references) to control the distance of the distribution of a standardized empirical mean to its Gaussian limit and obtain finite-sample guarantees (NAV part). 
Uniformly over the class delineated by our moment conditions, our CIs further coincide asymptotically with the standard ones based on the CLT and are thus uniformly asymptotically exact (AE part).
Besides, that coincidence implies that our intervals are asymptotically of minimal width among NAVAE ones according to the efficiency criterion of \cite{romano2000}.

\medskip

Our confidence intervals for the expectation are related to the ones proposed by \cite{romano2000} and \cite{austern2022efficient}.
Both papers primarily focus on efficiency instead of uniform asymptotic exactness.
Furthermore, they assume that the random variables are bounded or of bounded variation around their expectation.
In contrast, the confidence intervals we propose are NAVAE uniformly over a set of distributions that includes distributions with unbounded support (even after recentering at the expectation).
The recent paper by \cite{waudby2024estimating} proposes new confidence intervals that also rely on the boundedness assumption.
\cite{hall1995} construct confidence intervals that are asymptotically exact uniformly over a certain infinite-dimensional class of distributions, but their CIs are not shown to be non-asymptotically valid. 

\medskip

In the linear regression setting, \cite{gossner2013finite} build two non-asymptotic inference methods, valid under bounded outcome variables and allowing for heteroskedastic errors.
There are several limitations, though, either theoretical or practical: the first method is not exact asymptotically, and neither the first nor the second yield closed-form confidence sets, which may entail cumbersome computations.
\cite{dHaultfoeuille2024robust} and \cite{pouliot2024} propose non-asymptotically exact confidence sets by inverting permutation-based tests. 
In~\cite{dHaultfoeuille2024robust}, non-asymptotic exactness relies on an independence assumption between error terms and the regressors of interest, conditional on the remaining regressors, while in~\cite{pouliot2024}, the imposed condition is that the joint distribution of regressors and errors is invariant to a permutation of the errors. 
There are two main limitations to these approaches: the maintained conditions restrict possible heteroskedasticity patterns, and, unlike ours, the proposed confidence sets are not of closed form.
On the other hand, these methods are shown to be asymptotically exact under assumptions close to ours.
In the statistics literature, \cite{kuchibhotla2024hulc} recently proposed a method to build NAVAE confidence sets on functionals in a broad class of statistical models.
While these results are very general and proved under mild moment assumptions on the data distribution, the proposed method is not efficient in linear regression models.
This means that it delivers confidence intervals that are not of the smallest possible width asymptotically.
In a general \(M\)-estimation framework, \cite{brunel2019nonasymptotic} derive non-asymptotic confidence intervals with closed-form expressions on individual components of the coefficients' vector.
The main drawback of their result is a lack of asymptotic exactness.
All in all, to the best of our knowledge, our procedure is the first to meet the following five criteria at the same time in linear regression: 
\begin{enumerate}
    \item[(i)] non-asymptotic validity,
    \item[(ii)] uniform asymptotic exactness,
    \item[(iii)] efficiency \textit{\`a la} \cite{romano2000},
    \item[(iv)] allowing for arbitrary heteroskedasticity,
    \item[(v)] having a closed-form expression.
\end{enumerate}

\medskip

The rest of the paper is organized as follows.
Section~\ref{sec:not_qual_measure} introduces some notation and a number of quality measures for confidence sets.
In this section, we also present and prove some generic impossibility/possibility results on constructing NAVAE confidence sets.
Section~\ref{sec:expectation_example} discusses how to construct NAVAE confidence intervals for an expectation and illustrates the major trade-offs and strategies in this simpler case.
Some additional material on this question is provided in Appendix~\ref{appendix:sec:illustration_expectation}.
Our CIs for linear regressions are defined and shown to be NAVAE in Section~\ref{sec:lin_mod_exo}, and we further study the rate of convergence of their coverage towards the nominal level in that section.
We discuss the practical implementation of our methods in Section~\ref{sec:practical_considerations_plugin} together with some simulations.
The rest of the proofs and useful intermediary lemmas are reported in Appendices~\ref{appendix:sec:proof_Section1}, \ref{appendix:sec:proofs_section3_expectation}, \ref{appendix:sec:proof_section4_OLS},
and~\ref{appendix:sec:additional_lemmas}.
All confidence intervals introduced in this paper are implemented in the open-source \texttt{R} package \texttt{NAVAECI} \cite{NAVAECI}.

\section{Notation and quality measures for confidence sets}
\label{sec:not_qual_measure}

\subsection{Notation}

For a random variable~$D$, we denote by $\Distributionde{D}$ its distribution and by \(\support(D)\) or \(\support(\Distributionde{D})\) its support.
Similarly, \(\Distributionde{D, \, U}\) denotes the joint distribution of a pair of random variables~\((D, U)\).
For a parameter $\theta$ associated with a given statistical model, \(\Distribution_\theta\) denotes a distribution indexed by that parameter.
Remember that \(\Distribution_{D, \, U} = \Distribution_D \otimes \Distribution_U\) means that $D$ and $U$ are independent.
For any set $\mathcal{D}$, $\mathcal{P}(\mathcal{D})$ denotes the set of all probability distributions on~$\mathcal{D}$.
For any real vector \(u = (u_1, \ldots, u_d)\), \(\norme{u} := (u_1^2 + \ldots + u_d^2)^{1/2}\) denotes its Euclidean or \(\ell_2\)-norm.
For matrices, we consider the operator norm induced by the Euclidean norm: for any real matrix~\(M\), \(\norme{M}\) denotes its spectral norm (also known as the operator 2-norm), namely the square root of the largest eigenvalue of \(M'M\), with \(M'\) the transpose of~\(M\).
We also denote by \(\minvp(M)\) (respectively \(\maxvp(M)\)) the smallest (resp. largest) eigenvalue of~\(M\), \(\trace(M)\) its trace, and \(M^\dagger\) the Moore–Penrose pseudo-inverse of a square matrix~\(M\).
\(\operatorvec(\cdot)\) denotes the vectorization of a matrix, that is, if \(M\) is a \(m \times n\) matrix, \(\operatorvec(M)\) is the \(mn \times 1\) column vector obtained by stacking the columns of the matrix~\(M\) on top of one another.
For any positive integer~\(p\), \(\mathbb{I}_p\) denotes the \(p \times p\) identity matrix.
For any distribution $\Distribution$ and real number ${\tau \in (0,1)}$, $q_\Distribution(\tau)$ denotes the quantile at order~$\tau$ of the distribution.
Since we remain in an independent and identically distributed (\iid{}) setup throughout the article, we sometimes drop the subscript~\(i\) of random variables to lighten notations. 
In other words, if \(D_1, \ldots, D_n\) are \(n\) \iid{} random variables, \(D\) without subscript denotes a generic random variable with the same distribution.
Moreover, for a statistical model $(P_\theta)_{\theta \in \Theta}$ we denote by \(\Probsousthetan\) (respectively \(\Expsoustheta\), $\Varsoustheta$, and so on) the probability (respectively expectation, variance, etc.) with respect to the joint distribution~\(P_\theta^{\otimes n}\).

\subsection{Quality measures for confidence sets}
\label{ssec:qual_measure}

We start by formally defining several attributes of confidence sets that characterize their quality.
We do so in a general framework that encompasses linear models.
Let $(\mathcal{X}, \mathcal{B})$ be a measurable space, and assume that we observe ${n \in \INTST}$ \iid{} $\mathcal{X}$-valued random elements $(\xi_i)_{i=1}^n$ from a common probability space $(\Omega, \mathcal{A}, P_\theta)$, where $P_\theta$ is a distribution from a given statistical model $(P_\theta)_{\theta \in \Theta}$.

\begin{rem}
    In most applications $\mathcal{X} = \Rb^d$ for some $d \in \INTST$ and
    in many cases, the parameter space $\Theta$ can be written as a product space $\Theta = \Rb^p \times \Theta_2$,
    where $\Theta_2$ is a topological space and~\({p \in \INTST}\).
    This arises, for instance, in semi-parametric models, which are ubiquitous in econometrics. 
    A model is said to be semi-parametric when the distribution of the data is characterized by both a finite-dimensional parameter $\theta_1$, which is the one of interest, and an infinite-dimensional ``nuisance'' parameter $\theta_2$, which may belong, for example, to some set of probability distributions.
    In linear regression models, $\Theta_2$ is a non-parametric set of joint distributions for regressors and error terms.
    In such a case, the target $T(\theta_1, \theta_2)$ usually only depends on $\theta_1$, for example, if we want to estimate one of the coefficients of a linear regression.
\end{rem}

We seek to construct a confidence set for a real-valued function $T(\cdot)$ of the parameter~$\theta$.
Loosely speaking, a confidence set is a subset of $\Rb$ computable from the data that aims to contain \(T(\theta)\) with prescribed probability \({1 - \alpha}\), known as its confidence or nominal level, for some user-chosen \({\alpha \in (0,1)}\).
Formally, we can define a method $\CS(\, \cdot \,)$ of constructing confidence sets in several equivalent ways:
\begin{enumerate}
    \item for every $n \in \INTST$, 
    $\CS(\, \cdot \,)$ is a function from $(0, 1) \times \mathcal{X}^n$
    to the set of Borel subsets $\mathcal{B}(\Rb)$;

    \item $\CS(\, \cdot \,)$ is a function from $(0, 1) \times \bigsqcup_{n = 1}^{+\infty} \mathcal{X}^n$
    to $\mathcal{B}(\Rb)$;

    \item $\CS(\, \cdot \,)$ is a function from $(0, 1) \times \INTST \times \mathcal{X}^{\INTST}$
    to $\mathcal{B}(\Rb)$ where
    $\CS(1 - \alpha, n, (\xi_i)_{i \in \INTST})$ only depends on the first $n$-th entries of $(\xi_i)_{i \in \INTST}$;
\end{enumerate}
such that
$\{\omega \in \Omega: \, T(\theta) \in
\CS(1 - \alpha, (\xi_i(\omega))_{i=1}^n)\}$
is measurable for every $\alpha \in (0, 1)$, \(n \in \INTST\), and $\theta \in \Theta$.
This is the minimal requirement to be able to define the coverage probabilities
$\Probsousthetan\big(\text{CS}(1 - \alpha, (\xi_i)_{i=1}^n) \ni T(\theta) \big)$.
For a given nominal level, we then write a confidence set as $\CS(1 - \alpha, (\xi_i)_{i=1}^n)$, simplified to $\CSalphan$.
For brevity, we also use the notation $\CSalphan$ to denote the sequence of confidence sets
\((\CS(1 - \alpha, (\xi_i)_{i=1}^n))_{n \in \INTST}\).
Confidence intervals (CIs) can be seen as a particular case of confidence sets (CSs) that are intervals.

\medskip

Several criteria exist to assess the quality of $\CSalphan$ for a given $\alpha$.
$\CSalphan$ is said to be \emph{asymptotically valid pointwise over $\Theta$ at level $1 - \alpha$} if
\begin{equation}
\label{eq:pointwise_valid_asymp}
    \forall \theta \in \Theta, \,\liminf_{n \to +\infty}
    \Probsousthetan \big(\CSalphan \ni T(\theta)\big)
    \geq 1-\alpha.
\end{equation}
$\CSalphan$ is said to be \emph{asymptotically exact pointwise over $\Theta$ at level $1 - \alpha$}
if
\begin{equation}
\label{eq:pointwise_exact_asymp}
    \forall \theta \in \Theta, \lim_{n \to +\infty}
    \Probsousthetan \big(\CSalphan \ni T(\theta)\big)
    = 1-\alpha.
\end{equation}

A stronger asymptotic criterion exists:
$\CSalphan$ is said to be \emph{asymptotically valid uniformly over $\Theta$ at level $1 - \alpha$} if
\begin{equation}
\label{eq:unif_valid_asymp}
    \liminf_{n \to +\infty} \inf_{\theta \in \Theta} 
    \Probsousthetan \big(\CSalphan \ni T(\theta) \big)
    \geq 1-\alpha,
\end{equation}
and $\CSalphan$ is said to be \emph{asymptotically exact uniformly over $\Theta$ at level $1 - \alpha$} if
\begin{equation}
\label{eq:unif_exact_asymp}
    \lim_{n\to+\infty} \, \sup_{\theta\in\Theta} \bigg|
    \Probsousthetan \big(\CSalphan \ni T(\theta) \big) - (1-\alpha) 
    \bigg| = 0.
\end{equation}
We say that $\CSalphan$ is \emph{non-asymptotically valid over $\Theta$ at level $1 - \alpha$} if
\begin{equation}
\label{eq:definition_non-asymptotic_validity}
    \forall n \geq 1, \,
    \forall \theta\in\Theta, \,
    \Probsousthetan\big(\CSalphan \ni T(\theta) \big)
    \geq 1-\alpha.
\end{equation}
This property evolves into \emph{non-asymptotic exactness over $\Theta$ at level $1 - \alpha$} if the following stronger condition holds
\begin{equation}
\label{eq:definition_non-asymptotic_exactness}
   \forall n\geq 1, \,
   \forall \theta\in\Theta, \,
   \Probsousthetan \big(\CSalphan \ni T(\theta) \big)
   = 1-\alpha.
\end{equation}
For any of those definitions, when the level $1 - \alpha$ is not specified, it means that the method of constructing confidence sets satisfies the corresponding property for all \(\alpha \in (0,1)\).
For instance, we say that $\CS(\, \cdot \,)$ is \emph{asymptotically exact pointwise over $\Theta$} if \(\CSalphan\) is so at all levels~\(1 - \alpha\).

\medskip

These definitions follow usual conventions and enable us to compare the quality of competing CSs along several dimensions.
Note that these concepts are not exclusive: Figure~\ref{tab:summary_quality_measures} summarizes the implications between all those quality measures.
For instance, exactness always implies validity.
On the other hand, CSs that are valid but not exact are called \textit{conservative} since their coverage probability is strictly larger than the targeted nominal level.

\begin{figure}[H]
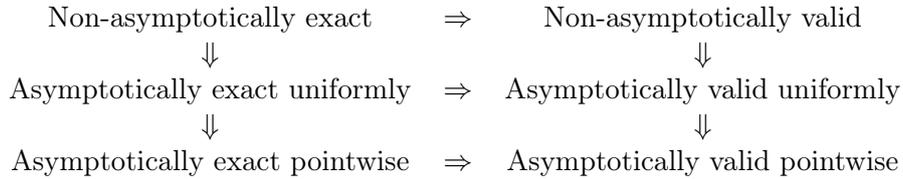

    \centering
    \begin{tabular}{ccc}
        Non-asymptotically exact & $\Rightarrow$ &
        Non-asymptotically valid \\
        $\Downarrow$ &  & $\Downarrow$ \\
        Asymptotically exact uniformly  & $\Rightarrow$ &
        Asymptotically valid uniformly \\
        $\Downarrow$ &  & $\Downarrow$ \\
        Asymptotically exact pointwise  & $\Rightarrow$ &
        Asymptotically valid pointwise
    \end{tabular}
    \caption{Relationships between the different quality measures of a confidence set.}
    \label{tab:summary_quality_measures}
\end{figure}

Unlike non-asymptotic properties, asymptotic ones can be further characterized by their pointwise or uniform nature.
Asymptotic uniform validity relates to the notion of ``honesty.''
Compared to pointwise guarantees, it can be argued as more reliable regarding CSs' finite-sample performance (\cite{ArmstrongKolesarQE2020SimpleHonnestCI}).

\medskip

Property~\eqref{eq:pointwise_exact_asymp} (pointwise AE) is what applied econometricians implicitly have in mind when they rely on the asymptotic normality of an estimator to conduct inference.
This property is usually achievable over a non-parametric set $\Theta$.
In models where \eqref{eq:pointwise_exact_asymp} holds, it is often possible to define a large subset $\widetilde{\Theta} \subset \Theta$ on which \eqref{eq:unif_exact_asymp} (uniform AE) is verified; see e.g. \cite{kasy2019}.
Regarding finite-sample inference, Property~\eqref{eq:definition_non-asymptotic_validity} (the NAV part of NAVAE) has been shown to hold in many models.
These results are predominantly found in the mathematical statistics literature (see \cite{brunel2019nonasymptotic} for a recent illustration).
These are powerful findings.
Yet, the resulting CSs are usually asymptotically conservative, uniformly and even pointwise; in other words, they are not AE, hence not NAVAE confident sets.
Finally, the strongest notion \eqref{eq:definition_non-asymptotic_exactness}, namely non-asymptotic exactness (which, a fortiori, implies the NAVAE property uniformly over the statistical model), can be obtained at the cost of placing fairly strong restrictions on $\Distribution_{\theta}$, for instance by imposing that $\Distribution_{\theta}$ belong to a parametric family.
Overall, it appears challenging to build NAVAE confidence sets uniformly over a non-parametric set of distributions.
We continue this discussion in Appendix~\ref{appendix:sec:illustration_expectation}, focusing on inference on an expectation.

\medskip

The implications that are not displayed in Figure \ref{tab:summary_quality_measures} are not satisfied.
This is straightforward to see since there exist confidence sets that are, for example, asymptotically valid pointwise but not non-asymptotically valid.
Nevertheless, if we are given confidence sets that are asymptotically valid or exact \emph{uniformly} over some parameter set $\Theta$, it is always possible to construct confidence sets that are non-asymptotically valid over $\Theta$.
This result is presented in the following proposition.

\begin{prop}[Obtaining non-asymptotic validity from asymptotic uniform validity or exactnessn]
~ 
    \begin{enumerate}
        \item[(i)] 
        If there exists an asymptotically valid confidence set \(\CSalphan\) uniformly over $\Theta$ at level \(1 - \alpha \in (0,1)\),
        then, for every $\widetilde\alpha > \alpha$,
        there exists a non-asymptotically valid confidence set \(\widetilde{\CS}({1-\widetilde\alpha, n})\) over $\Theta$ at level~\(1-\widetilde{\alpha}\).
        Furthermore, if $\CS({1-\alpha, n})$ is almost surely different from~\(T(\Theta)\) for $n$ large enough, so is $\widetilde{\CS}({1-\widetilde\alpha, n})$.
        
        \item[(ii)]         
        If there exists a method $\CS(\, \cdot \,)$ of constructing confidence sets that is asymptotically exact uniformly over $\Theta$,
        then, there exists a method \(\widetilde{\CS}(\, \cdot \,)\) of constructing confidence sets that is both non-asymptotically valid and asymptotically exact (NAVAE) uniformly over $\Theta$.
    \end{enumerate}
    \label{prop:NAVAE_from_asymptoticUnif}
\end{prop}

\begin{proof}[Proof of Proposition~\ref{prop:NAVAE_from_asymptoticUnif}]
    \textit{(i)}
    Let $\widetilde\alpha > \alpha$ and $f(n) := \inf_{\theta\in\Theta}
    \Probsousthetan \big(\CSalphan\ni T(\theta) \big)$, for any positive integer~\(n\).
    By definition, $\liminf_{n\to+\infty} f(n) \geq 1-\alpha$.
    As \(1 - \widetilde{\alpha} < 1 - \alpha\), there exist $n^* \in \INTST$ such that, for every $n \geq n^*$,
    $f(n)
    \geq 1 - \widetilde\alpha$.
    Then, we define $\widetilde{\CS} (1-\widetilde\alpha, n) := \CS(1-\alpha, n)$ for $n \geq n^*$ and $\widetilde{\CS}(1-\widetilde\alpha, n) := T(\Theta)$ for $n < n^*$.
    Note that $\widetilde{\CS}(1-\widetilde\alpha, n)$ is indeed non-asymptotically valid over~\(\Theta\) at level~\(1 - \widetilde{\alpha}\).
    The second part of \textit{(i)} is a direct consequence of the construction of $\widetilde{\CS}(1-\widetilde\alpha, n)$.

    \medskip
    
    \noindent
    \textit{(ii)}
    Let $f(\alpha, n) := \sup_{\theta\in\Theta} \left| \Probsousthetan\big(\CSalphan\ni T(\theta) \big) - (1-\alpha) \right|$ for any~\(n \in \INTST\) and \(\alpha \in (0,1)\).
    By definition, $\lim_{n\to+\infty} f(\alpha, n) = 0$ for every fixed $\alpha \in (0, 1)$.
    Fix \(\alpha \in (0,1)\).
    For every integer $k \geq 2$, let $\alpha_k := \alpha \times (1 - 1/k)$ and $n_k$ such that, for all $n \geq n_k$,
    $f(\alpha_k, n) \leq \alpha / k$.
    Then, we construct an increasing integer sequence \((n'_k)_{k \geq 2}\) by $n'_2 := n_2$ and $n'_k := 1 + \max(n'_{k-1}, n_k)$ recursively.
    Thus, we can define
    $\widetilde{\CS}(1-\alpha, n)
    := \CS(1-\alpha_k, n)$
    for $n$ such that $n'_k \leq n < n'_{k+1}$ and $\widetilde{\CS}(1-\alpha, n) := T(\Theta)$
    for $n$ such that $n < n_2$.
    We now check the properties of \(\widetilde{\CS}(1-\alpha, n)\).
    First, for \(n\) large enough,
    \begin{align*}
        \sup_{\theta\in\Theta} \Big| \Probsousthetan\big(\widetilde{\CS}(1-\alpha, n) \ni T(\theta) \big) &- (1-\alpha) \Big|
        = \sup_{\theta\in\Theta} \left| \Probsousthetan\big(\CS(1-\alpha_k, n) \ni T(\theta) \big) - (1-\alpha) \right| \\
        &\leq \sup_{\theta\in\Theta} \left| \Probsousthetan\big(\CS(1-\alpha_k, n) \ni T(\theta) \big) - (1-\alpha_k) \right| + \alpha / k \\
        &= f(\alpha_k, n) + \alpha / k
        \leq 2 \alpha / k,
    \end{align*}
    so $\widetilde{\CS}(1-\alpha, n)$ is asymptotically exact uniformly over $\Theta$ because \(k \to +\infty\) when \(n \to +\infty\).
    Second, $\inf_{\theta\in\Theta}
    \Probsousthetan \big(\widetilde{\CS}(1-\alpha, n) \ni T(\theta) \big)
    \geq 1 - \alpha_k - \alpha / k = 1 - \alpha$, and $\widetilde{\CS}(1-\alpha, n)$ is indeed non-asymptotically valid over $\Theta$.
    Those properties hold for every~\(\alpha \in 0,1)\), yielding the desired NAVAE property for the method \(\widetilde{\CS}(\,\cdot\,)\) of constructing confidence sets.
\end{proof}

\begin{rem}
    Inspecting the previous proof shows that a sufficient condition for the existence of a NAVAE confidence set uniformly over $\Theta$ at a given level $1-\alpha$ is the existence of a sequence $(\alpha_k)_{k \geq 1} \in (0, 1)$ such that (1) $\alpha_k \to \alpha$ as $k \to \infty$; (2) $\forall k \in \INTST, \, \alpha_k < \alpha$; (3) an asymptotically exact confidence set $\CS(1-\alpha_k, n)$ uniformly over $\Theta$ exists at each level $1 - \alpha_k$, \(k \in \INTST\).
\label{rem:better_NAVAE_from_asymptoticUnif}
\end{rem}

As a direct consequence of Proposition~\ref{prop:NAVAE_from_asymptoticUnif}(i), we can state the subsequent corollary.
\begin{cor}    
    If there exists a method $\CS(\, \cdot \,)$ of constructing confidence sets that is asymptotically valid uniformly over $\Theta$,
    then, there exists a method \(\widetilde{\CS}(\, \cdot \,)\) of constructing confidence sets that is
    non-asymptotically valid over $\Theta$.
\end{cor}

Note that the impossibility claim made on Bernoulli distributions in Proposition~\ref{prop:impossibility_NAVAE_Bernoulli} can actually be strengthened with the help of Proposition~\ref{prop:NAVAE_from_asymptoticUnif}: not only NAVAE confidence sets do not exist over $\mathfrak{Ber}$, but even finding a method of constructing confidence sets that is asymptotically exact uniformly over $\mathfrak{Ber}$ is impossible.

\begin{cor}
    For any method~\(\CS(\, \cdot \,)\) of constructing confidence sets, the set of real numbers   
    \begin{equation*}
        \big\{
        \alpha \in (0, 1): \, (\CSalphan)_{n \geq 1}
        \textnormal{ is asymptotically exact uniformly over } \mathfrak{Ber}
        \textnormal{ at level } 1 - \alpha
        \big\}
    \end{equation*}
    has an empty interior.
    Consequently, there is no method of constructing confidence sets that is asymptotically exact uniformly over~\(\mathfrak{Ber}\).    
    \label{cor:impossibility_asympt_exact_Bernoulli_2}
\end{cor}

\begin{proof}
    If the interior of this set was not empty, we could find an
    interval $[\alpha_-, \alpha_+]$ included inside.
    Therefore, by Remark~\ref{rem:better_NAVAE_from_asymptoticUnif},
    we would obtain the existence of a NAVAE confidence set uniformly over~\(\mathfrak{Ber}\) at nominal level~\(1 - \alpha_+\).
    But this is not possible by Proposition~\ref{prop:impossibility_NAVAE_Bernoulli}.
\end{proof}

We conjecture that the set mentioned in Corollary~\ref{cor:impossibility_asympt_exact_Bernoulli_2} is empty: for any $\alpha \in (0, 1)$, there is no confidence set that is asymptotically exact uniformly over $\mathfrak{Ber}$ at level~\(1 - \alpha\).
Extending Corollary~\ref{cor:impossibility_asympt_exact_Bernoulli_2} in this direction would be non-trivial and is left for future research.

\medskip

The ``automatic'' construction of a NAVAE confidence set $\widetilde{\CS}(1-\alpha, n)$ as given in the proof of Proposition~\ref{prop:NAVAE_from_asymptoticUnif}(ii) is mainly interesting from a theoretical point-of-view since the approach is difficult to apply in practice as such (the quantities \(f(\alpha_k,n)\) are not available in general).
Still, if an upper bound \(g(\cdot, \cdot)\) on \(f(\cdot, \cdot)\) is available, then one can construct the suitably enlarged confidence sets introduced in the proof of Proposition~\ref{prop:NAVAE_from_asymptoticUnif}(ii), simply using \(g\) instead of~\(f\).
The set constructed from $g$ remains non-asymptotically valid, and it is asymptotically exact uniformly over $\Theta$ if and only if $\lim_{n \to +\infty} g(\alpha, n) = 0$ for all $\alpha$.
Inspired by this construction, we use bounds on the distance to the asymptotic normality for sample means/OLS estimators in the rest of the paper to suitably enlarge CLT-based confidence intervals.

\medskip

Among NAVAE CIs, it is sometimes possible to ask for more: in particular, we could aim for efficient intervals \textit{\`a la} 
\cite{romano2000} (see their Theorem~2.1(i) and Remark~3).
We briefly summarize their results using our notations.
A non-asymptotically valid confidence interval $[U_n, L_n]$ is said to be \textit{efficient} if its width is asymptotically minimal among all non-asymptotically valid confidence intervals.
Under some regularity conditions, any efficient confidence interval of a parameter $T(\theta)$ estimated by some suitable $\widehat{T(\theta)}$ must be of the form
\begin{equation}
\label{eq:def:efficient_CI}
    \widehat{T(\theta)}
    \pm \frac{\tau(\theta)}{\sqrt{n}} \QuantileGaussAt{1 - \alpha/2}
    + o_{P_\theta}(n^{-1/2})
\end{equation}
where $\tau(\theta)^2$ is the corresponding asymptotic variance.
In other words, an efficient CI must be an asymptotically negligible enlargement of a CLT-based confidence interval.
As shown in the next two sections, all the CIs we build are exactly of the form~\eqref{eq:def:efficient_CI}, which implies they are efficient in addition to being NAVAE.

\section{NAVAE confidence intervals for expectations}
\label{sec:expectation_example}

In this section, we are interested in presenting a simple case: 
conducting inference on the expectation of a real random variable that admits (at least) a non-zero finite variance, that is, with the formalization of the previous section,
\begin{equation*}
    \Theta 
    = \Rb \times \left\{P \in \mathcal{P}(\Rb) :
    \int x \, dP = 0, \, \int x^2 \, dP \in (0, +\infty) \right\}
    =: \Theta_1 \times \Theta_2,
\end{equation*}
where the parameter is the pair combining an expectation and 
the distribution of the corresponding centered variable (which is a nuisance parameter).
In this way, we can decompose any random variable $\xi$ following $P_\theta$ where $\theta = (\theta_1, P)$ by $\xi := \theta_1 + V$ where $V \sim P$.
Thus, with our formalization, $p=1$ and $T(\theta_1, P) = \theta_1$.
Note that, in this sense, $\Theta$ is in bijection with the set of univariate distributions with finite but non-zero variance, and we will use this identification in the following.
We assume the variance to be non-zero to remove pathological cases.

\medskip

In this canonical scenario, we illustrate the challenges of building NAVAE confidence intervals that ``converge'' asymptotically to standard CIs based on the CLT.
Henceforth, we fix a desired nominal level \(1 - \alpha \in (0,1)\).

\subsection{Setting and motivation}
\label{ssec:setting_and_motivation}

Let $\sigmadeuxnhat:= n^{-1} \sum_{i=1}^n(\xi_i-\overline{\xi}_n)^2$, with $\overline{\xi}_n := n^{-1} \sumiunn \xi_i$. We know by the Central Limit Theorem (CLT) and Slutsky's lemma that
\begin{equation}
\label{eq:definition_CI_CLT_for_expectation}
    \CIalphannumero{CLT} := 
    \left[
    \overline{\xi}_n \pm  \frac{\widehat{\sigma}}{\sqrt{n}} \QuantileGaussAt{1 - \alpha/2}
    \right]
\end{equation}
is asymptotically exact pointwise over $\Theta$ at level~\(1 - \alpha\).
Among practitioners, it is the most common way of conducting inference on an expectation and, therefore, a natural candidate on which to base a uniformly NAVAE confidence interval.

\medskip

Our goal is to propose an asymptotically negligible modification of \(\CIalphannumero{CLT}\) that is nonetheless sufficient to make the resulting confidence interval NAVAE over a large (i.e., non-parametric) subset of $\Theta$.
As suggested in Section~\ref{ssec:qual_measure}, an appealing idea consists in controlling non-asymptotically the distance between the unknown distribution of the sample mean and its Gaussian limiting distribution.
To do so, we leverage Berry-Esseen inequalities
(\cite{berry1941, esseen1942})
or Edgeworth expansions as these tools ``quantify'' the CLT (by controlling, for any sample size, some distance between the distribution of a normalized sum and the standard Gaussian distribution \(\mathcal{N}(0, 1)\)).
It is key to note that placing additional restrictions on $\Theta$ to reach our goal is not superfluous: indeed, we have $\mathfrak{Ber} \subset \Theta$ so that no method~\(\CS(\, \cdot \,)\) of constructing confidence sets can be  NAVAE uniformly over~$\Theta$.

\medskip

We now introduce some additional notation to develop these ideas.
Let $\sigma(\theta) := \sqrt{\Varsoustheta(\xi)}$,
${\lambda_3(\theta) := \Expsoustheta\big[(\xi-\Expsoustheta[\xi])^3/\sigma(\theta)^3\big]}$,
and $K_4(\theta) := \Expsoustheta\big[ (\xi-\Expsoustheta[\xi])^4/\sigma(\theta)^4\big]$ for any parameter~\(\theta \in \Theta\).
Finally, $S_n$ denotes the standardized sum ${\sum_{i=1}^n (\xi_i - \Expsoustheta[\xi]) \, / \, (\sigma(\theta) \sqrt{n})}$. 
Berry-Esseen (BE) inequalities control the uniform distance between the cumulative distribution function (c.d.f) of the empirical mean (properly centered and standardized) and the c.d.f~\(\Phi\) of its limit \(\Normale(0,1)\) distribution (below \(\varphi\) denotes the p.d.f of a standard Normal distribution), 
\begin{equation}
\label{eq:definition_DeltanB}
    \DeltanB(\theta) :=
    \sup_{x \in \Rb} \Big| \Probsousthetan(S_n \leq x) - \Phi(x) \Big|.
\end{equation}
Edgeworth expansions (EE) are refinements adjusting for the presence of non-asymptotic skewness.
They control the following uniform distance
\begin{equation}
\label{eq:definition_DeltanE}
    \DeltanE(\theta) := \sup_{x \in \Rb}
    \bigg| \Probsousthetan(S_n \leq x) - \Phi(x)
    - \frac{\lambda_{3}(\theta)}{6\sqrt{n}}(1-x^2)\varphi(x) \bigg|.
\end{equation}

\medskip

What kind of additional constraints can we place on the parameter set to derive an explicit bound $\delta_n$ on $\DeltanB$ or $\DeltanE$? 
Existing results require a control on some moments of the distribution, beyond moments of order two\footnote{
Recall that restricting to subsets $\widetilde{\Theta}$ of $\Theta$ that impose more than $\E_\theta[\xi^2]<+\infty$ is necessary to build NAVAE CSs that are uniform over $\widetilde{\Theta}$.
If one is willing to focus on non-parametric subsets $\widetilde{\Theta}$ of $\Theta$, the following is shown in \cite{romano2004}: 
$\CIalphannumero{CLT}$ is uniformly asymptotically exact over $\widetilde{\Theta}$ if and only if 
$\lim_{\lambda\to\infty}\sup_{\theta\in\widetilde{\Theta}}\Expsoustheta\!\left[\frac{(\xi - \Expsoustheta[\xi])^2}{\sigma(\theta)^2}\Indicator\!\left\{\frac{|\xi - \Expsoustheta[\xi]|}{\sigma(\theta)} > \lambda\right\}\right] = 0$ \((\ast)\).
Given  Proposition~\ref{prop:NAVAE_from_asymptoticUnif}, it is thus possible to derive NAVAE CSs on such parameter sets. 
This explains why NAVAE CSs can be constructed on $\ThetaBEE$,
since the previous condition~\((\ast)\) is satisfied for \(\widetilde{\Theta} = \ThetaBEE\).
}: the standardized $(2+\mu)$-th absolute moment for BE, the fourth one for EE.
For the sake of brevity, we only consider the smaller class of distributions \(\ThetaBEE:= \{ \theta \in \Theta : K_4(\theta) \leq K \}\) for some known positive constant~$K$, which allows to leverage the strength of BE and EE bounds at the same time.
Note that EE-based inequalities are more complex but yield (asymptotically) tighter bounds than BE; see Remark~\ref{rem:valid_delta_n} below and~\cite{derumigny2022explicitBE} for a detailed comparison between these inequalities.

\medskip

Equipped with a bound $\delta_n$ on $\DeltanB$ or $\DeltanE$, we can now explain how to construct NAVAE confidence intervals for \(\theta_1 = \Expsoustheta[\xi]\). 
It is enlightening to start with the simplified case of a distribution with known variance to give the main intuitions behind our construction.

\subsection{Known variance} 

To simplify exposition, let us assume that $\delta_n$ is a bound on $\DeltanB$ for a moment (we show below that the same result holds if \(\delta_n\) is a bound on \(\DeltanE\) instead).
Using the BE inequality and simple computations, we obtain the following: for any positive real number~\(x\),
\begin{equation}
\label{eq:bound_Phi_BE}
    \Probsousthetan\!\left( \frac{\sqrt{n}\left| \overline{\xi}_n - \theta_1 \right|}{\sigma(\theta)} > x \right) 
    \leq 
    2 \big( \Phi(-x) + \DeltanB(\theta)  \big)
    \leq 
    2 \big( \Phi(-x) + \delta_n  \big).    
\end{equation}

Then, for any given \(\alpha \in (0,1)\), setting \(x\) such that the right-hand side of Equation~\eqref{eq:bound_Phi_BE} is equal to~\(\alpha\) and considering the complementary event give
\begin{equation}
    \Probsousthetan\!\left( \frac{\sqrt{n}\left| \overline{\xi}_n - \theta_1 \right|}{\sigma(\theta)} 
    \leq 
    \QuantileGaussAt{1 - \frac{\alpha}{2} + \delta_n}
    \right) 
    \geq  1 - \alpha,
    \label{eq:bound_normal_quantile}
\end{equation}
whenever this ``modified Gaussian quantile'' is well-defined,
namely when $1 - \alpha/2 + \delta_n < 1$. 
When that condition is not met, we can still claim that $\theta_1 \in \Rb$ with probability one (therefore at least $1-\alpha$) by definition of $\Theta$.
Consequently, in the case of a known variance~\(\sigmaknown^2\), the confidence interval
\begin{equation*}
    \CIsigmaknown
    := \begin{cases} \left[
    \overline{\xi}_n \pm \dfrac{\sigmaknown}{\sqrt{n}}
    \QuantileGaussAt{1 - \dfrac{\alpha}{2} + \delta_n}
    \right] &
    \text{if } \delta_n < \dfrac{\alpha}{2}, \\
    \Rb & \text{else},
    \end{cases}
\end{equation*}
is non-asymptotically valid over
$\ThetaBEEsigmaknown := \{ \theta \in \ThetaBEE :
\Varsoustheta(\xi) = \sigmaknown^2 \}$.

\medskip

The first inequality in Equation~\eqref{eq:bound_Phi_BE} also holds when replacing $\DeltanB$ by $\DeltanE$ since the term $\frac{\lambda_{3}(\theta)}{6\sqrt{n}}(1-x^2)\varphi(x)$ is symmetric and vanishes.
Therefore, either the EE or the BE construction yields a non-asymptotically valid CI: a known bound \(\delta_n\) on the minimum of \(\DeltanB\) and \(\DeltanE\) is enough in the definition of~\(\CIsigmaknown\).
Proposition~\ref{prop:mean_CI_BE-Edg_sigmaknown} summarizes this discussion (see Appendix~\ref{proof:prop:mean_CI_BE-Edg_sigmaknown} for its proof).

\begin{prop}
    \label{prop:mean_CI_BE-Edg_sigmaknown}
    Let \(n \geq 1\) and
    $\delta_n \geq \sup_{\theta \in \ThetaBEE}
    \min\{ \DeltanE(\theta) , \DeltanB(\theta) \}$. Then,
    \newline $\CIsigmaknown$
    is non-asymptotically valid over $\ThetaBEEsigmaknown$
    at level~\(1 - \alpha\).
\end{prop}

In Remark~\ref{rem:valid_delta_n}, we show 
that there exist bounds
\(\delta_n\) decreasing to~\(0\) when~\(n\) goes to infinity for our class $\ThetaBEE$.
Hence, for any~\(\alpha \in (0,1)\), for \(n\) large enough, there is enough information in the sample for our confidence interval 
\(\CIsigmaknown\) to be informative in the sense of being strictly included in the whole real line~\(\Rb\). 
Since $\delta_n$ is deterministic and converges to zero, we remark that, for any $\theta$ in $\Thetasigmaknown := \{ \theta \in \Theta :
\Varsoustheta(\xi) = \sigmaknown^2 \}$,
\begin{equation}
\label{eq:asymp_similarity_with_clt_ci}
    \dfrac{\sigmaknown}{\sqrt{n}}
    \QuantileGaussAt{1 - \dfrac{\alpha}{2} + \delta_n} = \dfrac{\sigmaknown}{\sqrt{n}}
    \QuantileGaussAt{1 - \dfrac{\alpha}{2}} + o(1/\sqrt{n}).
\end{equation}
Therefore, \(\CIsigmaknown\) behaves like a confidence interval based on the asymptotic normality of the sample mean and is thus pointwise asymptotically exact over $\Thetasigmaknown$.
In fact, our CI can be shown to be uniformly asymptotically exact over $\ThetaBEEsigmaknown$, as stated in Proposition~\ref{prop:mean_CI_BE-Edg_asymp_sigmaknown} and proved in Appendix~\ref{proof:prop:mean_CI_BE-Edg_asymp_sigmaknown}.
\begin{prop}
    \label{prop:mean_CI_BE-Edg_asymp_sigmaknown}
    
    Let $(\delta_n)$ be a sequence
    such that
    ${\delta_n \geq \sup_{\theta \in \ThetaBEE}
    \min\{ \DeltanE(\theta) , \DeltanB(\theta) \}}$
    and 
    $\delta_n \to 0$.
    Then
    \begin{itemize}
    \item[(\textit{i})]
    $\CIsigmaknown$ is asymptotically exact pointwise over $\Thetasigmaknown$ at level~\(1 - \alpha\);    
    \item[(\textit{ii})]
    $\CIsigmaknown$ is asymptotically exact uniformly over~\(\ThetaBEEsigmaknown\) at level~\(1 - \alpha\).
    \end{itemize}
\end{prop}

Propositions~\ref{prop:mean_CI_BE-Edg_sigmaknown} and~\ref{prop:mean_CI_BE-Edg_asymp_sigmaknown}(ii) ensure $\CIsigmaknown$ is NAVAE uniformly over $\ThetaBEEsigmaknown$.
Combined with Equation~\eqref{eq:asymp_similarity_with_clt_ci},
this result implies that our confidence interval is of the form~\eqref{eq:def:efficient_CI} and therefore efficient.

\begin{rem}[Possible choices for $\delta_n$]
\label{rem:valid_delta_n}
    Proposition~\ref{prop:mean_CI_BE-Edg_sigmaknown} and subsequent results
    require a known bound~\(\delta_n\), on either \(\DeltanB(\theta)\) or \(\DeltanE(\theta)\).
    Furthermore, the smaller the bound \(\delta_n\), the shorter the resulting CI.
    Relying on \cite{shevtsova2013} and the inequality \(\E[|\xi|^3] \leq (\E[\xi^4])^{3/4}\),
    we deduce that \(\delta_{1,n} := 0.4690 K^{3/4} / \sqrt{n}\) is a valid bound on \(\DeltanB(\theta)\) over \(\ThetaBEE\).
    \cite{derumigny2022explicitBE} provides 
    \(\delta_{2,n} := 0.1995 (K^{3/4} + 1) / \sqrt{n} + O(n^{-1}) \) as a bound on \(\DeltanE(\theta)\) over \(\ThetaBEE\), with an explicit remainder term implemented in the \texttt{R} package \texttt{BoundEdgeworth} (\cite{BoundEdgeworth}).
    
    Consequently, Proposition~\ref{prop:mean_CI_BE-Edg_sigmaknown} holds with \(\delta_n := \min(\delta_{1,n}, \delta_{2,n})\).
    Although \(\delta_{1,n}\) and \(\delta_{2,n}\) decrease to~\(0\) as \(n\) goes to infinity at the same rate, \(\delta_{2,n}\) is asymptotically smaller than \(\delta_{1,n}\).
    For \(n\) large enough, our confidence intervals thus rely on an Edgeworth expansion.
    
    Under additional restrictions (that rule out discrete distributions),
    \cite{derumigny2022explicitBE} show that the quantity
    \(\delta_{3,n} := (0.195 K + 0.01465 K^{3/2}) / n + O(n^{-5/4})\)
    upper bounds~\(\DeltanE(\theta)\), allowing to build CIs that are one order of magnitude shorter than Berry-Esseen-based ones (using the previously mentioned package).
    See Remark~\ref{rem:no_plug_in_bounds} for discussion about the choice of $K$ in practice.
\end{rem}

\subsection{Unknown variance}
\label{ssec:unknown_variance_mean}

When the variance is unknown, a construction in the spirit of $\CIsigmaknown$ remains possible at the cost of controlling the variance estimation error. 
The process we use is twofold. 
First, the ratio between the ``oracle'' variance estimator \(\sigmazerocarrenhat := n^{-1} \sum_{i=1}^n (\xi_i - \theta_1)^2\) and its limit $\sigma^2(\theta)$ is tightly controlled from below using Theorem~2.19 in~\cite{pena2008self}.
That theorem allows us to write, for every $a > 1$ and $\theta \in \ThetaBEE$,
\begin{equation*}
    \Probsousthetan\!\left( \frac{\widehat{\sigma}_0^2}{\sigmadeux(\theta)} < \frac{1}{a} \right)
    \leq 
    \exp\!\left( -\frac{n(1-1/a)^2}
    {2K} \right)
    =: \, \nuVar(a).
\end{equation*}
From now on, we introduce a sequence \((a_n)\) of tuning parameters. 
Indeed, to build CIs with good asymptotic properties, we need $a_n$ to depend on $n$ as specified in Proposition~\ref{prop:mean_CI_BE-Edg_asymp_sigmaunknown}.
Second, we control \(\sqrt{n}(\overline{\xi}_n - \theta_1)/\sigma(\theta)\) similarly to what was done when building $\CIsigmaknown$.

\medskip

Combining the two steps, we obtain the following confidence interval \(\CIsigmanhat:=\)
\begin{equation*}
    \begin{cases}
        \left[ \overline{\xi}_n \pm
        \dfrac{\sigmanhat}{\sqrt{n}} \Cn
        \QuantileGauss\!\left({1 - \dfrac{\alpha}{2} + \delta_n +  \dfrac{\nuVar}{2}}\right)
        \right] &
        \text{if } 1 - \dfrac{\alpha}{2} + \delta_n + \dfrac{\nuVar}{2} < \Phi(\sqrt{n/a_n}), \\
        \Rb & \text{else},
    \end{cases} 
\end{equation*}
where 
\begin{equation*}
    \Cn := 
    \left(\frac{1}{a_n} - \frac{1}{n} \, \QuantileGauss\!\left( 1 - \frac{\alpha}{2}
    + \delta_n + \frac{\nuVar}{2} \right)^{\!2\,}\right)^{-1/2}.
\end{equation*}

Besides the data \((\xi_i)_{i=1}^n\) and the level~\(1 - \alpha\), \(\CIsigmanhat\) depends on three quantities: \(\delta_n\), \(K\), and the tuning parameter~\(a_n\).
\(\CIsigmanhat\) can be shown to be non-asymptotically valid over \(\ThetaBEE\).
Proposition~\ref{prop:mean_CI_BE-Edg_sigmaunknown} formalizes our finite-sample results on \(\CIsigmanhat\). 
Its proof can be found in Appendix~\ref{proof:prop:mean_CI_BE-Edg_sigmaunknown}.

\begin{prop}
    \label{prop:mean_CI_BE-Edg_sigmaunknown}
    Let \(n \geq 1\), $a_n \in (1,+\infty)$, and
    $\delta_n \geq \sup_{\theta \in \ThetaBEE}
    \min\{ \DeltanE(\theta) , \DeltanB(\theta) \}$.
    Then, $\CIsigmanhat$ is non-asymptotically valid over~\( \ThetaBEE\) at level~\(1 - \alpha\).
\end{prop}

In line with $\CIsigmaknown$, $\CIsigmanhat$ can be shown to be asymptotically close to $\CIalphannumero{CLT}$. 
In fact, if $a_n$ is such that $a_n \to 1$ and $n(1-1/a_n)^2 \to +\infty$, we can see that for any positive value of~\(K\) and for any~\(\theta \in \Theta\), \(\nuVar \to 0\) as \(n \to +\infty\).
This implies
\begin{equation}
    \frac{\widehat{\sigma}}{\sqrt{n}} C_{n} 
    \QuantileGaussAtBig{1 - \frac{\alpha}{2}
    + \delta_n + \frac{\nuVar}{2} }
    = \frac{\sigma}{\sqrt{n}}
    \QuantileGaussAtBig{1 - \frac{\alpha}{2}}
    + o_{P_\theta} \big(1/\sqrt{n} \big),
    \label{eq:Taylor_asympt_point_exact_forTheMean}
\end{equation}
where, in this representation, the value~\(K\) is hidden in the \(o_{P_\theta} \big(1/\sqrt{n} \big)\) term.
Consequently, as is the case for \(\CIalphannumero{CLT}\), \(\CIsigmanhat\) is asymptotically exact pointwise over the whole parameter space~\(\Theta\). 
As stated in the next proposition, $\CIsigmanhat$ is also asymptotically exact uniformly over $\ThetaBEE$. 
The proof can be found in Section~\ref{proof:prop:mean_CI_BE-Edg_asymp_sigmaunknown}.

\begin{prop}
\label{prop:mean_CI_BE-Edg_asymp_sigmaunknown}
    
    Let $(\delta_n)$ be a sequence
    such that 
    ${\delta_n \geq \sup_{\theta \in \ThetaBEE}
    \min\{ \DeltanE(\theta) , \DeltanB(\theta) \}}$
    and 
    $\delta_n \to 0$.
    Assume that $b_n := a_n - 1$ is such that $b_n \to 0$ and \(b_n \sqrt{n} \to +\infty\).    
    Then
    \begin{itemize}
    
    \item[(\textit{i})]
    $\CIsigmanhat$ is asymptotically exact pointwise over $\Theta$ at level~\(1 - \alpha\);
    
    \item[(\textit{ii})]
    $\CIsigmanhat$ is asymptotically exact uniformly over~\(\ThetaBEE\) at level~\(1 - \alpha\).

    \end{itemize}
    
\end{prop}

Propositions~\ref{prop:mean_CI_BE-Edg_sigmaunknown},~\ref{prop:mean_CI_BE-Edg_asymp_sigmaunknown}(ii), and Equation~\eqref{eq:Taylor_asympt_point_exact_forTheMean} ensure $\CIsigmanhat$ is NAVAE uniformly over $\ThetaBEE$ and efficient since it is of the form~\eqref{eq:def:efficient_CI}.

\medskip

Similar to \(\CIsigmaknown\), our interval might be non-informative, namely equal to~\(\Rb\).
The tuning parameter \(a_n\) must satisfy the constraint
\begin{equation}
    1 - \dfrac{\alpha}{2} + \delta_n + \dfrac{\nuVar(a_n)}{2} < \Phi(\sqrt{n/a_n}).
\label{eq:constraint_a_n}
\end{equation}
for the interval to be informative.
%
%
For $n$ small, there might not be any value $a_n \in (
1, +\infty)$ for which the constraint holds. 
However, it will be the case for \(n\) large enough as shown in Proposition~\ref{prop:open_interval_a} (proved in Appendix~\ref{proof:prop:open_interval_a}).
Note that as soon as there is one solution, we know that the set of possible $a_n$ forms an open interval.

\begin{prop}
    Let $\alpha \in (0, 1/2)$, $K > 0$.
    \begin{enumerate}
        \item[(i)] Let $n, \delta_n > 0$.
        The subset of values of $a_n$ of $(1, +\infty)$
        satisfying the constraint \eqref{eq:constraint_a_n}
        is an open interval $I_n$ (potentially empty).
        If $I_n$ is not empty, then it must be of the form $(a_{1,n}, a_{2,n})$, with $1 < a_{1,n}$ and $a_{2,n} < +\infty$.

        If moreover $(\delta_n)$ is a decreasing sequence, then the sequence $I_n$ is increasing (with respect to the inclusion order).
        Furthermore $(a_{1,n})$ and $(a_{2,n})$ are respectively decreasing and increasing sequences, with respective limits $1$ and $+\infty$.

        \item[(ii)] Let $a_n = 1 + b_n \in (1, +\infty)$ with $b_n \to 0$ and \(b_n \sqrt{n} \to +\infty\).
        Assume that for $n$ large enough,
        $1 - \dfrac{\alpha}{2} + \delta_n$ is bounded above by a constant strictly smaller than $1$.
        Then, for $n$ large enough, the constraint \eqref{eq:constraint_a_n} is satisfied.
    \end{enumerate}
    
\label{prop:open_interval_a}
\end{prop}

\begin{rem}
\label{rem:optimized_choice_an}
    When \(I_n\) is not empty, a natural choice for \(a_n\) is to minimize the width of our confidence interval, which can be done numerically.     
    Such an automatic selection rule for \(a_n\) is appealing as it avoids having to choose manually a value for this parameter and yields the shortest CI.
    Furthermore, for a fixed bound \(K\) (as opposed to plug-in, see Section~\ref{ssec:plug_in_discussion}), that optimal \(a_n\) is not data-dependent.
    Table~\ref{tab:coverage_expectation_case} in Section~\ref{subsec:simu_expectation} compares our optimized choice to the ad hoc alternative \(a_n = 1 + n^{-1/5}\), which satisfies the rate requirement of Proposition~\ref{prop:mean_CI_BE-Edg_asymp_sigmaunknown}.
    We find that these competing choices lead to comparable inference results in simulations.    
\end{rem}

In the simple case of an expectation, we have thus managed to construct closed-form and non-randomized NAVAE confidence intervals uniformly over a non-parametric class of distributions delineated by moment conditions only.
As highlighted in the explanatory case of a known variance, the key behind that construction is to enlarge \(\CIalphannumero{CLT}\) properly through asymptotically negligible modifications based on Berry-Esseen inequalities or Edgeworth expansions.
In Section~\ref{sec:lin_mod_exo}, we move on to linear regressions.
Following the same principle, 
we derive confidence intervals that enjoy the same appealing theoretical properties.


\section{NAVAE confidence intervals for linear regressions' coefficients}
\label{sec:lin_mod_exo}

\subsection{Model formulation}

For the sake of completeness and precision, we first state the statistical model and its assumptions.

\begin{hyp}[Linear model]
\label{hyp:basic_linear_model}
We observe $n$ independent replications
$(X_1,Y_1), \dots, $ $(X_n,Y_n)$
following the distribution $\Distribution_{X,Y},$
where $X$ is an explanatory random vector of dimension $p$
and $Y$ is an outcome real random variable 
such that, for some real random variable~\(\eps\) and vector~\(\betazero\): 
\begin{align}
\label{eq:basic_linear_model}    
    &Y = X' \betazero + \eps
    \quad \textnormal{and} \quad
    \Distributionde{X,\eps} \in \setDistribXeps,
    \textnormal{ where } \\
    &\setDistribXeps := \Big\{ 
    \Distributionde{X,\eps} \in \mathcal{P}(\Rb^{p+1}) :
    \E[X\eps] = 0,
    \,
    \E\big[\normesmall{X}^4\big] < +\infty,
    \,
    \minvp(\E[XX']) > 0,
    \, \nonumber \\
    &\hspace{4.6cm}
    \minvp\!\left(\E\!\left[XX'\varepsilon^2\right]\right) > 0,
    \,
    \maxvp\!\left(\E\!\left[XX'\varepsilon^2\right]\right) < +\infty
    \Big\}.
    \nonumber
\end{align}
\end{hyp}

\noindent The parameter set of the associated statistical model is \(\Theta := \left\{ \theta = (\betazero, \Distribution_{X,\eps}) \in \Rb^p \times \setDistribXeps \right\} \).
In this model, $\Distributionde{\theta}$ denotes a distribution of $(X, Y)$ indexed by $\theta \in \Theta$.
For brevity, we implicitly omit the index $\theta$ in the expectation operator~$\E$.
In what follows, we consider several subsets of $\Theta$ characterized by additional restrictions that enable us to build several confidence sets.
We will investigate and compare their properties in line with the discussion conducted in Section~\ref{sec:expectation_example}.

\medskip 

Assumption~\ref{hyp:basic_linear_model} sets a basic linear regression model.
\(\Theta\) is indeed the largest parameter set compatible with usual economic assumptions and minimal statistical conditions.
The condition \(\E[X\eps] = 0\) corresponds to the (weak) exogeneity of covariates, i.e., the orthogonality condition of the linear projection of~\(Y\) on~\(X\).
It is implied by the (strong) exogeneity assumption \(\E[\eps | X] = 0\), but does not require the conditional expectation of \(Y\) given \(X\) to be linear.
The other moment conditions allow for heteroskedasticity while ensuring the asymptotic normality of the Ordinary Least Squares (OLS) estimator of $\betazero$:
\begin{equation*}
    \betahat := 
    \Bigg( \frac{1}{n} \sumiunn X_i X_i' \Bigg)^{\!\!\dagger}
    \Bigg( \frac{1}{n} \sumiunn X_i Y_i \Bigg).
\end{equation*}
More precisely, under Assumption~\ref{hyp:basic_linear_model}, a finite second-order moment, \(\E\big[ \normesmall{X}^2 \big] < +\infty\), is sufficient for that result.
However, a finite fourth-order moment is necessary to consistently estimate the asymptotic variance of~\(\betahat\) by its empirical counterpart and perform inference on~\(\betazero\) in practice.
That is why we incorporate this additional requirement directly into our basic statistical model.
Besides, remark that the condition \( \minvp(\E(XX')) > 0 \) is equivalent to the invertibility of the matrix \( \E[XX'] \).
Thus, under Assumption~\ref{hyp:basic_linear_model}, \((n^{-1} \sumiunn X_i X_i')^{\dagger} = (n^{-1} \sumiunn X_i X_i')^{-1}\) with probability approaching one as the sample size~\(n\) goes to infinity.

\medskip 

For a given known vector~$\uu$ of $\Rb^p$, our goal is to build a confidence interval for a linear functional of the form $\uu' \betazero$.
It encompasses CIs for each individual component of $\betazero$ (taking for $\uu$ the canonical vectors) and also differences of coefficients that appear when investigating the relative impact of two covariates.
We consider henceforth an arbitrary vector $\uu \in \Rb^p \setminus \{0_{\Rb^p}\}$.

\subsection{Properties of the usual confidence interval}

As mentioned in the introduction, the standard way to proceed is to construct a CI centered at the estimator $\uu' \betahat$ relying on the asymptotic normality of the OLS estimator:
\begin{equation}
    \CIAsympu := \left[ \uu' \betahat
    \pm \frac{\QuantileGaussAt{1 - \alpha/2}}{\sqrt{n}}
    \sqrt{\uu' \Vhat \uu} \right],
    \label{def:CIAsympu}
\end{equation}
where 
\begin{equation*}
    \Vhat := \left(\frac{1}{n}\sum_{i=1}^nX_iX_i'\right)^{\!\!\dagger}
    \left(\frac{1}{n}\sum_{i=1}^nX_iX_i'\,\widehat{\eps}_i^{\,2}\right)
    \left(\frac{1}{n}\sum_{i=1}^nX_iX_i'\right)^{\!\!\dagger}
\end{equation*}
is the standard estimator of the asymptotic variance
\(\AsymptoticVarbetahat := \E[XX']^{-1} \E[XX'\eps^2] \E[XX']^{-1}\)
of $\betahat$, and
\(\widehat{\eps}_i := Y_i-X_i'\widehat{\beta}\) is the residual for the $i$-th observation.


\medskip

The pros and cons of $\CIAsympu$ are well-understood and very close to those of $\CIalphannumero{CLT}$ that were detailed in Section~\ref{sec:expectation_example}.
By applications of the Law of Large Numbers, the CLT, and Slutsky's lemma, $\CIAsympu$ is known to be asymptotically exact pointwise over~$\Theta$.
Besides, following \cite{kasy2019}, $\CIAsympu$ can be strengthened to be asymptotically exact uniformly at level~\(1 - \alpha\) over
\begin{equation*}
    \ThetaKasy := \left\{\theta\in\Theta:
    \minvp\big( \E[XX'] \big) \geq m, \,
    \E\big[ \normesmall{X}^4 \big] \leq M, \,
    m \leq \E \big[ \normesmall{X\eps}^4 \big] \leq M \right\}, 
\end{equation*}
with $0 < m \leq M < +\infty$.

\medskip

Furthermore, $\CIAsympu$ becomes non-asymptotically exact at level~\(1 - \alpha\) over 
\begin{equation*}
    \ThetaGauss := \big\{\theta \in \Theta : \Distribution_{X,\eps} = \Distribution_X \otimes \Distribution_\eps, \Distribution_\eps = \Normale(0,\sigmadeux), \sigmadeux \in \Rstarplus \},
\end{equation*}
provided $\Vhat$ is replaced with the (homoscedastic) estimator \(
(n-p)^{-1}
\sum_{i=1}^n \widehat{\eps}_i^{\,2} \left(
n^{-1}
\sum_{i=1}^n X_i X_i'\right)^{\dagger}\)
and the quantile $\QuantileGaussAtsmall{1 - \alpha/2}$
with the 
quantile of a Student distribution with ${n - p}$ degrees of freedom.

\medskip

However, assuming that the true distribution belongs to $\ThetaGauss$
-- so as to achieve such finite-sample properties --
is often considered too restrictive in practice:
Gaussian error terms impede skewed or heavier-tail shocks,
and independence between $\eps$ and $X$ rules out heteroskedasticity.
In what follows, we propose NAVAE CIs without relying on such independence or parametric assumptions.
This corresponds to the same objectives as in Section~\ref{sec:expectation_example} now for the case of linear regressions; they will be met using the same tools (Berry-Esseen-type inequalities or bounds on Edgeworth expansions) in this different setup.

\subsection{Presentation of our confidence interval}
\label{ssec:fin_ols_ci_construction}

\subsubsection{Intuitions and assumptions}
\label{subsubsec:intuition_assumptions_for_our_OLS}

Remark that
\begin{equation*}
    \uu' \betahat
    = \uu' \betazero
    + \uu' \bigg(\frac{1}{n} \sumiunn X_i X_i' \bigg)^{\!\!\dagger} \, \frac{1}{n}  \sumiunn X_i \eps_i,
\end{equation*}
but the second term in the previous sum cannot be written as an empirical mean of \iid{} variables, rendering it harder to analyze theoretically.
Under Assumption~\ref{hyp:basic_linear_model}, the quantity \((n^{-1} \sumiunn X_i X_i')^{\dagger}\) converges in probability to \(\E[XX']^{-1}\).
Therefore, we obtain the following linearization
\begin{equation}
\label{eq:linearization_estimator}
    \uu' \widehat{\beta} - \uu' \betazero
    \approx \frac{1}{n} \sumiunn
    \underbrace{\uu' \, \E[XX']^{-1} X_i \eps_i}_{\displaystyle =: \xi_i},
\end{equation}
and the error of the estimator can thus be approximated by an average of $n$ \iid{} random variables distributed as
$\xi := \uu' \, \E[XX']^{-1} X \eps$.

\medskip 

Like in Section~\ref{sec:expectation_example}, our CIs cannot be non-asymptotically valid on $\Theta$ itself because this space is simply too large.
Therefore, we now introduce the subset $\ThetaBEE$ of $\Theta$ described in the following Assumption~\ref{hyp:subset_exo_non-asymptotic_validity} in order to obtain non-asymptotic guarantees.  
To state our conditions, we introduce the rotated version of $X$, defined as $\widetilde{X}:= \E(XX')^{-1/2} X$.

\begin{hyp}[Bounds on DGP]
\label{hyp:subset_exo_non-asymptotic_validity}
    Let $\minlambdaEXXt$, $\majorantEXXprimeconcentration$, \(\majorantEnormeXepsquatre\), and \(\majorantKurtxi\) be some positive constants.
    We define $\ThetaBEE$ to be the set of parameters $\theta = (\betazero, \Distributionde{X, \eps}) \in \Theta$ 
    such that the joint distribution $\Distributionde{X, \eps}$ satisfies:
    \begin{enumerate}
        \item[(i)] \( \minvp(\E[XX']) \geq \minlambdaEXXt \);
        \item[(ii)] \(\E\!\left[ \normebig{
        \operatorvec \big(\widetilde{X}\widetilde{X}' - \mathbb{I}_p \big) }^2
        \right] \leq \majorantEXXprimeconcentration\);
        \item[(iii)] \(\E\!\left[ \normebig{\widetilde{X}\eps}^4 \right]
        \leq \majorantEnormeXepsquatre\);
        \item[(iv)] \( \E\!\left[\xi^4\right] / \, \E\!\left[\xi^2\right]^2
        \leq \majorantKurtxi\).
    \end{enumerate}
\end{hyp}

Assumption~\ref{hyp:subset_exo_non-asymptotic_validity} defines a broad non-parametric class of distributions delineated by the different constants \(\minlambdaEXXt\), \(\majorantEXXprimeconcentration\), \(\majorantEnormeXepsquatre\), and \(\majorantKurtxi\). 
These constants appear explicitly in the construction of our CIs, as did the constant \(K\) in the previous section for the case of an expectation.
The user needs to specify their values in practice, which should be done with care.
We elaborate on these choices in Section~\ref{sec:practical_considerations_plugin}.
Relying on explicit constants may seem restrictive compared to standard asymptotic inference.
However, outside of some specific parametric models, using such types of bounds or other known quantities (like the median-bias of the estimator in \cite{kuchibhotla2024hulc}) is unavoidable to obtain non-asymptotic properties.

\medskip

Overall, the different parts of Assumption~\ref{hyp:subset_exo_non-asymptotic_validity} strengthen the moment conditions of the basic linear model~$\Theta$.
Part~(i) rules out \(\E[XX']\) matrices arbitrarily close to being singular, an unfavorable situation in which $\betazero$ is not identified.
Part~(ii) helps control the concentration of \( n^{-1} \sumiunn X_i X_i' \) and ensures, together with part~(i), that it is invertible with large probability for every (large enough) sample size.
Part~(iii) complements (ii) to control the linearization \eqref{eq:linearization_estimator} of the OLS estimator.
Part~(iii) is implied by \ref{hyp:subset_exo_non-asymptotic_validity}(i) and the simpler (but stricter) constraint
\(\E\!\left[ \normebig{X \eps}^4 \right]
\leq C\) for some constant $C > 0$.
Part~(iv) allows to bound the Edgeworth expansion of the distribution of $n^{-1} \sum_{i = 1}^n \xi_i$ and enables to derive a tight control on the distance between $n^{-1} \sum_{i = 1}^n \xi_i^2$ and its expectation, which happens to be the asymptotic variance associated with $\sqrt{n}u'(\betahat-\betazero)$.

\medskip

Parts~(i) to~(iii) of the previous assumption are critical in the construction of our CI as they enable us to ensure that with large probability 
\begin{equation*}
    \frac{\sqrt{n}\uu' (\betahat - \betazero)}{\sqrt{\uu' \Vhat \uu}} \approx \frac{n^{-1/2}\sumiunn\xi_i}{\sqrt{n^{-1}\sumiunn \xi_i^2}},
\end{equation*}
i.e., the centered and scaled OLS estimator gets ``close'' to a self-normalized sum of centered \textit{i.i.d.} random variables on a large-probability event. 
More precisely, they first guarantee that the linearization
\begin{equation*}
    \sqrt{n} \uu' ( \betahat - \betazero ) \in
    \left[ \frac{1}{\sqrt{n}} \sumiunn \xi_i \pm \RnLin(\gamma) \right]
\end{equation*}
holds with probability at least \(1-2\gamma\), for any \(\gamma > 0\), where
\begin{align*}
    \RnLin(\gamma)
    &:= \sqrt{2} \|u\| \minlambdaEXXt^{-1/2}
    \frac{\widetilde{\gamma}}{1 - \widetilde{\gamma}}
    \left(\frac{\majorantEnormeXepsquatre}{\gamma}\right)^{1/4}
\end{align*}
is an explicit bound on the linearization error term
and \(\widetilde{\gamma}
:= \sqrt{\majorantEXXprimeconcentration/(n\gamma)}\).
Second, with probability at least \(1-2\gamma\) too, they enable us to prove that
\begin{equation*}
    \left| \uu' \Vhat \uu - n^{-1}\sumiunn \xi_i^2 \right| \leq \|\uu\|^2\RnVar(\gamma),
\end{equation*}
where
\begin{align*}
    \RnVar(\gamma)
    := & \;\;
     \frac{2}{n\minlambdaEXXt^3} \left( \frac{\widetilde{\gamma}}{1-\widetilde{\gamma}} + 1 \right)  ^2 \sqrt{\frac{\majorantEnormeXepsquatre}{\gamma}}
    \times 
    \frac{1}{n}\sum_{i=1}^n\|X_i\|^4 \\
    & + 
    \frac{2\sqrt{2}}{\minlambdaEXXt^{5/2} \sqrt{n}} \left( \frac{\widetilde{\gamma}}{1-\widetilde{\gamma}} + 1 \right)
    \left( \frac{\majorantEnormeXepsquatre}{\gamma} \right)^{1/4}
    \times 
    \frac{1}{n}\sum_{i=1}^n\|X_i\|^3\,|\epshat_i| \\
    &+ \frac{\majorantEXXprimeconcentration/(n\gamma)}{\minlambdaEXXt^{2}(1-\widetilde{\gamma})^2} \times 
    \frac{1}{n}\sum_{i=1}^n\|X_i\epshat_i\|^2 \\
    &+ \frac{2\widetilde{\gamma}}{\minlambdaEXXt (1-\widetilde{\gamma})} \times \norme{\frac{1}{n}\sum_{i=1}^nX_iX_i'S^{\dagger}\epshat_i^2},
\end{align*}
and \(S\) is a shorthand notation for $n^{-1} \sum_{i=1}^n X_i X_i'$.

\medskip

Remark that \(n^{-1/2}\sumiunn\xi_i\,\big/ \sqrt{n^{-1}\sumiunn \xi_i^2}\) is the ratio of a sample mean (centered at the true expectation, here equal to 0) and the square-root of its corresponding oracle variance (see Section~\ref{ssec:unknown_variance_mean} for a definition). This quantity can thus be managed using the method introduced in Section~\ref{ssec:unknown_variance_mean}, which relies on Berry-Esseen or Edgeworth-type controls and an exponential deviation inequality between $n^{-1}\sumiunn \xi_i^2$ and its limit.
This is where Assumption~\ref{hyp:subset_exo_non-asymptotic_validity}.(iv) comes into play,
with \(\majorantKurtxi\) playing the same role as~\(K\) in the case of an expectation.

\subsubsection{Formal definition}
\label{subsubsec:formal_definition_our_OLS}

For any tuning parameters $\omega_n \in (0, 1)$ and $a_n \in (1, +\infty)$,
we define the ``modified Gaussian quantile''
\begin{equation*}
    \QnEdg := \sqrt{a_n} \, q_{\mathcal{N}(0,1)}
    \big(1 - \alpha/2 + \nuEdg \big) +  \nuApprox ,
\end{equation*}
which depend on some perturbation terms $\nuEdg$ and $\nuApprox$ defined as
\begin{align*}
    \nuEdg &:= \frac{\omega_n \alpha
    + \exp\!\big(-n(1-1/a_n)^2/(2 \majorantKurtxi)\big)}{2}
    + \delta_n, \\    
    \nuApprox &:= \frac{\RnLin(\omega_n\alpha/2)}{\sqrt{\uu' \Vhat \uu
    + \|u\|^2 \RnVar(\omega_n\alpha/2)}}.
\end{align*}
where $\delta_n \geq
\sup_{\theta \in \ThetaBEE}
\min \! \left\{
\DeltanB(\theta) \, , \, \DeltanE(\theta)
\right\}$.
This condition is similar to the one used in Section~\ref{sec:expectation_example} in the framework of an expectation.
Remember that \(\DeltanB(\theta)\) and \(\DeltanE(\theta)\)
are defined in Equations~\eqref{eq:definition_DeltanB} 
and~\eqref{eq:definition_DeltanE} with \(\xi\) here equal to \(\uu' \, \E[XX']^{-1} X \eps\).
Therefore, the choices of \(\delta_n\) introduced in Remark~\ref{rem:valid_delta_n} can be used in the current setting as well, replacing \(\majorantKurtxi\) with \(K\).

\medskip

To ensure an informative CI is feasible, we impose that $n$ be larger than \\
$n_0 := \max \{n \in \INTST: 
n \leq 2 \majorantEXXprimeconcentration / (\omega_n \alpha)
\text{ or } \nuEdg \geq \alpha / 2 \}$.
All in all, our confidence interval is centered at $\uu' \betahat$
(for $n > n_0$) and defined by
\begin{align*}
    \CIEdgu
    &:= \text{CI}_\uu^{\text{Edg}}(1-\alpha, n, \delta_n,\omega_n, a_n,
    \majorantKurtxi, \majorantEXXprimeconcentration, \majorantEnormeXepsquatre, \minlambdaEXXt)
    \\ &
    := 
    \begin{cases}
        \Rb \hspace{2cm} & \text{ if } n \leq n_0, \\
        \left[ \uu' \widehat{\beta}
        \pm \dfrac{\QnEdg}{\sqrt{n}} \sqrt{u' \Vhat u
        + \|u\|^2 \RnVar(\omega_n\alpha/2)} \right] & \text{ else.} \\
    \end{cases}
\end{align*}
Note that $\CIEdgu$ depends on two tuning parameters,
\(\omega_n\) and \(a_n\),
the quantity $\delta_n$, and the four bounds $\majorantKurtxi, \majorantEXXprimeconcentration, \majorantEnormeXepsquatre, \minlambdaEXXt$ that delineate the set~\(\ThetaBEE\) and appear in Assumption~\ref{hyp:subset_exo_non-asymptotic_validity}.
We do not indicate this dependence to lighten notations.
This interval is similar to \(\CIAsympu\) in Equation~\eqref{def:CIAsympu} with the addition of $\norme{\uu}^2\RnVar(\omega\alpha/2)$ in the variance term and the modified Gaussian quantile $\QnEdg$ which depends on $\nuEdg$ and $\nuApprox$.
The term $\nuApprox$ is a random quantity\footnote{
Strictly speaking, the denominator of \(\nuApprox\) could be null in some pathological situations; nevertheless, this does not impact our confidence interval since this term simplifies.
} since $\RnVar$ and $\Vhat$ depend on the sample.
Therefore, unlike \(\QuantileGaussAt{1-\alpha/2}\), $\QnEdg$ is a random quantity, introducing another randomness source into the width of~$\CIEdgu$.

\subsection{Properties of our confidence interval}

In this section, we state several results on the asymptotic and non-asymptotic properties of our confidence interval $\CIEdgu$.
We start by a result on the non-asymptotic validity,
%
%
proved in Section~\ref{proof:thm:exogenous_case_non-asymptotic_validity}.

\begin{thm}
\label{thm:exogenous_case_non-asymptotic_validity}
    Let \(n \geq 1\), 
    $\delta_n \geq \sup_{\theta \in \ThetaBEE}
    \min\{ \DeltanE(\theta) , \DeltanB(\theta) \}$,
    $a_n \in (1, +\infty)$,
    and \(\omega_n \in (0, 1)\). 
    Then, \(\CIEdgu\) is non-asymptotically valid over~\(\ThetaBEE\) at level \(1-\alpha\).
\end{thm}

This theorem shows it is possible to build a CI that is non-asymptotically valid over a large class of data-generating processes (here $\ThetaBEE$) without imposing independence between $X$ and $\eps$, nor parametric restrictions on~$\eps$. 
What is the behavior of our CI as $n$ goes to infinity? 
Under appropriate restrictions on $a_n$, $\omega_n$ and $\delta_n$, it is shown in Section~\ref{proof:thm:exogenous_case_asymptotic_pointwise_exactness} that, for any positive constants
\(\majorantKurtxi\), \(\majorantEXXprimeconcentration\), \(\majorantEnormeXepsquatre\), and \(\minlambdaEXXt\), and for any \(\theta \in \Theta\), 
\begin{align}
    \CIEdgu = \left[ \uu' \widehat{\beta}
    \pm \frac{\QuantileGaussAt{1-\alpha/2}+o_{P_\theta}(1)}{\sqrt{n}} \sqrt{u' \Vhat u
    + o_{P_\theta}(1)} \right],
    \label{eq:decomp_CI_fin_oP}
\end{align}
with the values
$\majorantKurtxi, \majorantEXXprimeconcentration, \majorantEnormeXepsquatre, \minlambdaEXXt$
hidden in the $o_{P_\theta}(1)$ terms in this representation.
In other words, our interval coincides at the limit \(n \to \infty\) with the standard interval obtained from the CLT and therefore is asymptotically exact pointwise over the whole parameter set~\(\Theta\).
Theorem~\ref{thm:exogenous_case_asymptotic_pointwise_exactness} (proved in Section~\ref{proof:thm:exogenous_case_asymptotic_pointwise_exactness}) formalizes the latter result.

\begin{thm}
\label{thm:exogenous_case_asymptotic_pointwise_exactness}
    Let $(\delta_n)$ be such that 
    ${\delta_n \geq \sup_{\theta \in \ThetaBEE}
    \min\{ \DeltanE(\theta) , \DeltanB(\theta) \}}$
    and 
    $\delta_n \to 0$.
    Assume that \(\omega_n \to 0\),
    \(\omega_n n^{2/3} \to +\infty\),
    and that \(b_n := a_n - 1\)
    is such that
    \(b_n \to 0\) and \(b_n \sqrt{n} \to +\infty\).
    Then, for every $(\majorantKurtxi, \majorantEXXprimeconcentration, \majorantEnormeXepsquatre, \minlambdaEXXt) \in (0, +\infty)^4$, $\CIEdgu$ is asymptotically exact pointwise over $\Theta$ at level \(1-\alpha\).
\end{thm}

To obtain uniform asymptotic exactness, we place ourselves on a subset of \(\ThetaBEE\) defined in the following Assumption~\ref{hyp:subset_exo_unif_asymptotic_exactness} while choosing $a_n$, $\omega_n$ and $\delta_n$ as in Theorem~\ref{thm:exogenous_case_asymptotic_pointwise_exactness}.
The result is stated in Proposition~\ref{prop:exogenous_case_asymptotic_uniform_exactness} below.

\begin{hyp}[Bounds on DGP (continued)]
\label{hyp:subset_exo_unif_asymptotic_exactness}
    Let \(\minlambdaEXXtepssquare > 0\), $\rho \geq 0$, and \(\majorantEnormeXquatreunplusrho > 0\) be some constants.    
    We define $\ThetaBEEstrict$ to be the set of parameters
    \(\theta = (\betazero, \Distributionde{X, \eps}) \in \ThetaBEE\) such that the joint distribution $\Distributionde{X, \eps}$ satisfies:
    \begin{enumerate}
        \item[(\textit{i})]
        \(\minvp(\E[XX'\varepsilon^2]) \geq \minlambdaEXXtepssquare\);    
        \item[(\textit{ii})]
        \(\E\!\left[ \| X \|^{4(1+\rho)}\right] \leq \majorantEnormeXquatreunplusrho\).
    \end{enumerate}
\end{hyp}

Assumption~\ref{hyp:subset_exo_unif_asymptotic_exactness}(i) implies in particular a bound on the kurtosis of $\majorantKurtxi$ on $\ThetaBEEstrict$.
Therefore, this gives a potentially stricter bound on $\majorantKurtxi$ than what was previously assumed on $\ThetaBEE$.

\begin{prop}
\label{prop:exogenous_case_asymptotic_uniform_exactness}
    Assume that the sequences $(a_n)$, $(\omega_n)$ and $(\delta_n)$ satisfy the conditions of Theorem~\ref{thm:exogenous_case_asymptotic_pointwise_exactness}.
    Then, $\CIEdgu$ is asymptotically exact uniformly over $\ThetaBEEstrict$ at level \(1-\alpha\).
\end{prop}

Combining Theorem~\ref{thm:exogenous_case_non-asymptotic_validity}, Proposition~\ref{prop:exogenous_case_asymptotic_uniform_exactness}, and Equation~\eqref{eq:decomp_CI_fin_oP} guarantees that $\CIEdgu$ be NAVAE uniformly over $\ThetaBEEstrict$ and efficient (since it is of the form~\eqref{eq:def:efficient_CI}). 

\medskip

To achieve asymptotic uniform exactness, we restrict the parameter space from $\ThetaBEE$ (where non-asymptotic validity holds) to a smaller subset, $\ThetaBEEstrict$, by strengthening two moment conditions.
Part~(\textit{i}) requires that \(\minvp(\E[XX'\eps^2])\) is not only positive, but equal to or larger than the constant~\(\minlambdaEXXtepssquare\).
Part~(\textit{ii}) reinforces \(\E[\norme{X}^4] < +\infty\) in two directions.
First, it specifies an explicit upper bound on moments of $P_X$.
Second, it enables (if \(\rho > 0\)) to consider higher order moments, which yields faster rates (see Theorem~\ref{thm:exogenous_case_asymptotic_uniform_exactness}).
Overall, \(\ThetaBEEstrict \subsetneq \ThetaBEE \subsetneq \Theta\), and the properties of our interval under those different assumptions are summarized in Table~\ref{tab:properties_our_ci} below.

\begin{table}[htb]
    \vspace{0.3em}
    \renewcommand{\arraystretch}{1.3}
    \centering
    \begin{tabular}{l|lll}
        \textit{Property of \(\CIEdgu\) at level~\(1 - \alpha\) over} & \(\ThetaBEEstrict\) & \(\ThetaBEE\) & \(\Theta\) \\
        \hline
        Asymptotic pointwise exactness \eqref{eq:pointwise_exact_asymp} -- Theorem~\ref{thm:exogenous_case_asymptotic_pointwise_exactness} & yes & yes & yes
        \\
        Non-asymptotic validity \eqref{eq:definition_non-asymptotic_validity} -- Theorem~\ref{thm:exogenous_case_non-asymptotic_validity} & yes & yes & no
        \\         
        Asymptotic uniform exactness \eqref{eq:unif_exact_asymp} -- Proposition~\ref{prop:exogenous_case_asymptotic_uniform_exactness} & yes & no & no
    \end{tabular}
    \caption{Properties of our confidence interval \(\CIEdgu\) over different parameter spaces.
    They require adequate choices of the tuning parameters $a_n$, $\omega_n$, and $\delta_n$ as specified in the formal statement of those results.
    Remember that exactness implies validity.}
    \label{tab:properties_our_ci}
\end{table}

Theorem~\ref{thm:exogenous_case_asymptotic_pointwise_exactness} and Proposition~\ref{prop:exogenous_case_asymptotic_uniform_exactness} specify adequate rates for the choice of \(a_n\) and \(\omega_n\), although they do not provide explicit values for those tuning parameters.
As in the case of the expectation, we could consider minimizing the width of the interval (when it is informative) to choose them.
However, the situation is more intricate in the OLS case\footnote{Since the width of \(\CIEdgu\) is stochastic, and may not even admit an expectation.}.
Therefore, we follow another path here by focusing on the coverage probabilities.

\medskip

Proposition~\ref{prop:exogenous_case_asymptotic_uniform_exactness} is actually a consequence of the following theorem that gives a precise bound on the worst-case distance between the coverage probability of our CI and its nominal level.
It quantifies the error of the uniform asymptotic exactness property.

\begin{thm}[Non-asymptotic bound on uniform exactness with rates]
\label{thm:exogenous_case_asymptotic_uniform_exactness}
    Assume that the sequences $(a_n)$, $(\omega_n)$ and $(\delta_n)$ satisfy the conditions of Theorem~\ref{thm:exogenous_case_asymptotic_pointwise_exactness}.
    Then, there exist $C>0$ and $n_* > 1$ (both depending on $\alpha$) such that,
    for every $n \geq n_*$,
    we have
    \begin{align*}
        \sup_{\theta \in \ThetaBEEstrict}
        \Probsousthetan \Big( \uu' \betazero \in \CIEdgu \Big)
        &\leq 1 - \alpha
        + C\left\{\sqrt{b_n} + \nuEdg + \mathfrak{e}_n \right\},
    \end{align*}
    where $\mathfrak{e}_n :=$
    \begin{align*}
        \left\{
        \begin{array}{ll}
            \dfrac{1}{n^{1/4} \omega_n^{3/8}}
            & \mbox{if } \rho = 0, \\[1.5em]
            \min\!\left\{ \! \dfrac{1}{n^{1/4} \omega_n^{3/8}} 
            \, , \, 
            \dfrac{1}{(n\omega_n)^{1/4}}
            \! + \! \dfrac{1}{\sqrt{n}\omega_n^{3/4}}
            \! + \! \ln(n)\left(
            \dfrac{\Indicator\{\rho \leq 1\}}{n^{\rho}}
            \! + \! \dfrac{\Indicator\{\rho > 1\}}{n^{(1+\rho)/2}} 
            \right) \! \right\}
            & \mbox{if } \rho \in (0 , + \infty), \\[1.5em]
            \dfrac{1}{(n\omega_n)^{1/4}} + \dfrac{1}{\sqrt{n}\omega_n^{3/4}}
            & \mbox{if } \rho = + \infty.
        \end{array}
        \right.
    \end{align*}
\end{thm}

Combining this result with the non-asymptotic validity of our confidence interval (Theorem~\ref{thm:exogenous_case_non-asymptotic_validity}), we directly obtain that
\begin{align*}
    \sup_{\theta \in \ThetaBEEstrict}
    \left|\Probsousthetan \Big( \uu' \betazero \in \CIEdgu \Big) - (1 - \alpha)\right|
    &\leq C\left\{\sqrt{b_n} + \nuEdg + \mathfrak{e}_n \right\},
\end{align*}
for the same constant $C > 0$ and $n \geq n_*$.
We now specialize Theorem~\ref{thm:exogenous_case_asymptotic_uniform_exactness} to obtain the best rates for explicit values of the tuning parameters of the form $\omega_n = n^{-a}$ and $b_n = n^{-b}$.

\begin{prop}
\label{prop:exogenous_case_exact_rates}
    Assume that the sequence $(\delta_n)$ satisfies the conditions of Theorem~\ref{thm:exogenous_case_asymptotic_pointwise_exactness}.
    Let $\rho \in [0, +\infty]$.
    If $\omega_n = n^{-r(\rho)}$ and
    $b_n = n^{- b}$ for any $b \in [2/5, 1/2)$,
    then
    \begin{align*}
        \sup_{\theta \in \ThetaBEEstrict}
        \Probsousthetan \Big( \uu' \betazero \in \CIEdgu \Big)
        &\leq 1 - \alpha
        + C
        n^{- r(\rho)},
    \end{align*}
    for some constant $C > 0$, where
    \begin{align*}
        r(\rho) := \frac{2}{11} \Indicator\{\rho < 2/11\}
        + \rho \Indicator\{2/11 \leq \rho \leq 1/5\}
        + \frac{1}{5} \Indicator\{1/5 < \rho\}.
    \end{align*}
\end{prop}

This proposition is proved in Section~\ref{proof:prop:exogenous_case_exact_rates}.
Interestingly, additional constraints on the moments of $\| X \|$ from $4$ to $4 \times (1 + 2/11) = 52 / 11 \approx 4.72$ does not seem to improve the rate.
Between the $4.72$-th moment and the $4 \times (1 + 1/5) = 24 / 5 = 4.75$-th moment, the rate $r(\rho)$ is the identity function.
Beyond the $4.75$-th moment, the rate is fixed and not improved by boundedness of additional moments.
Overall, the effect of $\rho$ is quite mild since with four moments the rate is close to $n^{-0.1819}$ whereas it attains $n^{-0.20}$ when all moments are bounded.

\section{Practical considerations and simulation study}
\label{sec:practical_considerations_plugin}

\texttt{R} code for our simulations is available at\newline  \url{https://github.com/AlexisDerumigny/Reproducibility-NAVAE_CI_LinModels}.

\subsection{Plug-in}
\label{ssec:plug_in_discussion}

Be it for expectations or the coefficients of a linear regression, our assumptions impose lower or upper bounds on moments (or functions thereof) of the distribution \(\Distributionde{\xi}\) (Section~\ref{sec:expectation_example}) and \(\Distribution_{X, \, \eps}\) (Section~\ref{sec:lin_mod_exo}).
Remember that those bounds are required to compute our CIs.
While a priori choices for those bounds are natural in specific cases (see Remark~\ref{rem:no_plug_in_bounds} below), it often remains difficult to form intuition about their values or how to choose them.
Instead, a natural idea is to replace these (unknown) bounds with estimates of the corresponding moments.
In the case of expectations, the confidence interval depends only on one bound~\(K\) (see Section~\ref{ssec:setting_and_motivation}), on the kurtosis of the distribution of the observations, which can be approximated by
\begin{itemize}
    \item \(\widehat{K} := n^{-1}
    \sum_{i=1}^n (\xi_i - \MeanXi)^4 \, / \,
    \big[n^{-1} \sum_{i=1}^n (\xi_i - \MeanXi)^2\big]^2 \).
\end{itemize}
For linear regressions, several bounds are involved (see Assumption~\ref{hyp:subset_exo_non-asymptotic_validity}) that can be approximated by
\begin{itemize}
    \item \(\widehat\minlambdaEXXt 
    := \minvp(n^{-1} \sum_{i=1}^n X_i X_i'))\),
    
    \item \(\widehat\majorantEXXprimeconcentration
    := n^{-1} \sumiunn 
    \Big\| \operatorvec
    \Big( \widehat{\widetilde{X}_i} \widehat{\widetilde{X}'_i}
    - \mathbb{I}_p \Big)
    \Big\|^2\) with
    \(\widehat{\widetilde{X}_i}
    := \left(\big( n^{-1} \sum_{j=1}^{n} X_j X_j' \big)^{\dagger}\right)^{1/2} X_i\),
    
    \item \(\widehat\majorantEnormeXepsquatre
    := n^{-1} \sum_{i=1}^n 
    \Big\|\widehat{\widetilde{X}_i} \epshat_i \Big\|^4\)
    with \(\widehat{\eps}_i := Y_i-X_i'\widehat{\beta}\),
        
    \item \(\widehat \majorantKurtxi
    := n^{-1} \sum_{i=1}^n \widehat{\xi_i}^4 \,/\, \Big( n^{-1} \sum_{i=1}^n \widehat{\xi_i}^2 \Big)^2\)
    with \(\widehat\xi_i
    := \uu'\big( n^{-1} \sum_{j=1}^{n} X_j X_j' \big)^{\dagger}
    X_i \epshat_i\).  
\end{itemize}

Using plug-in estimates rather than deterministic bounds has one major drawback from a theoretical point of view: the resulting CI is no longer non-asymptotically valid.
Nevertheless, in both cases (Sections~\ref{sec:expectation_example} and~\ref{sec:lin_mod_exo}), it is still pointwise asymptotically exact over \(\ThetaBEE\).
More generally, with a similar reasoning as that of Theorem~\ref{thm:exogenous_case_asymptotic_pointwise_exactness} (where the choice of the bounds \((\majorantKurtxi, \majorantEXXprimeconcentration, \majorantEnormeXepsquatre, \minlambdaEXXt)\) does not matter), plug-in CIs remain asymptotically exact pointwise as soon as the previous plug-in approximations have finite non-zero limits in probability. 
In a sense, the use of plug-in in our approach can be compared to the approximation error faced in practice when one uses an inference procedure based on simulations such as that of \cite{cherno2009} or \cite{diciccio2017robust}. 
To further control the impact of sampling uncertainty when using a plug-in version of our interval, a practical possibility would be to multiply the plug-in estimates by \((1 + M/\sqrt{n})\) for some positive constant~\(M\).

\begin{rem}[Some possibilities to overcome plug-in]
\label{rem:no_plug_in_bounds}
Choosing reasonable values for \(K\) and \(\majorantKurtxi\) without resorting to a plug-in strategy turns out to be possible. 
A large class of univariate distributions exhibits a bound of at most $9$ on the kurtosis: Normal, Laplace, asymmetric Laplace, Logistic, Uniform, Student with at least five degrees of freedom, two-point symmetric mixtures of Normals, Gumbel, hyperbolic secant, and skewed Normal. 
This class includes both symmetric and asymmetric distributions, some of which only have a few number of finite moments (Student distributions with few degrees of freedom).
We investigate the impact of the choice \(K = 9\) for expectations (respectively, \(\majorantKurtxi = 9\) for linear regressions) as an alternative to the plug-in approach in the following simulation section.
%
\end{rem}

\subsection{Simulations for Section~\ref{sec:expectation_example}: inference on an expectation}
\label{subsec:simu_expectation}

\subsubsection{Framework}

This section presents some simulation results on our confidence interval for an expectation.
We consider an i.i.d. sample from an Exponential distribution with expectation set to~\(1\), the targeted parameter~\(\theta_1\) here.
Compared to Normal distributions, remember that Exponential distributions are skewed (with a skewness coefficient equal to~\(2\)) and display fatter tails with a kurtosis of~\(9\).
In that sense, we consider the boundary case (while remaining correctly specified) when setting a bound~\(K = 9\) on the kurtosis as discussed in Remark~\ref{rem:no_plug_in_bounds}.
We present the results for that choice \(K = 9\) and for a plug-in version of our CI using the empirical kurtosis~\(\widehat{K}\) instead.

\medskip

In addition to the bound~\(K\), \(\CIsigmaknown\) depends on \(\delta_n\) and \(\CIsigmanhat\) depends additionally on \(a_n\).
We follow Remark~\ref{rem:valid_delta_n} by choosing \(\delta_n\) as the minimum between \(\delta_{1,n}\) proposed by \cite{shevtsova2013} and \(\delta_{2,n}\) by \cite{derumigny2022explicitBE}.
We thus do not take advantage of the continuity of the Exponential distribution (see \(\delta_{3,n}\) and the comparisons made in \cite{derumigny2022explicitBE} for further details).
We set \(a_n = 1 + n^{-1/5}\) satisfying the requirements of Proposition~\ref{prop:mean_CI_BE-Edg_asymp_sigmaunknown} for the asymptotic exactness of our CI.
Choosing that tuning parameter is a tradeoff between exiting the \(\Rb\) regime earlier
(meaning, for smaller~\(n\) or smaller~\(\alpha\))
and the precision of our CI (meaning, its width).
For instance, with a smaller power of \(n\), say, \(a_n = 1 + n^{-1/10}\), the minimal \(\alpha\) that exits the \(\Rb\) regime with \(K = 9\) and \(n = 1,000\) is \(15.6\)\% instead of 26.1\% (see Table~\ref{tab:minimal_level_to_exit_R_regime} below), but the resulting CI's width is larger.

\subsubsection{About the minimum level \texorpdfstring{$\alphamin$}{alpha min}}

We focus here on the case of an unknown variance.
Table~\ref{tab:minimal_level_to_exit_R_regime} reports $\alphamin$, defined as the minimal \(\alpha\) for which our CI is informative, that is, satisfies \(\CIsigmanhat \subsetneq \Rb\).
Note that $\alphamin$ depends on $K$, $a_n$ and $\delta_n$.
For instance, for \(K = 9\), \(a_n = 1 + n^{-1/5}\) and \(\delta_n\) chosen as above, a sample with \(n = 5,000\) observations allows to compute an informative
\(\CIsigmanhat\) at level~90\% but not at~95\%
(since \(\alphamin \approx 7\)\%).
Intuitively, given the tools we use to build our CI and the values of \(K\), \(a_n\), and \(\delta_n\), a 95\% nominal level is too strong a requirement to guarantee non-asymptotic validity for that sample size.
With \(K = \widehat{K}\), the condition to exit the \(\Rb\) regime becomes random and so does $\alphamin$.
We report the mean and median of \(\alphamin\) obtained over \(M = 20,000\) Monte-Carlo repetitions.
As expected, for large enough sample sizes, using \(K = 9\) or \(K = \widehat{K}\) leads to indistinguishable results as \(\widehat{K}\) converges to the kurtosis of an Exponential (equal to~\(9\)).

\begin{table}[ht]
    \centering
    \begin{tabular}{c| c c c c c c c}
        &
        \(n = 500\) & 
        \(n = 1\)k & 
        \(n = 5\)k & 
        \(n = 10\)k & 
        \(n = 50\)k & 
        \(n = 100\)k & 
        \(n = 200\)k
        \\
        \hline
        $\alphamin(9)$ &        
        0.466 & 0.261 & 0.0703 & 0.0488 & 0.0200 & 0.0111 & 0.00647 \\
        $\E[\alphamin(\widehat{K})]$ &
        0.423 & 0.249 & 0.0702 & 0.0486 & 0.0200 & 0.0111 & 0.00647 \\
        $\text{Median}[\alphamin(\widehat{K})]$ &
        0.393 & 0.229 & 0.0685 & 0.0482 & 0.0199 & 0.0111 & 0.00647
    \end{tabular}   
    \caption{Minimal level \(\alphamin(K)\) that exits the \(\mathbb{R}\) regime of \(\CIsigmanhat\),
    with the choice \(a_n = 1 + n^{-1/5}\) for different sample sizes and different values of the bound $K$ (fixed to $9$ or data-driven).
    } 
    \label{tab:minimal_level_to_exit_R_regime}
\end{table}

\subsubsection{Coverage performance}

We now focus on the choice \(K = 9\) and consider coverage performance for \(\alpha = 0.10\). We present the results by comparing our CIs with the classical one derived from the Central Limit Theorem and Slutsky's lemma, \(\CIalphannumero{CLT}\).
Table~\ref{tab:coverage_expectation_case} reports the empirical coverage over \(M = 20,000\) Monte-Carlo repetitions.
As expected from Propositions~\ref{prop:mean_CI_BE-Edg_sigmaunknown} and~\ref{prop:mean_CI_BE-Edg_asymp_sigmaunknown}, the coverage of our CI is always greater than the nominal level \(1 - \alpha\), 90\% here, and decreases to the nominal level when the sample size increases.

\medskip

In the unknown variance case, we compare two versions of our CI: one with the a priori choice \(a_n = 1 + n^{-1/5}\) -- which satisfies the requirements of Proposition~\ref{prop:mean_CI_BE-Edg_asymp_sigmaunknown} -- and one with the optimized choice of \(a_n\) discussed in Remark~\ref{rem:optimized_choice_an}.
The latter choice logically reduces coverage closer to the targeted nominal level, although the improvement remains limited.

\medskip

Finally, the third row presents the coverage of \(\CIsigmaknown\).
As expected, the coverage of this oracle CI lies in between those of the standard CLT-based CI and of our feasible one since \(\CIsigmaknown\) bypasses the additional step of replacing the variance by a consistent estimator.

\begin{table}[ht]
    \centering
    \begin{tabular}{c| c c c c c}
        Confidence interval &
        \(n = 5\)k & 
        \(n = 10\)k & 
        \(n = 50\)k & 
        \(n = 100\)k & 
        \(n = 200\)k
        \\
        \hline        
        \(\CIsigmanhat\), \(a_n = 1 + n^{-1/5}\) & 
        0.981 & 0.965 & 0.934 & 0.927 & 0.920        
        \\
        \(\CIsigmanhat\), optimized \(a_n \) &
        0.981 & 0.963 & 0.927 & 0.918 & 0.910        
        \\
        \(\CIsigmaknown\) & 
        0.970 & 0.952 & 0.919 & 0.911 & 0.905        
        \\
        \(\CIalphannumero{CLT}\) &
        0.901 & 0.903 & 0.899 & 0.900 & 0.899                    
    \end{tabular}   
    \caption{Coverage (approximated using Monte Carlo simulations) for different sample sizes, with the choice \(K = 9\) for the bound on the kurtosis for our CIs and \(\alpha = 0.10\).} 
    \label{tab:coverage_expectation_case}
\end{table}

\subsubsection{Width of the confidence intervals}

Figure~\ref{fig:width_ourCI_versus_CLTCI_expectation_case} compares the confidence intervals \(\CIsigmaknown\), \(\CIsigmanhat\) and \(\CIalphannumero{CLT}\) regarding their width using the bound \(K = 9\) on the kurtosis and the optimized tuning parameter \(a_n\) when the variance is unknown.

\medskip

Panel~(a) reports the average width of the intervals as a function of the sample size~\(n\) over \(M = 20,000\) Monte-Carlo repetitions (the absolute widths of the latter two CIs are data-dependent, hence stochastic, through \(\widehat{\sigma}\)).

\medskip

Panel~(b) shows the relative width of our CIs with respect to the usual CLT-based CI.
When the variance is assumed to be known, we also use that information for the CLT-based interval and compare \(\CIsigmaknown\) to the analogue of \(\CIalphannumero{CLT}\) with a known variance, replacing \(\sigmaknown\) with \(\sigmanhat\) in Equation~\eqref{eq:definition_CI_CLT_for_expectation}.
The ratio of their widths (blue line in Panel~(b)) is equal to \(\QuantileGaussAt{1 - \alpha/2 + \delta_n} \, / \; \QuantileGaussAt{1 - \alpha/2}\) and is not data-dependent (only \(n\) and $K$ matter, through \(\delta_n\)).
When the variance is unknown, the width ratio between \(\CIsigmanhat\) and \(\CIalphannumero{CLT}\) (red line) is equal to \(C_n \QuantileGaussAt{1 - \alpha/2 + \delta_n + \nuVar/2} \, / \; \QuantileGaussAt{1 - \alpha/2}\), and is also not data-dependent as \(\sigmanhat\) cancels out.

\medskip

The convergence to~\(1\) of the relative widths illustrates that our CI coincides in the limit with the classical one obtained from the CLT.
For information, the inflexion point (at \(n = \) 38,707) corresponds to the switch from \(\delta_{1, n}\) to \(\delta_{2, n}\), that is, from Berry-Esseen inequalities to Edgeworth Expansions to obtain \(\delta_n\) (see Remark~\ref{rem:valid_delta_n}).

\begin{figure}[ht]
     \centering
     \begin{subfigure}[b]{0.45\textwidth}
         \centering
         \includegraphics[width=1\textwidth]{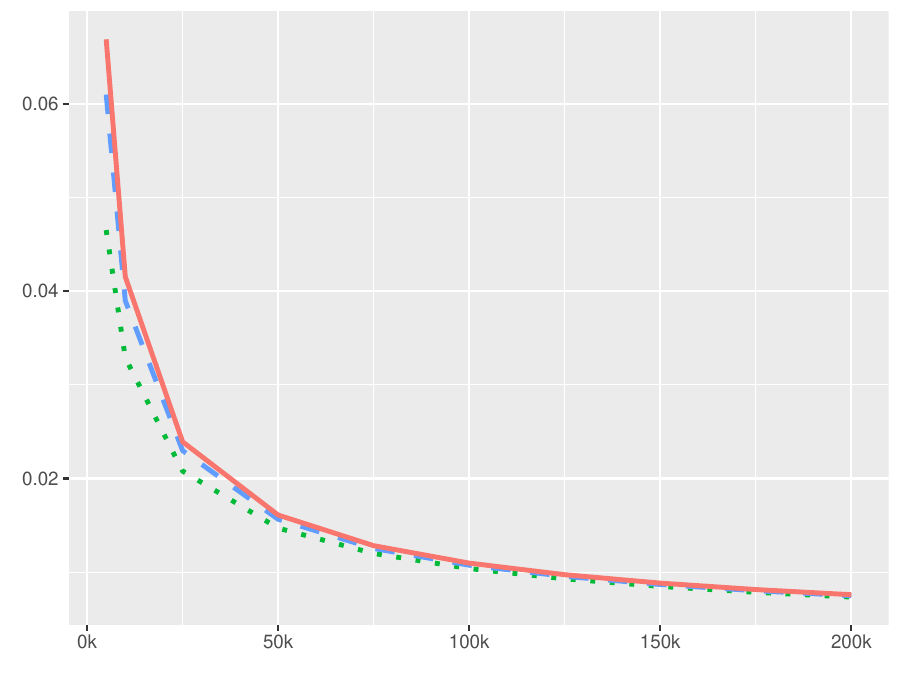}
         \caption{{\footnotesize Absolute width of \(\CIsigmanhat\) with optimized \(a_n\) (solid red), \(\CIsigmaknown\) (dashed blue), and \(\CIalphannumero{CLT}\) (dotted green).}}
         \label{fig:subfig:absolute_width_expectation_case}
     \end{subfigure}
     \hspace{0.05\textwidth}
     \begin{subfigure}[b]{0.45\textwidth}
         \centering
         \includegraphics[width=1\textwidth]{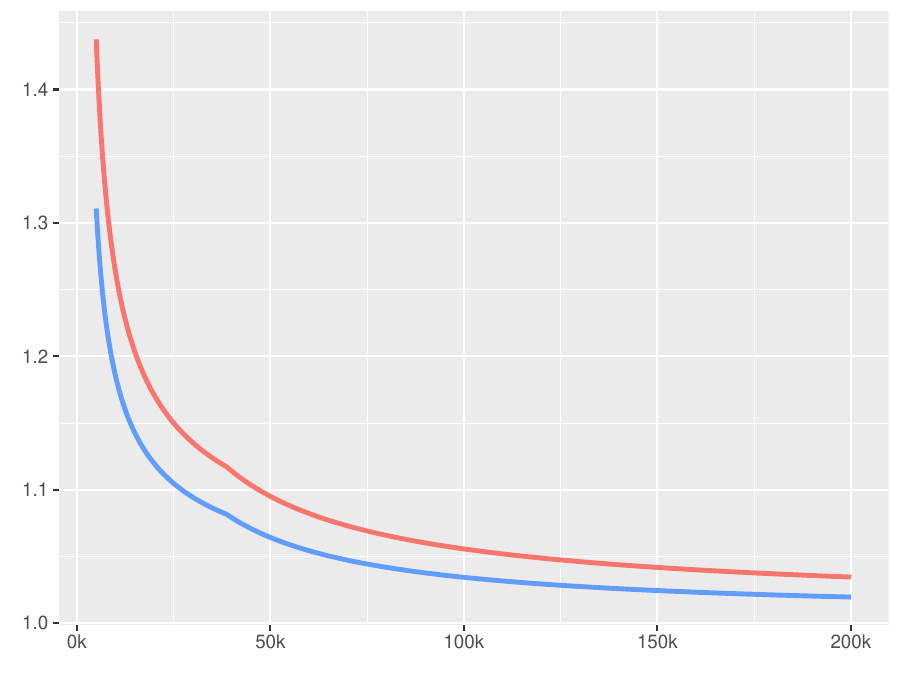}         
         \caption{{\footnotesize Relative width of \(\CIsigmanhat\) with optimized \(a_n\) (red) and \(\CIsigmaknown\) (blue),
         each with respect to \(\CIalphannumero{CLT}\).}}
         \label{fig:subfig:relative_width_expectation_case}
     \end{subfigure}
    \caption{Width as a function of the sample size, for the choices \(K = 9\) and \(\alpha = 0.10\).}  
    \label{fig:width_ourCI_versus_CLTCI_expectation_case}         
\end{figure}

\subsection{Simulations for Section~\ref{sec:lin_mod_exo}: inference on linear regression' coefficients}

We now consider a simulation study on linear regressions. 
We generate i.i.d. samples from the model
\(Y = 2 + 1 X_1 - 3 X_2 + \eps\) where \((X_1, X_2)\) follows a centered bivariate Normal distribution, with respective variances 1 and 2 and a correlation set to 0.5, and \(\eps\) drawn from a Gumbel distribution 
parameterized such that the error terms' expectation is null and their conditional variance is equal to \((X_1 + X_2)^2\), implying heteroskedasticity.
Gumbel distributions are skewed and have heavier tails than Normal ones.
We perform inference on \(\betazeroindice{2}\), that is, we choose \(\uu = (0,1,0)'\).

\medskip

For \(\delta_n\), we choose again the minimum between \(\delta_{1, n}\) and \(\delta_{2, n}\) (see Remark~\ref{rem:valid_delta_n}).
While our CIs for expectations require only one bound \(K\) on the kurtosis, several bounds are needed to compute \(\CIEdgu\). 
As explained in Section~\ref{ssec:plug_in_discussion}, while the choice \(\majorantKurtxi = 9\) is sensible for the kurtosis of the influence function, fixing values for \(\minlambdaEXXt\), \(\majorantEXXprimeconcentration\), and \(\majorantEnormeXepsquatre\) is not straightforward. In the simulations, we stick to the choice \(\majorantKurtxi=9\) and resort to plug-in versions of the other bounds. For the tuning parameters, we set \(b_n = 20 \times n^{-2/5}\) and \(\omega_n = n^{-1/5}\) which complies with the assumptions underlying Propositions~\ref{prop:exogenous_case_asymptotic_uniform_exactness} and~\ref{prop:exogenous_case_exact_rates}.

\medskip

Remember that our interval is informative (not equal to~\(\Rb\)) when the sample size is such that \(n > 2 \majorantEXXprimeconcentration / (\omega_n \alpha)\) and \(\nuEdg < \alpha / 2\).
With fixed bounds, both conditions are non-stochastic.
In our simulations, the former becomes stochastic as we use the plug-in \(\widehat{\majorantEXXprimeconcentration}\).
However, in our setting and for \(\alpha = 0.10\), only the latter condition, \(\nuEdg < \alpha / 2\), happens to be binding and is satisfied for \(n \geq\) 3,656.

\medskip

Figure~\ref{fig:width_ourCI_versus_CLTCI_OLS_case} reports the absolute and relative width of \(\CIEdgu\) with respect to \(\CIAsympu\) for \(n\) ranging from 5,000 to 200,000; more precisely, we report their average over \(M =\) 20,000 Monte-Carlo repetitions.
Our CI remains far wider than the one based on the CLT, even for large sample sizes.
This can stem from two possible sources: either a lack of tightness of our intervals or a substantial gap between the shortest possible NAVAE CI and \(\CIAsympu\) over the class of distributions we consider.
Finding the shortest non-asymptotically valid confidence interval is a hard task and this question is left for future research.

\begin{figure}[ht]
    \centering
     \begin{subfigure}[b]{0.45\textwidth}
         \centering
         \includegraphics[width=1\textwidth]{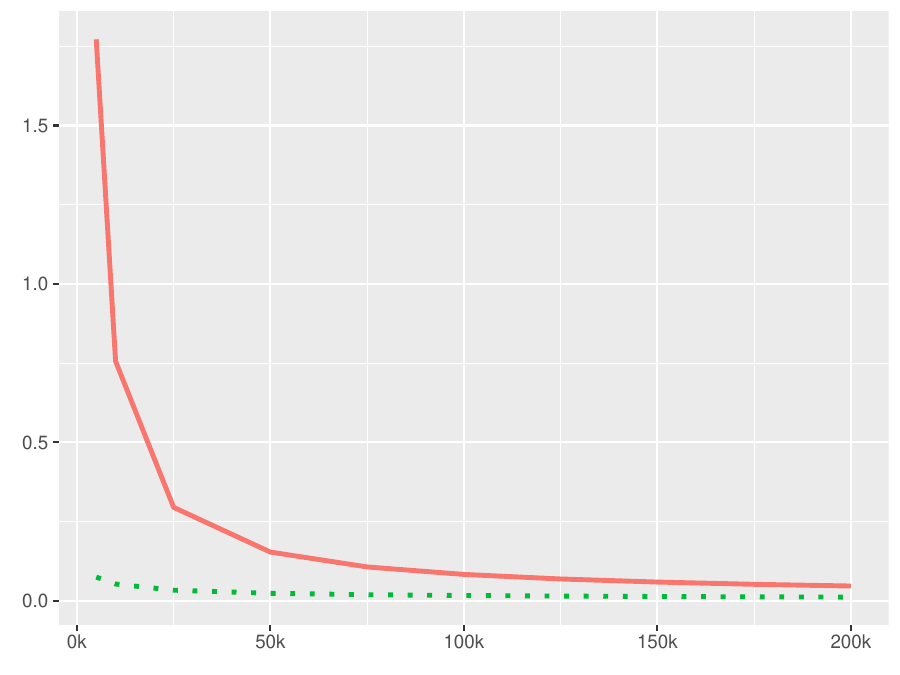}
         \caption{{\footnotesize Absolute width of \(\CIEdgu\) with \(a_n = 1 + 20 \times n^{-2/5}\) and \(\omega_n = n^{-1/5}\) (solid red) and \(\CIAsympu\) (dotted green).}}
         \label{fig:subfig:absolute_width_ols_case}
     \end{subfigure}
     \hspace{0.05\textwidth}
     \begin{subfigure}[b]{0.45\textwidth}
         \centering
         \includegraphics[width=1\textwidth]{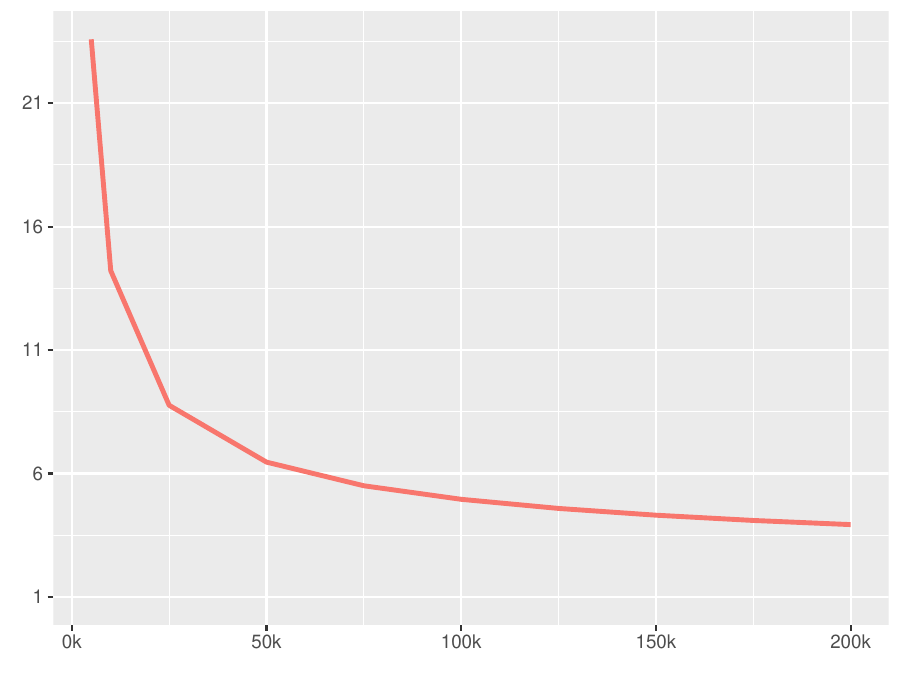}         
         \caption{{\footnotesize Relative width of \(\CIEdgu\) with \(a_n = 1 + 20 \times n^{-2/5}\) and \(\omega_n = n^{-1/5}\) (red) with respect to \(\CIAsympu\).}}
         \label{fig:subfig:relative_width_ols_case}
     \end{subfigure}
    \caption{Width as a function of the sample size, for the choices \(K = 9\) and \(\alpha = 0.10\).}  
    \label{fig:width_ourCI_versus_CLTCI_OLS_case}         
\end{figure}

\bibliography{mainV7_ArXiv}

\begin{thebibliography}{10}

\bibitem{ArmstrongKolesarQE2020SimpleHonnestCI}
{\sc Armstrong, T.~B., and Koles{\'a}r, M.}
\newblock Simple and honest confidence intervals in nonparametric regression.
\newblock {\em Quantitative Economics 11}, 1 (2020), 1--39.

\bibitem{austern2022efficient}
{\sc Austern, M., and Mackey, L.}
\newblock Efficient concentration with gaussian approximation.
\newblock {\em arXiv preprint arXiv:2208.09922\/} (2022).

\bibitem{bahadur1956}
{\sc Bahadur, R.~R., and Savage, L.~J.}
\newblock The nonexistence of certain statistical procedures in nonparametric problems.
\newblock {\em The Annals of Mathematical Statistics 27}, 4 (1956), 1115--1122.

\bibitem{berry1941}
{\sc Berry, A.}
\newblock The accuracy of the gaussian approximation to the sum of independent variates.
\newblock {\em Transactions of the American Mathematical Society 49\/} (1941), 122--136.

\bibitem{bertanha2016impossible}
{\sc Bertanha, M., and Moreira, M.~J.}
\newblock Impossible inference in econometrics: Theory and applications.
\newblock {\em Journal of Econometrics 218}, 2 (2020), 247--270.

\bibitem{boucheron2013concentration}
{\sc Boucheron, S., Lugosi, G., and Massart, P.}
\newblock {\em Concentration inequalities: A nonasymptotic theory of independence}.
\newblock Oxford university press, 2013.

\bibitem{cherno2009}
{\sc Chernozhukov, V., Hansen, C., and Jansson, M.}
\newblock Finite sample inference for quantile regression models.
\newblock {\em Journal of Econometrics 152}, 2 (2009), 93 -- 103.
\newblock Nonparametric and Robust Methods in Econometrics.

\bibitem{cramer1962}
{\sc Cramer, H.}
\newblock {\em Random Variables and Probability Distributions}, 2~ed.
\newblock Cambridge University Press, 1962.

\bibitem{BoundEdgeworth}
{\sc Derumigny, A., Girard, L., and Guyonvarch, Y.}
\newblock {\em {BoundEdgeworth}: Bound on the Error of the First-Order Edgeworth Expansion}, 2023.
\newblock {R} package version 0.1.2.2. Available at \url{https://github.com/AlexisDerumigny/BoundEdgeworth}.

\bibitem{derumigny2022explicitBE}
{\sc Derumigny, A., Girard, L., and Guyonvarch, Y.}
\newblock Explicit non-asymptotic bounds for the distance to the first-order edgeworth expansion.
\newblock {\em Sankhya A\/} (2023).

\bibitem{NAVAECI}
{\sc Derumigny, A., Girard, L., and Guyonvarch, Y.}
\newblock {\em NAVAECI: Non-Asymptotically Valid and Asymptotically Exact (NAVAE) Confidence Intervals}, 2025.
\newblock {R} package version 0.1.0. Available at \url{https://github.com/AlexisDerumigny/NAVAECI}.

\bibitem{dHaultfoeuille2024robust}
{\sc D'Haultf{\oe}uille, X., and Tuvaandorj, P.}
\newblock A robust permutation test for subvector inference in linear regressions.
\newblock {\em Quantitative Economics 15}, 1 (2024), 27--87.

\bibitem{diciccio2017robust}
{\sc DiCiccio, C.~J., and Romano, J.~P.}
\newblock Robust permutation tests for correlation and regression coefficients.
\newblock {\em Journal of the American Statistical Association 112}, 519 (2017), 1211--1220.

\bibitem{dufour1997some}
{\sc Dufour, J.-M.}
\newblock Some impossibility theorems in econometrics with applications to structural and dynamic models.
\newblock {\em Econometrica: Journal of the Econometric Society\/} (1997), 1365--1387.

\bibitem{esseen1942}
{\sc Esseen, C.-G.}
\newblock On the liapunoff limit of error in the theory of probability.
\newblock {\em Arkiv för Matematik, Astronomi och Fysik.\/} (1942).

\bibitem{esseen1945}
{\sc Esseen, C.-G.}
\newblock Fourier analysis of distribution functions. a mathematical study of the {L}aplace-{G}aussian law.
\newblock {\em Acta Math. 77\/} (1945), 1--125.

\bibitem{gossner2013finite}
{\sc Gossner, O., and Schlag, K.~H.}
\newblock Finite-sample exact tests for linear regressions with bounded dependent variables.
\newblock {\em Journal of Econometrics 177}, 1 (2013), 75--84.

\bibitem{Gut2005}
{\sc Gut, A.}
\newblock {\em Probability: A Graduate Course}.
\newblock Springer, 2005.

\bibitem{hall1995}
{\sc Hall, P., and Jing, B.-Y.}
\newblock Uniform coverage bounds for confidence intervals and berry-esseen theorems for edgeworth expansion.
\newblock {\em Ann. Statist. 23}, 2 (04 1995), 363--375.

\bibitem{hogben2006handbook}
{\sc Hogben, L.}
\newblock {\em Handbook of linear algebra}.
\newblock CRC press, 2006.

\bibitem{kasy2019}
{\sc Kasy, M.}
\newblock Uniformity and the delta method.
\newblock {\em Journal of Econometric Methods 8}, 1 (2019).

\bibitem{kuchibhotla2024hulc}
{\sc Kuchibhotla, A.~K., Balakrishnan, S., and Wasserman, L.}
\newblock The hulc: confidence regions from convex hulls.
\newblock {\em Journal of the Royal Statistical Society Series B: Statistical Methodology 86}, 3 (2024), 586--622.

\bibitem{pena2008self}
{\sc Pe{\~n}a, V., Lai, T., and Shao, Q.}
\newblock {\em Self-Normalized Processes: Limit Theory and Statistical Applications}.
\newblock Probability and Its Applications. Springer Berlin Heidelberg, 2008.

\bibitem{pouliot2024}
{\sc Pouliot, G.~A.}
\newblock An exact t-test, 2024.

\bibitem{romano2004}
{\sc Romano, J.~P.}
\newblock On non-parametric testing, the uniform behaviour of the t-test, and related problems.
\newblock {\em Scandinavian Journal of Statistics 31}, 4 (2004), 567--584.

\bibitem{romano2000}
{\sc Romano, J.~P., and Wolf, M.}
\newblock Finite sample nonparametric inference and large sample efficiency.
\newblock {\em The Annals of Statistics 28}, 3 (2000), 756--778.

\bibitem{brunel2019nonasymptotic}
{\sc Schreuder, N., Brunel, V.-E., and Dalalyan, A.}
\newblock A nonasymptotic law of iterated logarithm for general m-estimators.
\newblock In {\em Proceedings of the Twenty Third International Conference on Artificial Intelligence and Statistics\/} (26--28 Aug 2020), S.~Chiappa and R.~Calandra, Eds., vol.~108 of {\em Proceedings of Machine Learning Research}, PMLR, pp.~1331--1341.

\bibitem{shevtsova2013}
{\sc Shevtsova, I.}
\newblock On the absolute constants in the berry--esseen inequality and its structural and nonuniform improvements.
\newblock {\em Informatika i Ee Primeneniya [Informatics and its Applications] 7}, 1 (2013), 124--125.

\bibitem{vanderVaart2000}
{\sc van~der Vaart, A.}
\newblock {\em Asymptotics Statistics}.
\newblock Cambridge University Press, 2000.

\bibitem{waudby2024estimating}
{\sc Waudby-Smith, I., and Ramdas, A.}
\newblock Estimating means of bounded random variables by betting.
\newblock {\em Journal of the Royal Statistical Society Series B: Statistical Methodology 86}, 1 (2024), 1--27.

\end{thebibliography}

\newpage
\appendix

\section{Existing (non-NAVAE) confidence intervals for an expectation}
\label{appendix:sec:illustration_expectation}

This appendix considers the setting of Section~\ref{sec:expectation_example}:
conducting inference on the expectation of a real random variable that admits (at least) a non-zero finite second-order moment.
It reviews classical confidence sets and their corresponding properties in light of the NAVAE target.

\medskip 

Let $\sigmadeuxnhat:= n^{-1} \sum_{i=1}^n(\xi_i-\overline{\xi}_n)^2$, with $\overline{\xi}_n := n^{-1} \sumiunn \xi_i$. 
From the Central Limit Theorem (CLT) and Slutsky's lemma,
\[
\CSalphannumero{CLT} := \left[\overline{\xi}_n\pm\frac{q_{\mathcal{N}(0,1)}(1-\alpha/2)\sqrt{\sigmadeuxnhat}}{\sqrt{n}}\right]
\]
is asymptotically exact pointwise (property~\eqref{eq:pointwise_exact_asymp}) over 
$\Theta$, that is, over distributions with a non-zero finite second-order moment.
Following \cite{kasy2019} (in particular its Proposition~1), $\CSalphannumero{CLT}$ is also asymptotically exact uniformly (property~\eqref{eq:unif_exact_asymp}) over the subset
\[
\ThetaKasy:=\left\{\theta\in\Theta:\Varsoustheta(\xi)\geq m,\Expsousthetan[\xi^4]\leq M\right\}
,
\]
with $m > 0$, $M<+\infty$.
That uniform property holds on a large subset of $\Theta$ insofar as \(\ThetaKasy\) is nonparametric.
However, \(\CSalphannumero{CLT}\) is not non-asymptotically valid over~\(\Theta\) (neither over~\(\ThetaKasy\)).
It is thus not NAVAE over~\(\Theta\).
In fact, as explained in Section~\ref{sec:expectation_example}, some restrictions on~\(\Theta\) are necessary to obtain NAVAE confidence sets.

\medskip 

Imposing parametric distributional restrictions is a first possibility.
Up to using the quantiles of a Student instead of a standard Normal distribution, it is well-known that this simple modification of $\CSalphannumero{CLT}$, denoted by
\[
\CSalphannumero{St} :=
\left[ \overline{\xi}_n \pm \frac{q_{\textnormal{Student}(n-1)}(1-\alpha/2)\sqrt{(n / (n-1))\sigmadeuxnhat}}{\sqrt{n}}\right],
\]
is non-asymptotically exact (Property~\eqref{eq:definition_non-asymptotic_exactness}) over
\[
\ThetaGauss := \left\{ \theta\in\Theta : \Distributionde{\theta} = \Normale(\theta_1, \sigmadeux), \theta_1 \in \Rb, \sigmadeux \in \Rstarplus \right\}
,
\]
the set of Gaussian distributions.
Consequently, \(\CSalphannumero{St}\) is a fortiori NAVAE over~\(\ThetaGauss\).
Besides, it is also asymptotically exact pointwise over \(\Theta\), as \(\CSalphannumero{CLT}\).

\medskip 

Another way to non-asymptotic guarantees dispenses with parametric restrictions at the expense of known bounds on some moments of \(\Distributionde{\xi}\).
Let \(\ThetaBC:=\left\{\theta\in\Theta: \Varsoustheta(\xi)\leq M\right\}\),
for some \({M < +\infty}\).
An alternative CS for the expectation $\theta_1$ is
\[
\CSalphannumero{BC}:=\left[\overline{\xi}_n \pm 
\frac{\sqrt{M}}{\sqrt{\alpha n}} \right].
\]
Using the Bienaymé-Chebyshev inequality, it can be shown that $\CSalphannumero{BC}$ is non-asymptotically valid (Property~\eqref{eq:definition_non-asymptotic_validity}) over~\(\ThetaBC\).
Note that \(\CSalphannumero{BC}\) requires a known upper bound on the variance.
On the other hand, compared to \(\ThetaGauss\), \(\ThetaBC\) can be deemed a large subset of \(\Theta\) since it is infinite-dimensional.

\medskip

However, this CS has one major drawback: it is asymptotically conservative, even pointwise, over \(\ThetaBC\).
To see this, one can simply remark that for every $\alpha\in(0,1),$ $1/\sqrt{\alpha}>q_{\mathcal{N}(0,1)}(1-\alpha/2)$.
This implies that $1/\sqrt{\alpha}=q_{\mathcal{N}(0,1)}(1-\beta/2)$ for some $\beta<\alpha$.
Thus, for any \(\theta \in \ThetaBC\), the probability that \(\CSalphannumero{BC}\) contains \(\theta_1\) has a limit at least \(1 - \beta > 1 - \alpha\) when $n$ goes to infinity.
\(\CSalphannumero{BC}\) is thus not NAVAE over~\(\ThetaBC\).

\medskip

$\CSalphannumero{BC}$ is a basic instance of a CS constructed thanks to a concentration inequality.
There exist many different concentration inequalities (see \cite{boucheron2013concentration} for an in-depth exposition) relying on alternative restrictions $\Thetatilde$ on the distribution of the data, each yielding a CS that is non-asymptotically valid over the relevant $\Thetatilde$.

For another example, thanks to the Hoeffding inequality,
\[
\CSHoeff := \left[ \overline{D}_n \pm \frac{(b-a)}{2} \frac{ \sqrt{2 \ln(2 / \alpha)}}{\sqrt{n}} \right]
\]
is non-asymptotically valid over
\( \Theta_{\mathrm{Hoeff}}
:= \left\{ \theta \in \Theta :
\mathrm{support}(\Distributionde{\theta})
\subseteq [a,b] \right\} \),
for some real numbers \( -\infty < a < b < +\infty\).

\medskip
    
In general, CSs based on concentration inequalities display the same suboptimal asymptotic behavior as $\CSalphannumero{BC}$.
In the case of \(\CSHoeff\), we remark that for every \(\alpha \in (0,1)\),
\(\sqrt{2\ln(2/\alpha)} > \QuantileGaussAt{1 - \alpha/2}\).
As \(\CSalphannumero{BC}\), $\CSHoeff$ thus cannot be exact asymptotically, even pointwise.

\medskip

The simple case of a scalar expectation illustrates the difficulty of constructing confidence sets with both non-asymptotic guarantees and asymptotic exactness on large subsets of the parameter space $\Theta$.
As \(\CSalphannumero{St}\) illustrates, non-asymptotic exactness can be achieved under parametric restrictions.
In contrast, provided known bounds on some moments, concentration inequalities yield non-asymptotic validity on large nonparametric subsets \(\Thetatilde\), but the resulting intervals are less precise: they are not exact, even asymptotically.



\begin{table}[htb]
    \vspace{0.2em}
    \renewcommand{\arraystretch}{1.3}
    \centering
    \begin{tabular}{l|lll}
        \textit{Property} & \textit{Asympt. point.} & \textit{Asympt. unif.} & \textit{Non-asympt.} \\
        \hline
        $\CSalphannumero{CLT}$
        & \textbf{exact} over $\Theta$
        & \textbf{exact} over $\ThetaKasy$
        & -\\
        $\CSalphannumero{St}$
        & \textbf{exact} over $\Theta$
        & \textbf{exact} over $\ThetaKasy$
        & \textbf{exact} over $\ThetaGauss$ \\
        $\CSalphannumero{BC}$
        & valid over $\ThetaBC$
        & valid over $\ThetaBC$
        & valid over $\ThetaBC$ \\
        $\CIsigmaknown$
        & \textbf{exact} over $\Thetasigmaknown$
        & \textbf{exact} over $\ThetaBEEsigmaknown$
        & valid over $\ThetaBEEsigmaknown$ \\
        $\CIsigmanhat$
        & \textbf{exact} over $\Theta$
        & \textbf{exact} over $\ThetaBEE$
        & valid over $\ThetaBEE$ \\
    \end{tabular}
    \caption{Quality measures of the considered confidence intervals over different parameter sets.
    Note that exact intervals are always valid, but the reverse implication does not hold.
    We remind the reader of the following notation:
    }
    
    \vspace{-1.7em}
    
    \begin{align*}
    \begin{array}{ll}
        \ThetaGauss := \left\{ \theta\in\Theta :
        \Distributionde{\theta} = \Normale(\theta_1, \sigmadeux), 
        \theta_1 \in \Rb, \sigmadeux \in \Rstarplus \right\},
        &
        \ThetaBC := \left\{\theta\in\Theta: \Varsoustheta(\xi)\leq M\right\}, \\
        \ThetaKasy \; := \left\{\theta\in\Theta:\Varsoustheta(\xi)\geq m,
        \Expsousthetan[\xi^4]\leq M\right\},
        &
        \ThetaBEE := \{ \theta \in \Theta : K_4(\theta) \leq K \}, \\
        \ThetaBEEsigmaknown \!
        := \{ \theta \in \Theta :
        \Varsoustheta(\xi) = \sigmaknown^2,
        K_4(\theta) \leq K \},
        & 
        \Thetasigmaknown \!
        := \{ \theta \in \Theta :
        \Varsoustheta(\xi) = \sigmaknown^2 \},
        \\
    \end{array}
    \end{align*}

    \begin{flushleft}
        and $\Theta$ is the set of all univariate distributions with finite variance.
    \end{flushleft}
        
    \label{tab:results_CI_expect}
\end{table}

\newpage

\section{Proof of results in Section~1}
\label{appendix:sec:proof_Section1}

\subsection{Proof of Proposition~\ref{prop:impossibility_NAVAE_Bernoulli}}
\label{proof:prop:impossibility_NAVAE_Bernoulli}

\begin{proof}
Fix \(\alpha \in (0,1)\).
We follow an indirect proof: we assume that there exists such a confidence set and show a contradiction.
For any sample $(\xi_1, \dots, \xi_n)$, \(n \in \INTST\), we denote that confidence set by $\CS(\xi_1, \dots, \xi_n)$.

\paragraph{First step.}
We derive a necessary condition for that confidence set to be non-asymptotically valid over~\(\mathfrak{Ber}\), which is the first part of being NAVAE.
Let \(n \in \INTST\), and define 
\begin{equation*}
    q_n := \max \Big( \big( (1 + \alpha) / 2 \big)^{1/n} \, , \, 
    \big( 1 - \alpha / 2 \big)^{1/n} \Big)
    , \qquad
    p_n := 1 - q_n,
\end{equation*}
and ${P}_{p_n}$  the Bernoulli distribution with parameter/expectation~\(p_n\).
Remember that we consider i.i.d observations and that, for brevity, for any parameter~\(\theta \in \Theta\), \(\mathbb{P}_{\theta}\) denotes the probability with respect to the joint distribution \(P_\theta^{\otimes n}\), where \(\Theta = (0, 1)\) in the current statistical model of Bernoulli distributions.
Consequently,
\(
    \mathbb{P}_{p_n}(\xi_1 = \dots = \xi_n = 0) = (1 - p_n)^n = q_n^n
\).

\medskip 

Imagine that $p_n \notin \CS(\underbrace{0, \dots, 0}_{\text{\(n\) zeros}})$. 
If so, 
\begin{align*}
    \mathbb{P}_{p_n}(p_n \notin \CS(\xi_1, \dots, \xi_n))
    & 
    \geq \mathbb{P}_{p_n}(\{p_n \notin \CS(\xi_1, \dots, \xi_n)\} \cap \{\xi_1 = \dots = \xi_n = 0\}) \\
    &
    = \mathbb{P}_{p_n}(\xi_1 = \dots = \xi_n = 0)
    = q_n^n \geq (1 + \alpha) / 2 > \alpha.
\end{align*}
Thus, the CS we consider, \(\CS(\xi_1, \dots, \xi_n)\), would not be non-asymptotically valid (NAV) over~\(\mathfrak{Ber}\).

\medskip 

Besides, the previous reasoning is valid for any~\(n \in \INTST\).
In other words, a necessary condition for \(\CS(\xi_1, \dots, \xi_n)\) to be NAV over~\(\mathfrak{Ber}\) is 
\begin{equation}
\label{eq:temp_nonexistence_navae_cs_over_Ber}
    \forall n \in \INTST,
    p_n \in \CS(\underbrace{0, \dots, 0}_{\text{\(n\) zeros}}),
\end{equation}
with \(p_n\) as just defined above.

\paragraph{Second step}
We show that \eqref{eq:temp_nonexistence_navae_cs_over_Ber} is in contradiction with \(\CS(\xi_1, \dots, \xi_n)\) being asymptotically exact uniformly over~\(\mathfrak{Ber}\), the second part of the NAVAE property. 

\medskip 

Let \(n \in \INTST\).
Because, by assumption, the CS we consider is NAVAE, it is in particular non-asymptotically valid.
From the first step, we thus have \(p_n \in \CS(0, \ldots, 0)\).
Therefore,
\begin{align*}
    \mathbb{P}_{p_n}(p_n \in \CS(\xi_1, \dots, \xi_n))
    & 
    \geq \mathbb{P}_{p_n}(\{p_n \in \CS(\xi_1, \dots, \xi_n)\} \cap \{\xi_1 = \dots = \xi_n = 0\}) \\
    &
    = \mathbb{P}_{p_n}(\xi_1 = \dots = \xi_n = 0)
    = q_n^n 
    \geq 1 - \frac{\alpha}{2}.
\end{align*}
Consequently, 
\begin{equation*}
    \big| 
    \mathbb{P}_{p_n}(p_n \in \CS(\xi_1, \dots, \xi_n))
    -
    (1 - \alpha)
    \big|
    \geq 
    \frac{\alpha}{2},
\end{equation*}
and, a fortiori,
\begin{equation}
\label{eq:temp2_nonexistence_navae_cs_over_Ber}
    \sup_{\theta \in \Theta}
    \big|
    \mathbb{P}_\theta(\theta \in \CS(\xi_1, \dots, \xi_n))
    -
    (1 - \alpha)
    \big|
    \geq 
    \frac{\alpha}{2}.
\end{equation}

\medskip 

Finally, remark that inequality~\eqref{eq:temp2_nonexistence_navae_cs_over_Ber} holds for any~\(n \in \INTST\).
Thus, it is impossible to have
\begin{equation*}
    \lim_{n \to +\infty}
    \sup_{\theta \in \Theta}
    \big|
    \mathbb{P}_\theta(\theta \in \CS(\xi_1, \dots, \xi_n))
    -
    (1 - \alpha)
    \big|
    = 0,
\end{equation*}
that is, \(\CS(\xi_1, \dots, \xi_n)\) is not asymptotically exact uniformly over~\(\mathfrak{Ber}\), hence not NAVAE over~\(\mathfrak{Ber}\), which yields the contradiction.

\end{proof}




\section{Proof of results in Section~\ref{sec:expectation_example}}
\label{appendix:sec:proofs_section3_expectation}

\subsection{Proof of Proposition~\ref{prop:mean_CI_BE-Edg_sigmaknown}}
\label{proof:prop:mean_CI_BE-Edg_sigmaknown}


\noindent
Let $n \geq 1$, $\alpha \in (0,1)$ and $\theta \in \ThetaBEEsigmaknown$. 
Assume that \(\delta_n \geq \alpha/2\). Then by definition,
\(\CIsigmaknown = \Rb\).
Therefore,
\begin{align*}
    \Probsousthetan \Big( \theta_1 \in \CIsigmaknown \Big) = 1 > 1 - \alpha,
\end{align*}
which finishes the proof in this case.
We now assume that \(\delta_n < \alpha/2\).
Therefore $\CIsigmaknown$ can be written as
\begin{equation*}
    \CIsigmaknown = \left[
    \overline{\xi}_n \pm \dfrac{\sigmaknown}{\sqrt{n}}
    \QuantileGaussAt{1 - \dfrac{\alpha}{2} + \delta_n}
    \right].
\end{equation*}
In this case, to show that
\begin{equation}
    \label{eq:nonasymp_edg_sigma_known}
    \Probsousthetan \Big( \theta_1 \in \CIsigmaknown \Big) \geq 1 - \alpha
\end{equation}
we first resort to Lemma~\ref{lem:edg_exp}(i) which ensures that for every $x>0$
\begin{align*}
    \Probsousthetan \Big( \sqrt{n}
    \big| \overline{\xi}_n - \theta_1 \big| / \sigmaknown > x \Big)
    \leq 2 \big\{ \Phi(-x) + \DeltanE \wedge \DeltanB \big\}
    \leq 2 \big\{ \Phi(-x) + \delta_n \big\}.
\end{align*}
Solving in $x$ the equation $2\big\{ \Phi(-x) + \delta_n \big\} = \alpha$ yields $x = \QuantileGaussAt{1 - \alpha/2 + \delta_n}$. Using this value of $x$ we obtain
\begin{align*}
    \Probsousthetan\left( \sqrt{n}\left| \overline{\xi}_n - \theta_1 \right|/\sigmaknown > \QuantileGaussAt{1 - \dfrac{\alpha}{2} + \delta_n} \right)
    \leq \alpha.
\end{align*}
This finishes the proof of~\eqref{eq:nonasymp_edg_sigma_known}. 

\bigskip

\subsection{Proof of Proposition~\ref{prop:mean_CI_BE-Edg_asymp_sigmaknown}}
\label{proof:prop:mean_CI_BE-Edg_asymp_sigmaknown}

\noindent
\textbf{Proof of the pointwise result for \(\CIsigmaknown\).} Our goal is to show
\begin{equation}
    \label{eq:pointwise_conv_sigma_known}
    \lim_{n\to+\infty}\left| \Probsousthetan\Big( \theta_1 \in \CIsigmaknown \Big) - (1-\alpha) \right| = 0 \quad \forall \theta\in\Thetasigmaknown.
\end{equation}

We consider a fixed but arbitrary $\theta\in\Thetasigmaknown$ for the rest of the proof. Since $\delta_n$ is deterministic and decreases to 0 by assumption, there exists $\underline{n}_\theta$ such that for every $n \geq \underline{n}_\theta$ the following holds almost surely
\begin{equation*}
    \CIsigmaknown = \left[
    \overline{\xi}_n \pm \dfrac{\sigmaknown}{\sqrt{n}}
    \QuantileGaussAtBig{1 - \dfrac{\alpha}{2} + \delta_n}
    \right].
\end{equation*}
As a result, Equation~\eqref{eq:pointwise_conv_sigma_known} is equivalent to
\begin{equation*}
    \lim_{n\to+\infty}\left| \Probsousthetan\left( \frac{\sqrt{n}\left| \overline{\xi}_n - \theta_1 \right|}{\sigmaknown} 
    \leq 
    \QuantileGaussAtBig{1 - \frac{\alpha}{2} + \delta_n} \right) - (1-\alpha) \right| = 0,
\end{equation*}
or
\begin{equation*}
    \lim_{n\to+\infty}\left| \Probsousthetan\left( \left|q_n\frac{\sqrt{n} (\overline{\xi}_n - \theta_1)}{\sigmaknown}\right|
    \leq 
    \QuantileGaussAtBig{1 - \frac{\alpha}{2}}\right) - (1-\alpha) \right| = 0,
\end{equation*}
with $q_n := \QuantileGaussAtBig{1 - \frac{\alpha}{2}} / \QuantileGaussAtBig{1 - \frac{\alpha}{2} +\delta_n}$.

\medskip

To conclude it is enough to show that
\begin{equation}
    \label{eq:weak_conv_mean_known_variance}
    \left| q_n\frac{\sqrt{n} (\overline{\xi}_n - \theta_1)}{\sigmaknown} \right| \convD |U|, \, \text{with } U\sim\mathcal{N}(0,1).
\end{equation}
By continuity of $\QuantileGauss$ and the definition of $\delta_n$, $q_n$ tends to 1 deterministically. By the CLT, $\sqrt{n} (\overline{\xi}_n - \theta_1)/\sigmaknown$ tends to the $\mathcal{N}(0,1)$ distribution. By Slutsky's theorem, $q_n\sqrt{n} (\overline{\xi}_n - \theta_1)/\sigmaknown$ tends to the $\mathcal{N}(0,1)$ distribution as well. The absolute value function is continuous on $\Rb$, hence by the continuous mapping in distribution we can conclude that Equation~\eqref{eq:weak_conv_mean_known_variance} is valid.

\bigskip

\noindent
\textbf{Proof of the uniform result for \(\CIsigmaknown\).}
Our goal is to show
\begin{equation}
    \label{eq:unif_conv_sigma_known}
    \lim_{n\to+\infty}\sup_{\theta\in\ThetaBEEsigmaknown}\left| \Probsousthetan\Big( \theta_1 \in \CIsigmaknown \Big) - (1-\alpha) \right| = 0.
\end{equation}
By assumption, $\delta_n \to 0$, and therefore
for every $\alpha\in(0,1)$, $\delta_n < \alpha / 2$ for $n$ large enough.
For such values of $n$,
$\CIsigmaknown$ takes the form
\begin{equation*}
    \CIsigmaknown = \left[
    \overline{\xi}_n \pm \dfrac{\sigmaknown}{\sqrt{n}}
    \QuantileGaussAt{1 - \dfrac{\alpha}{2} + \delta_n}
    \right].
\end{equation*}
Therefore, for every $\theta \in \ThetaBEEsigmaknown$,
\begin{equation*}
    \Probsousthetan\Big( \theta_1 \in \CIsigmaknown \Big) = \Probsousthetan\left( \sqrt{n}\left| \overline{\xi}_n - \theta_1 \right|/\sigmaknown \leq \QuantileGaussAt{1 - \dfrac{\alpha}{2} + \delta_n} \right).
\end{equation*}
It is thus sufficient to prove the following
\begin{equation*}
    \lim_{n\to+\infty}\sup_{\theta\in\ThetaBEEsigmaknown}\left| \Probsousthetan\left( \sqrt{n}\left| \overline{\xi}_n - \theta_1 \right|/\sigmaknown \leq \QuantileGaussAt{1 - \dfrac{\alpha}{2} + \delta_n} \right) - (1-\alpha) \right| = 0,    
\end{equation*}
or equivalently
\begin{equation}
    \label{eq:unif_conv_sigma_known_bis}
    \lim_{n\to+\infty}\sup_{\theta\in\ThetaBEEsigmaknown}
    \left| \Probsousthetan\left( \sqrt{n}\left| \overline{\xi}_n - \theta_1 \right|/\sigmaknown > \QuantileGaussAt{1 - \dfrac{\alpha}{2} + \delta_n} \right) - \alpha \right| = 0.    
\end{equation}

We now apply Lemma~\ref{lemma:bound_cdf_alpha_Phi}
with $A = \sqrt{n}\left(\overline{\xi}_n - \theta_1 \right) / \sigmaknown$, $\lambda = \lambdatroisn$
and $x = \QuantileGaussAt{1 - \dfrac{\alpha}{2} + \delta_n}$.
Remark that $\|F - \Phi \|_\infty = \DeltanB$ and $\|F - \Phi - \Edg\| = \DeltanE$ in that case, where $F$ is the distribution of~\(A\).
Therefore, for any $\theta \in \ThetaBEEsigmaknown$,
\begin{align*}
    &\left| \Probsousthetan\left( \sqrt{n}\left| \overline{\xi}_n - \theta_1 \right|/\sigmaknown > \QuantileGaussAt{1 - \dfrac{\alpha}{2} + \delta_n} \right) - \alpha \right| \\
    &\hspace{10em} \leq 2 \DeltanB \wedge \DeltanE
    + \left| 2 \Phi\left(-\QuantileGaussAt{1 - \dfrac{\alpha}{2} + \delta_n} \right) - \alpha \right| \\
    &\hspace{10em} \leq 2 \delta_n + \left| 2 \Phi\left(-\QuantileGaussAt{1 - \dfrac{\alpha}{2} + \delta_n} \right) - \alpha \right|.
\end{align*}
As $n \to +\infty$, the first term in the previous bound -- i.e. $\delta_n$ -- tends to $0$ uniformly in $\theta \in \ThetaBEEsigmaknown$ by definition of $\delta_n$.
The second term also tends to $0$ by continuity of $\Phi$ and $\QuantileGauss$, and because $\delta_n \to 0$ by assumption.
This concludes the proof of \eqref{eq:unif_conv_sigma_known_bis}.

\subsection{Proof of Proposition~\ref{prop:open_interval_a}}
\label{proof:prop:open_interval_a}

\begin{proof}
We start by proving the first part of (i).
Let
\begin{align*}
    g_{n,K}(a)
    &:= 1 - \dfrac{\alpha}{2} + \delta_n + \dfrac{\nuVar}{2} - \Phi(\sqrt{n/a}) \\
    &= 1 - \dfrac{\alpha}{2} + \delta_n
    + \dfrac{\exp\!\Big( -\frac{n(1-1/a)^2}
    {2K} \Big)}{2}
    - \int_{-\infty}^{\sqrt{n/a}}
    \frac{1}{\sqrt{2\pi}} e^{-x^2/2} dx.
\end{align*}
For $a \to 1$, we get
\begin{align*}
    \lim_{a \to 1} g_{n,K}(a)
    &= 1 - \dfrac{\alpha}{2} + \delta_n
    + \dfrac{1}{2}
    - \int_{-\infty}^{\sqrt{n}}
    \frac{1}{\sqrt{2\pi}} e^{-x^2/2} dx
    \geq \dfrac{1 - \alpha}{2} + \delta_n > 0.
\end{align*}
For $a \to +\infty$, we get
\begin{align*}
    \lim_{a \to +\infty} g_{n,K}(a)
    &= 1 - \dfrac{\alpha}{2} + \delta_n
    + \dfrac{e^{-n/2K}}{2}
    - \int_{-\infty}^{0} \frac{1}{\sqrt{2\pi}} e^{-x^2/2} dx \\
    &= \dfrac{1 - \alpha}{2} + \delta_n
    + \dfrac{e^{-n/2K}}{2} > 0.
\end{align*}
Since the limit at the two sides of $(1, +\infty)$ are both positive, we now study its derivative.
\begin{align*}
    g_{n,K}'(a)
    &= \dfrac{-n}{4K} \dfrac{d((1-1/a)^2)}{da}
    \exp\!\left( -\frac{n(1-1/a)^2}{2K} \right)
    - \frac{1}{\sqrt{2\pi}}
    \frac{d(\sqrt{n/a})}{da}
    e^{-n/2a} \\
    &= \dfrac{-n}{4K} \times 2(1-1/a)(1/a^2) \times
    \exp\!\left( -\frac{n(1-1/a)^2}{2K} \right)
    - \frac{1}{\sqrt{2\pi}} \times
    \frac{\sqrt{n}}{2} \times
    \frac{-1}{a \sqrt{a}} \times
    e^{-n/2a} \\
    &= \dfrac{-n}{2K} \times \frac{1 - 1/a}{a^2} \times
    \exp\!\left( -\frac{n(1-1/a)^2}{2K} \right)
    + \frac{1}{\sqrt{2\pi}} \times
    \frac{\sqrt{n}}{2} \times
    \frac{1}{a \sqrt{a}} \times
    e^{-n/2a} \\
    &= \dfrac{-n}{2K} \times
    \frac{a - 1}{a^3} \times
    \exp\!\left( -\frac{n(1-1/a)^2}{2K} \right)
    + \sqrt{\frac{n}{8\pi}}
    \frac{1}{a \sqrt{a}} \times
    e^{-n/2a}.
\end{align*}
So, we have
\begin{align*}
    \lim_{a \to 1} g_{n,K}'(a)
    = 0 + \sqrt{\frac{n}{8\pi}} e^{-n/2} > 0
\end{align*}
and
\begin{align*}
    \lim_{a \to +\infty} g_{n,K}'(a)
    = \dfrac{-n}{2K} \times
    \frac{1 - 0}{+ \infty} \times
    \exp\!\left( -\frac{n(1-0)^2}{2K} \right)
    + \sqrt{\frac{n}{8\pi}}
    \frac{1}{+\infty} \times
    e^{0}
    = 0
\end{align*}

\begin{lemma}
    For any $n, K > 0$, the set of $a$ such that $g_{n,K}'(a) \leq 0$ is either empty or a closed interval of $(0, +\infty)$.
\label{lemma:unique-solution_gprime}
\end{lemma}

We now use Lemma~\ref{lemma:unique-solution_gprime} and distinguish both cases:

\begin{itemize}
    \item In the first case, for every $a \in (0, +\infty)$, $g_{n,K}'(a) > 0$.
    Therefore, $g_{n,K}$ is increasing and 
    $\{a \in (0, +\infty): \, g_{n,K}(a) < 0\} = \emptyset$.

    \item In the second case, there exists $(a_-, a_+)$ such that
    $g_{n,K}$ is increasing on $(1, a_-)$, decreasing on $(a_-, a_+)$,
    and then increasing on $(a_+, +\infty)$.
    
    This means that for all $a \in (1, a_
    -)$, $g_{n,K}(a) > \lim_{a \to 1} g_{n,K}(a) > 0$.

    Therefore, if $g_{n,K}$ ever takes negative values, it must do so on the interval $(a_-, +\infty)$, and since on this interval $g_{n,K}$ is decreasing, then increasing, we know that the set $\{a \in (0, +\infty): \, g_{n,K}(a) < 0\}$ is an open interval (potentially empty).
\end{itemize}

This finishes the proof of the first statement of Proposition~\ref{prop:open_interval_a}(i).
Note that
$1 - \dfrac{\alpha}{2} + \delta_n + \dfrac{\nuVar(a)}{2}$
is decreasing in $n$ and that
$\Phi(\sqrt{n/a})$ is increasing in $n$,
so the set of values of $a_n$ that satisfies the constraint~\eqref{eq:constraint_a_n} is increasing.
Furthermore, for every fixed $a$, we have
\begin{align*}
    1 - \dfrac{\alpha}{2} + \delta_n + \dfrac{\nuVar(a)}{2}
    < \Phi(\sqrt{n/a}),
\end{align*}
for $n$ large enough.
This shows that every value $a \in (1, +\infty)$ eventually satisfies the constraint~\eqref{eq:constraint_a_n}, finishing the proof of (i).

\medskip

We now prove (ii).
We have
\begin{align*}
    \nuVar(a_n)
    &= \exp\!\left( - \frac{n(1-1/a_n)^2}{2K} \right) \\
    &= \exp\!\left( - \frac{n(1-1/(1 + b_n))^2}{2K} \right) \\
    &= \exp\!\left( - \frac{n b_n^2}{2K (1 + b_n)^2} \right)
    \to 0,
\end{align*}
as $n \to \infty$, since by assumption
\(b_n \sqrt{n} \to +\infty\).
On the other side, we have
\begin{align*}
    \Phi(\sqrt{n/a_n}) = \Phi(\sqrt{n/(1 + b_n}) \to 1.
\end{align*}
Therefore for $n$ large enough, the constraint \eqref{eq:constraint_a_n} is satisfied, since the left-hand side tends to a limit strictly smaller than $1$ and the right-hand side tends to $1$.

\end{proof}

\begin{proof}[Proof of Lemma~\ref{lemma:unique-solution_gprime}]
We have $g_{n,K}'(a) \geq 0$ if and only if
\begin{align*}
    \dfrac{n}{2K} \times
    \frac{a - 1}{a^3} \times
    \exp\!\left( -\frac{n(1-1/a)^2}{2K} \right)
    \leq \sqrt{\frac{n}{8\pi}}
    \frac{1}{a \sqrt{a}} \times
    e^{-n/2a}
\end{align*}
if and only if
\begin{align*}
    \frac{a - 1}{a \sqrt{a}} \times
    \exp\!\left( -\frac{n(1-1/a)^2}{2K}
    + \frac{n}{2a}\right)
    \leq \sqrt{\frac{K^2}{2\pi n}}
\end{align*}
if and only if
\begin{align*}
    \frac{a - 1}{a \sqrt{a}} \times
    \exp\!\left( \frac{-n}{2K}
    \bigg(1 - \frac{2}{a} + \frac{1}{a^2} - \frac{K}{a} \bigg) \right)
    \leq \sqrt{\frac{K^2}{2\pi n}}
\end{align*}
if and only if
\begin{align*}
    \frac{a - 1}{a \sqrt{a}} \times
    \exp\!\left( \frac{-n}{2K}
    \bigg(1 - \frac{K + 2}{a} + \frac{1}{a^2} \bigg) \right)
    \leq \sqrt{\frac{K^2}{2\pi n}}
\end{align*}
if and only if
\begin{align*}
    \frac{a - 1}{a \sqrt{a}} \times
    \exp\!\left( \frac{-n}{2K}
    \bigg( \Big(\frac{1}{a} - \frac{K + 2}{2}\Big)^2
    + \frac{4 - (K+2)^2}{4} \bigg) \right)
    \leq \sqrt{\frac{K^2}{2\pi n}}
\end{align*}
if and only if
\begin{align*}
    \frac{a - 1}{a \sqrt{a}} \times
    \exp\!\left( \frac{-n}{2K}
    \bigg( \Big(\frac{1}{a} - \frac{K + 2}{2}\Big)^2
    + \frac{- K^2 - 4K}{4} \bigg) \right)
    \leq \sqrt{\frac{K^2}{2\pi n}}
\end{align*}
if and only if
\begin{align*}
    \frac{a - 1}{a \sqrt{a}} \times
    \exp\!\left( \frac{-n}{2K}
    \times \Big(\frac{1}{a} - \frac{K + 2}{2}\Big)^2
    + \frac{n(K + 4)}{8} \right)
    \leq \sqrt{\frac{K^2}{2\pi n}}
\end{align*}
if and only if
\begin{align*}
    \frac{a - 1}{a \sqrt{a}} \times
    \exp\!\left( \frac{-n}{2K}
    \times \Big(\frac{1}{a} - \frac{K + 2}{2}\Big)^2 \right)
    \leq \sqrt{\frac{K^2}{2\pi n}} e^{-n(K + 4) / 8}
\end{align*}
Let $x = 1/a \in (0, 1)$.
So
\begin{align*}
    \frac{a - 1}{a \sqrt{a}}
    = x \sqrt{x} (1/x - 1)
    = \sqrt{x} (1 - x)
    = \sqrt{x} - x \sqrt{x}
\end{align*}
We can rewrite the previous inequality as 
\begin{align*}
    h_{n,K}(x)
    \leq \sqrt{\frac{K^2}{2\pi n}} e^{-n(K + 4) / 8},
\end{align*}
where
\begin{align*}
    h_{n,K}(x)
    := \sqrt{x} (1 - x) \times
    \exp\!\left( \frac{-n}{2K}
    \times \Big(x - \frac{K + 2}{2}\Big)^2 
    \right)
\end{align*}
Note that
\begin{align*}
    h_{n,K}(0) = h_{n,K}(1) = 0.
\end{align*}
We have
\begin{align*}
    h_{n,K}'(x)
    &= \frac{1}{2 \sqrt{x}} (1 - x) \times
    \exp\!\left( \frac{-n}{2K}
    \times \Big(x - \frac{K + 2}{2}\Big)^2 
    \right) \\
    &- \sqrt{x} \times
    \exp\!\left( \frac{-n}{2K}
    \times \Big(x - \frac{K + 2}{2}\Big)^2 
    \right) \\
    &+ \sqrt{x} (1 - x) \times
    \frac{-n}{2K}
    (2x - (K + 2))
    \exp\!\left( \frac{-n}{2K}
    \times \Big(x - \frac{K + 2}{2}\Big)^2 
    \right) \\
    &= \left(\frac{1}{2 \sqrt{x}} (1 - x) - \sqrt{x}
    + \sqrt{x} (1 - x) \times
    \frac{-n}{2K}
    (2x - (K + 2))
    \right)
    \exp\!\left( \frac{-n}{2K}
    \times \Big(x - \frac{K + 2}{2}\Big)^2 
    \right) \\
    &= \frac{1 - x - 2x
    + 2x(1 - x) \frac{-n}{2K} (2x - (K + 2))}{2 \sqrt{x}}
    \exp\!\left( \frac{-n}{2K}
    \times \Big(x - \frac{K + 2}{2}\Big)^2 
    \right) \\
    &= \frac{2K - 6Kx
    - 2n x(1 - x) (2x - (K + 2))}{2 \sqrt{x}}
    \exp\!\left( \frac{-n}{2K}
    \times \Big(x - \frac{K + 2}{2}\Big)^2 
    \right).
\end{align*}
Let us define the polynomial $p$ by
$p(x) := 2K - 6Kx
- 2n x(1 - x) (2x - (K + 2))$.
\begin{lemma}
    There exists $x^*$ such that for all $x \in (0, x^*)$, $p(x) > 0$ and for all $x \in (x^*, 1)$, $p(x) < 0$.
\label{lemma:study_p}
\end{lemma}

Therefore,
for all $x \in (0, x^*)$,
$h_{n,K}'(x) > 0$ and
for all $x \in (x^*, 1)$,
$h_{n,K}'(x) < 0$.
Therefore,
$h_{n,K}$ is increasing on the interval
$(0, x^*)$ and decreasing on the interval $(x^*, 1)$.

This shows that the set of $x$ such that 
\begin{align*}
    h_{n,K}(x)
    \geq \sqrt{\frac{K^2}{2\pi n}} e^{-n(K + 4) / 8},
\end{align*}
is either empty or a closed interval (which does not contain the boundary values $0$ and $1$).
This finishes the proof as claimed.
\end{proof}

\begin{proof}[Proof of Lemma~\ref{lemma:study_p}]

Remember that 
\begin{align*}
    p(x) := 2K - 6Kx
    - 2n x(1 - x) (2x - (K + 2)).
\end{align*}
So $p(0) = 2K$
and $p(1) = 2K - 6K = -4K < 0$.
We know that the dominant coefficient of $x$ in $p$ is positive, so
$\lim_{x \to -\infty} p(x) = -\infty$
and $\lim_{x \to +\infty} p(x) = +\infty$.

So, by the intermediate value theorem, $p$ has (at least) a root between $- \infty$ and $0$;
$p$ has also (at least) a root between $1$ and $=\infty$.
Since $p$ is a polynomial of order 3, we know that $p$ has at most $3$ roots, and combining the previous statement with the intermediate value theorem, we get that it has exactly one root in the interval $(0, 1)$.

We have therefore shown that there exists $x^*$ such that for all $x \in (0, x^*)$, $p(x) > 0$ and for all $x \in (x^*, 1)$, $p(x) < 0$.
\end{proof}

\subsection{Proof of Proposition~\ref{prop:mean_CI_BE-Edg_sigmaunknown}}
\label{proof:prop:mean_CI_BE-Edg_sigmaunknown}

Let $n \geq 1$, $\alpha \in (0,1)$ and $\theta \in \ThetaBEE$. In what follows, we write $a$ instead of $a_n$. 
Arguing as for \(\CIsigmaknown\), we remark that when
$1 - \dfrac{\alpha}{2} + \delta_n + \dfrac{\nuVar}{2} \geq \Phi(\sqrt{n/a})$,
\begin{align*}
    \Probsousthetan \Big( \theta_1 \in \CIsigmanhat \Big) = 1 > 1 - \alpha.
\end{align*}
We now assume that $1 - \dfrac{\alpha}{2} + \delta_n + \dfrac{\nuVar}{2} < \Phi(\sqrt{n/a})$. We then have
\begin{align*}
    \CIsigmanhat = \left[ \overline{\xi}_n
    \pm \dfrac{\sigmanhat}{\sqrt{n}} \Cn
    \QuantileGaussAt{1 - \dfrac{\alpha}{2} + \delta_n + \dfrac{\nuVar}{2}}
    \right].
\end{align*}
In this case, to show that
\begin{equation}
\label{eq:nonasymp_edg_sigma_unknown_1}
    \Probsousthetan\Big( \theta_1 \in \CIsigmanhat \Big) \geq 1 - \alpha
\end{equation}
we wish to prove that for $x^* := C_n\QuantileGaussAt{1 - \dfrac{\alpha}{2} + \delta_n + \dfrac{\nuVar}{2}}$,
\begin{equation*}
    \Probsousthetan\left( \sqrt{n}\left| \overline{\xi}_n - \theta_1 \right| > \sigmanhat x^* \right) \leq \alpha.
\end{equation*}
By definition (and using the fact that $\sigma(\theta)>0$ for every $\theta \in \ThetaBEE$), we have almost surely 
\begin{equation*}
    \frac{\sigmanhat}{\sigma(\theta)} = \sqrt{\frac{\frac{1}{n}\sum_{i=1}^n(\xi_i-\theta_1)^2}{\sigma(\theta)^2}-\frac{1}{n}\left(\frac{\sqrt{n}(\MeanXi-\theta_1)}{\sigma(\theta)}\right)^2}.
\end{equation*}
Letting $T_n := \sqrt{n}|\MeanXi-\theta_1|/\sigma(\theta)$ and $\sigmanhat_0 := \sqrt{\frac{1}{n}\sum_{i=1}^n(\xi_i-\theta_1)^2}$, we can thus write for every $x>0$
\begin{equation}
\label{eq:nonasymp_edg_sigma_unknown_2}
    \Probsousthetan\left( \sqrt{n}\left| \overline{\xi}_n - \theta_1 \right| > \sigmanhat x \right) = \Probsousthetan\left( T_n > x \sqrt{\frac{\sigmanhat^2_0}{\sigma(\theta)^2} - \frac{T_n^2}{n}} \right).
\end{equation}
We now need to control $T_n$ from above and $\sigmanhat^2_0/\sigma(\theta)^2$ from below. We use Theorem 2.19 in~\cite{pena2008self} which allows us to write for every $a > 1$
\begin{equation}
\label{eq:nonasymp_edg_sigma_unknown_3}
    \Prob\left( \frac{\widehat{\sigma}_0^2}{\sigma(\theta)^2} < \frac{1}{a} \right)
    \leq \exp\left( -\frac{n(1-1/a)^2}
    {2\majorantKurtxi} \right).
\end{equation}
We also resort to Lemma~\ref{lem:edg_exp}(i) which ensures that for every $y>0$,
\begin{align}
\label{eq:nonasymp_edg_sigma_unknown_4}
    \Probsousthetan\left( T_n > y \right)
        & \leq 2\Big\{ \Phi\big(-y\big) + \DeltanE \wedge \DeltanB \Big\}.
\end{align}
Combining~\eqref{eq:nonasymp_edg_sigma_unknown_2},~\eqref{eq:nonasymp_edg_sigma_unknown_3} and~\eqref{eq:nonasymp_edg_sigma_unknown_4}, we remark that for every $x>0$ and every $a>1$ and $y>0$ such that $1/a - y^2/n > 0$
\begin{align*}
    \Probsousthetan\left( \sqrt{n}\left| \overline{\xi}_n - \theta_1 \right| > \sigmanhat x \right) & \leq \Probsousthetan\left( \left\{ T_n > x\sqrt{\frac{\sigmanhat^2_0}{\sigma(\theta)^2} - \frac{T_n^2}{n}} \right\} \cap \left\{ T_n \leq y \right\} \cap \left\{ \frac{\widehat{\sigma}_0^2}{\sigma(\theta)^2} \geq \frac{1}{a} \right\} \right) \\
    & + \Probsousthetan\left( T_n > y \right) + \Probsousthetan\left( \frac{\widehat{\sigma}_0^2}{\sigma(\theta)^2} < \frac{1}{a}\right) \\
    & \leq \Probsousthetan\left( y > x\sqrt{\frac{1}{a} - \frac{y^2}{n}} \right) + 2\Big\{ \Phi\big(-y\big) + \DeltanE \wedge \DeltanB \Big\} + \exp\left( -\frac{n(1-1/a)^2}
    {2\majorantKurtxi} \right) \\
    &\leq \Probsousthetan\left( y > x\sqrt{\frac{1}{a} - \frac{y^2}{n}} \right) + 2\Big\{ \Phi\big(-y\big) + \delta_n \Big\} + \exp\left( -\frac{n(1-1/a)^2}
    {2\majorantKurtxi} \right).
\end{align*}
We can pick $x = y/\sqrt{\frac{1}{a} - \frac{y^2}{n}}$, which makes the probability on the right-hand side of the last equation equal to 0. There remains to find $y>0$ such that
\begin{equation*}
    2\Big\{ \Phi\big(-y\big) + \delta_n \Big\} + \exp\left( -\frac{n(1-1/a)^2}
    {2\majorantKurtxi} \right) = \alpha.
\end{equation*}
We obtain
\begin{equation*}
    y^* = \QuantileGaussAt{1 - \alpha/2 + \delta_n + \nuVar/2}.
\end{equation*}
We remark that under the maintained condition $1 - \dfrac{\alpha}{2} + \delta_n + \dfrac{\nuVar}{2} < \Phi(\sqrt{n/a})$, $y^*$ is well-defined and satisfies $1/a - (y^*)^2/n > 0$. Setting $x^* := y^*/\sqrt{\frac{1}{a} - \frac{(y^*)^2}{n}}$, we conclude
\begin{equation*}
    \Probsousthetan\left( \sqrt{n}\left| \overline{\xi}_n - \theta_1 \right| > \sigmanhat x^* \right) \leq \alpha
\end{equation*}
which is what we want. $\Box$

\subsection{Proof of Proposition~\ref{prop:mean_CI_BE-Edg_asymp_sigmaunknown}}
\label{proof:prop:mean_CI_BE-Edg_asymp_sigmaunknown}

\noindent
\textbf{Proof of the pointwise result for \(\CIsigmanhat\).} Our goal is to show
\begin{equation}
    \label{eq:pointwise_conv_sigma_hat}
    \lim_{n\to+\infty}\left| \Probsousthetan\Big( \theta_1 \in \CIsigmanhat \Big) - (1-\alpha) \right| = 0 \quad \forall \theta\in\Theta.
\end{equation}

We consider a fixed but arbitrary $\theta\in\Theta$ for the rest of the proof. Given the assumptions on $\delta_n$ and $a_n = 1 + b_n$, $\delta_n + \nuVar/2$ decreases to 0 deterministically.
As a result, there exists $\underline{n}_\theta$ such that for every $n \geq \underline{n}_\theta$ the following holds almost surely
\begin{equation*}
    \CIsigmanhat = \left[ \overline{\xi}_n \pm
        \dfrac{\sigmanhat}{\sqrt{n}} \Cn
        \QuantileGaussAt{1 - \dfrac{\alpha}{2} + \delta_n +  \dfrac{\nuVar}{2}}
        \right].
\end{equation*}

Given this observation, we can write that Equation~\eqref{eq:pointwise_conv_sigma_hat} is equivalent to
\begin{equation*}
    \lim_{n\to+\infty}\left| \Probsousthetan\left( \frac{\sqrt{n}\left| \overline{\xi}_n - \theta_1 \right|}{\sigmanhat} 
    \leq 
    C_n\QuantileGaussAtBig{1 - \frac{\alpha}{2} + \delta_n + \dfrac{\nuVar}{2} } \right) - (1-\alpha) \right| = 0,
\end{equation*}
or
\begin{equation*}
    \lim_{n\to+\infty}\left| \Probsousthetan\left( \left|q_n\frac{\sqrt{n} (\overline{\xi}_n - \theta_1)}{\sigmanhat}\right|
    \leq 
    \QuantileGaussAtBig{1 - \frac{\alpha}{2}}\right) - (1-\alpha) \right| = 0,
\end{equation*}
with $q_n := \QuantileGaussAtBig{1 - \frac{\alpha}{2}} / \Big(C_n\QuantileGaussAtBig{1 - \frac{\alpha}{2} +\delta_n + \dfrac{\nuVar}{2}}\Big)$.

\medskip

The rest of the proof is straightforward combining Slutsky's theorem, $\sigma(\theta)^2>0$, the CLT and the continuous mapping in distribution applied to the absolute value function (as done for \(\CIsigmaknown\)).

\bigskip

\noindent
\textbf{Proof of the uniform result for \(\CIsigmanhat\).}
Given the conditions imposed on $\delta_n$ and $a_n$, the condition
\begin{equation*}
    1 - \dfrac{\alpha}{2} + \delta_n + \dfrac{\nuVar}{2} < \Phi(\sqrt{n/a_n})
\end{equation*}
is satisfied for every $n \geq n_0$, for some $n_0$ that is valid for every $\theta\in\ThetaBEE$. As a result, we can claim that for every $\alpha\in(0,1)$ and every $n$ large enough, uniformly in $\theta\in\ThetaBEE$,
\begin{equation*}
    \Probsousthetan\Big( \theta_1 \in \CIsigmanhat \Big) = \Probsousthetan\left( \sqrt{n}\left| \overline{\xi}_n - \theta_1 \right| \leq \sigmanhat C_n\QuantileGaussAtBig{1 - \dfrac{\alpha}{2} + \delta_n + \dfrac{\nuVar}{2}} \right).
\end{equation*}
Proposition~\ref{prop:mean_CI_BE-Edg_sigmaunknown} also yields that for every $\alpha \in (0,1)$ and $n \geq n_0$ (uniformly in $\theta \in \ThetaBEE$)
\begin{equation*}
    \Probsousthetan\left( \sqrt{n}\left| \overline{\xi}_n - \theta_1 \right| \leq \sigmanhat C_n\QuantileGaussAtBig{1 - \dfrac{\alpha}{2} + \delta_n + \dfrac{\nuVar}{2}} \right) \geq 1-\alpha.
\end{equation*}

It is thus sufficient to prove the following
\begin{equation}
    \label{eq:unif_conv_sigma_unknown}
    \limsup_{n\to+\infty}\sup_{\theta\in\ThetaBEE} \Probsousthetan\left( \sqrt{n}\left| \overline{\xi}_n - \theta_1 \right|
    \leq \sigmanhat x_n \right) \leq 1 - \alpha,
\end{equation}
where $x_n := C_n\QuantileGaussAtBig{1 - \dfrac{\alpha}{2}
+ \delta_n + \dfrac{\nuVar}{2}}$. Let $\sigmanhat_0 := \sqrt{\frac{1}{n}\sum_{i=1}^n(\xi_i-\theta_1)^2}$. We can write
\begin{align}
    \label{eq:prob_up_bound_asymp_unif_exact_1}
    \Probsousthetan\left( \sqrt{n}\left| \overline{\xi}_n - \theta_1 \right|
    \leq \sigmanhat x_n \right) & \leq \Probsousthetan\left( \underbrace{ \left\{\sqrt{n}\left| \overline{\xi}_n - \theta_1 \right|
    \leq \sigmanhat x_n \right\} \cap \left\{ \left|\frac{\sigmazeronhat^2}{\sigma(\theta)^2} - 1\right| \leq \frac{\sqrt{K-1}}{n^{1/4}} \right\} }_{=: \mathcal{A}} \right) \nonumber \\
    & + \Probsousthetan\left( \left|\frac{\sigmazeronhat^2}{\sigma(\theta)^2} - 1\right| > \frac{\sqrt{K-1}}{n^{1/4}} \right).
\end{align}
As soon as $n > (K-1)^2$, $\sqrt{K-1}/n^{1/4} <1$, which implies that $\sigmazeronhat^2/\sigma^2$ and then $\sigmazeronhat^2$ is strictly positive on $\mathcal{A}$. As a result, $T_n := \sqrt{n}\left| \overline{\xi}_n - \theta_1 \right|/\sigmazeronhat$ is well-defined on $\mathcal{A}$ and Lemma~\ref{lem:deterministic_equality} ensures on the same event that
\begin{align*}
    \left\{ \sqrt{n}\left| \overline{\xi}_n - \theta_1 \right| \leq \sigmanhat x_n\right\} & = \left\{ \sqrt{n}\left| \overline{\xi}_n - \theta_1 \right|/\sigmazeronhat \leq (\sigmanhat/\sigmazeronhat) x_n\right\} \\
    & = \left\{ \left| T_n \right| \leq \sqrt{1-\frac{T_n^2}{n}} x_n\right\} \\
    & = \left\{ |T_n| \leq x_n\left(1+\frac{x_n^2}{n}\right)^{-1/2} \right\} \\
    & = \left\{ \sqrt{n}\left| \overline{\xi}_n - \theta_1 \right| / \sigma(\theta) \leq \frac{\sigmazeronhat}{\sigma(\theta)}x_n\left(1+\frac{x_n^2}{n}\right)^{-1/2} \right\}.
\end{align*}
Using the fact that $\sigmazeronhat^2/\sigma(\theta)^2 \leq 1 + \sqrt{K-1}/n^{1/4}$ on $\mathcal{A}$ together with~\eqref{eq:prob_up_bound_asymp_unif_exact_1}, we can write
\begin{align}
    \label{eq:prob_up_bound_asymp_unif_exact_2}
    \Probsousthetan\left( \sqrt{n}\left| \overline{\xi}_n - \theta_1 \right|
    \leq \sigmanhat x_n \right) & \leq \Probsousthetan\left( \sqrt{n}\left| \overline{\xi}_n - \theta_1 \right| / \sigma(\theta) \leq \left(1 + \frac{\sqrt{K-1}}{n^{1/4}} \right) x_n/\sqrt{(1+x_n^2/n)} \right) \nonumber \\
    & + \Probsousthetan\left( \left|\frac{\sigmazeronhat^2}{\sigma(\theta)^2} - 1\right| > \frac{\sqrt{K-1}}{n^{1/4}} \right).
\end{align}
Let $\widetilde{x}_n := \left(1 + \sqrt{K-1}/n^{1/4} \right) x_n/\sqrt{(1+x_n^2/n)}$. Application of Lemma~\ref{lemma:bound_cdf_alpha_Phi} with $A := \sqrt{n}(\overline{\xi}_n - \theta_1)$ and $\lambda := \E[(\xi-\theta_1)^3]/\sigma(\theta)^3$ ensures
\begin{align}
    \label{eq:be_edg_ineq}
    \left| \Probsousthetan\left( \sqrt{n}\left| \overline{\xi}_n - \theta_1 \right| / \sigma(\theta) \leq \widetilde{x}_n \right) - (1-\alpha) \right| & \leq 2 \left\{ \DeltanE \wedge \DeltanB \right\} + \left| 2\Phi(-\widetilde{x}_n) - \alpha \right| \nonumber \\
    & \leq 2\delta_n + \left| 2\Phi(-\widetilde{x}_n) - \alpha \right|.
\end{align}
Combining~\eqref{eq:prob_up_bound_asymp_unif_exact_2} and~\eqref{eq:be_edg_ineq}, and resorting to Markov's inequality, we arrive at
\begin{align*}
    \Probsousthetan\left( \sqrt{n}\left| \overline{\xi}_n - \theta_1 \right|
    \leq \sigmanhat x_n \right) & \leq 1 - \alpha + 2 \delta_n + \left| 2\Phi(-\widetilde{x}_n) - \alpha \right| + \Probsousthetan\left( \left|\frac{\sigmazeronhat^2}{\sigma(\theta)^2} - 1\right| > \frac{\sqrt{K-1}}{n^{1/4}} \right) \\
    & \leq 1 - \alpha + 2 \delta_n + \left| 2\Phi(-\widetilde{x}_n) - \alpha \right| + \frac{\sqrt{n}}{K-1}\times \frac{\Varsoustheta\left[(\xi-\theta_1)^2\right]}{\sigma(\theta)^4n} \\
    & \leq 1 - \alpha + 2 \delta_n + \left| 2\Phi(-\widetilde{x}_n) - \alpha \right| + n^{-1/2}.
\end{align*}
By definition of $\widetilde{x}_n$ and continuity of $\Phi(\cdot)$, the term $\left| 2\Phi(-\widetilde{x}_n) - \alpha \right|$ converges to 0 (and does not depend on any specific $n$). $\delta_n$ goes to zero by assumption as well. Taking the supremum over $\theta \in \ThetaBEE$ on both sides of the previous equation and taking $\limsup$ in $n$ on both sides again yields~\eqref{eq:unif_conv_sigma_unknown} and allows us to conclude. $\Box$


\section{Proof of results in Section~\ref{sec:lin_mod_exo}}
\label{appendix:sec:proof_section4_OLS}

\subsection{Proof of Theorem~\ref{thm:exogenous_case_asymptotic_pointwise_exactness} (pointwise asymptotic exactness)}
\label{proof:thm:exogenous_case_asymptotic_pointwise_exactness}

Let $\theta$ be some arbitrary but fixed element in $\Theta$. The proof is divided in two steps.

\paragraph{Step 1.}
First, we show that the relevant regime is the Edgeworth one asymptotically (as opposed to the \(\Rb\) regime).

We start by proving that \(\nuEdg\) tends to~0.
Because, by assumption, \(a_n = 1 + b_n\) with \(b_n = o(1)\) and \(b_n \sqrt{n} \to +\infty\), we have 
\begin{equation*}
    \sqrt{n} \left(1 - \frac{1}{a_n}\right) \to +\infty.
\end{equation*}
Furthermore, we also assume that $\omega_n \to 0$.
Combining those two limits, we obtain 
\begin{equation}
\label{eq:proof_exo_pointwise_asymptotic_exactness_nuExp_to_0}
\frac{\omega_n \alpha
    + \exp\!\big(-n(1-1/a_n)^2/(2 \majorantKurtxi)\big)}{2} 
    \to 0.
\end{equation}
We have $\delta_n = o(1)$ by assumption as well. 
From this and \eqref{eq:proof_exo_pointwise_asymptotic_exactness_nuExp_to_0}, we conclude that
\begin{equation*}
     \nuEdg := \frac{\omega_n \alpha
    + \exp\!\big(-n(1-1/a_n)^2/(2 \majorantKurtxi)\big)}{2}
    + \delta_n
    \to
    0.
\end{equation*}

\medskip

Note that the former limit is uniform in~\(\theta\) since $\nuEdg$ depends only on $n$, $\alpha$, and the constants in Assumption~\ref{hyp:subset_exo_non-asymptotic_validity}.
Therefore, for $n$ large enough, \(\nuEdg < \alpha/2\) and \(n > 2 \majorantEXXprimeconcentration / (\omega_n \alpha)\), and we have to be in the ``Edg'' regime irrespective of the distribution of the data, meaning that
\begin{equation*}
    \Probsousthetan\Big(
    \uu' \betazero \in \CIEdgu
    \Big)
    = \Probsousthetan\!\left(\uu' \betazero \in
    \left[ \uu' \widehat{\beta}
    \pm \dfrac{\QnEdg}{\sqrt{n}} \sqrt{u' \Vhat u
    + \|u\|^2 \RnVar(\omega_n\alpha/2)} \right]
    \right)
\end{equation*}

\paragraph{Step 2.}
We now prove that the right-hand side of the previous equation converges to~\(1 -\alpha\) when \(n\) goes to~\(+\infty\).
To do so, we show that
\begin{equation}
\label{eq:convN01_for_pointwise_asymptotic_exactness_of_Edg_regime}
    \frac{\QuantileGaussAt{1 - \alpha/2}}{\QnEdg} 
    \frac{\sqrt{n} \uu' (\betahat - \betazero)}{\sqrt{\uu' \Vhat \uu
    + \|u\|^2 \RnVar(\omega_n \alpha / 2)}}
    \convD \Normale(0,1).
\end{equation}

\medskip 

Given the definition of $\Theta$, classical results (see \cite{vanderVaart2000}, Chapters 2, 3 and 5 for instance) allow us to claim that
\begin{equation}
\label{eq:convN01_for_Slutsky}
    \frac{\sqrt{n} \uu' (\betahat - \betazero)}{\sqrt{\uu' \Vhat \uu}}
    \convD
    \Normale(0,1).
\end{equation}

\medskip

Under the assumptions $n\omega_n^{3/2} \to +\infty$, $\E[\|X\eps\|^2]<+\infty$ and $\E[\|X\|^4]<+\infty$, we have 
\begin{equation*}
    \RnVar(\omega_n \alpha / 2) \convP 0 
    \quad \text{and} \quad \RnLin(\omega_n\alpha/2) \convP 0.
\end{equation*}
In addition, we also have \(\uu' \Vhat \uu \convP \uu' V \uu > 0\).
Those three results imply that \(\nuApprox \convP 0\).

\medskip

By the convergence of \(\RnVar(\omega_n \alpha / 2)\) and the Continuous Mapping Theorem, we also have
\begin{equation}
\label{eq:ratio_sqrt_uuVhatuu_for_Slutsky}
    \frac{\sqrt{\uu' \Vhat \uu + \|u\|^2 \RnVar(\omega_n \alpha / 2)}}{\sqrt{\uu' \Vhat \uu}}
    \convP 
    1.
\end{equation}

\medskip 

From Step~1, we know that \(\nuEdg \to 0\).
Combined with the assumption \(a_n = 1 + o(1)\), we obtain
\begin{equation}
\label{eq:LnEdg_convP_QuantileGaussian_for_Slutsky}
    \QnEdg \convP \QuantileGaussAt{1 - \alpha/2}.
\end{equation}

\medskip

Equations~\eqref{eq:convN01_for_Slutsky}, \eqref{eq:ratio_sqrt_uuVhatuu_for_Slutsky} and~\eqref{eq:LnEdg_convP_QuantileGaussian_for_Slutsky} and Slutsky's lemma give the desired result of Equation~\eqref{eq:convN01_for_pointwise_asymptotic_exactness_of_Edg_regime}.
The latter implies that the probability
\begin{equation*}
\Probsousthetan\!\left( 
- \QuantileGaussAt{1-\alpha/2}
\leq 
\frac{\QuantileGaussAt{1 - \alpha/2}}{\QnEdg} 
    \frac{\sqrt{n} \uu' (\betahat - \betazero)}{\sqrt{\uu' \Vhat \uu + \|u\|^2 \RnVar(\omega_n\alpha/2)}}
\leq 
\QuantileGaussAt{1-\alpha/2}
\right)
\end{equation*}
converges to~\(1 - \alpha\).
Finally, the event considered in this probability is equivalent to
\begin{equation*}
    \left\{ - \frac{\QnEdg }{\sqrt{n}}
    \leq 
    \frac{\uu' (\betahat - \betazero)}{\sqrt{\uu' \Vhat \uu + \|u\|^2 \RnVar(\omega_n\alpha/2)}}
    \leq 
    \frac{\QnEdg }{\sqrt{n}} \right\},
\end{equation*}
also equivalent to the event \(\{ \uu'\betazero \in \CIEdgu\}\),
which concludes the proof.
\(\Box\)

\subsection{Proof of Theorem~\ref{thm:exogenous_case_non-asymptotic_validity} (non-asymptotic validity)}
\label{proof:thm:exogenous_case_non-asymptotic_validity}

For fixed \(\alpha\), \(n\), \(\omega\) and \(a\),
we remark that exactly one case out of the two that intervene in definition of \(\CIEdgu\) arises.
Furthermore, the conditions defining these cases are deterministic.
As a consequence, we can consider each case separately and check that the coverage of \(\uu' \betazero\) is at least~\(1 - \alpha\).

For the first regime, when \(\CIEdgu = \Rb\), it is obvious.
Otherwise, we have 
\begin{equation*}
    \CIEdgu = \bigg[ \uu' \widehat{\beta}
    \pm \frac{\QnEdg}{\sqrt{n}} \sqrt{u' \Vhat u
    + \|u\|^2 \RnVar(\omega\alpha/2)} \bigg],
\end{equation*}
and the coverage is guaranteed by Lemma~\ref{lem:OLS_Edg} stated and proven below.
$\Box$

\begin{lemma}
\label{lem:OLS_Edg}
    For every \(\alpha \in (0,1)\), $a > 1$, $n \geq 1$ and $\omega \in (0,1)$, 
    if \(n \omega > 2\majorantEXXprimeconcentration / \alpha \) and \(\nuEdg < \alpha/2\),
    we have, for every $\theta \in \ThetaBEE$,
    \begin{align*}
        \Probsousthetan
        \Big( |\uu' \widehat{\beta} - \uu' \beta_{0}|
        \leq \QnEdg n^{-1/2} \sqrt{u' \Vhat u + \norme{\uu}^2\RnVar(\omega\alpha/2)} \Big) 
        \geq 1-\alpha.
    \end{align*}
\end{lemma}

\noindent
\textit{Proof of Lemma~\ref{lem:OLS_Edg}.} Let us define $\gamma := \omega \alpha / 2$.
We want to show for every $\theta\in\ThetaBEE$
\begin{equation}
\label{eq:OLS_final_concentration_inequality}
    \Probsousthetan
    \Bigg( |\uu' \widehat{\beta} - \uu' \beta_{0}| >  \QnEdg n^{-1/2} \sqrt{u' \Vhat u + \|u\|^2 \RnVar(\gamma)} \Bigg) \leq \alpha.
\end{equation}

\medskip

The proof is divided in three steps. 
In the first two ones, we derive two key intermediary inequalities that hold with high probability.
In the final one, we combine those building bricks to obtain our result.


\paragraph{Step 1. Control of variance and linearization.} 

In this first step, we determine a high probability event (Step~1a.) on which we are able to control the residual term coming from the linearization of \(\uu' \betahat\) (Step~1b.) and upper bound the oracle variance by a feasible quantity (Step~1c.). In this first step, we resort to a number of shorthands to clarify exposition: we let $\widetilde{X} := \E[XX']^{-1/2}X$, $S := n^{-1} \sum_{i=1}^n X_i X_i'$ and $\widetilde{S} := n^{-1} \sum_{i=1}^n \widetilde{X}_i \widetilde{X}_i'$.

\paragraph{Step 1a. Finding an event $\eventA$ of high probability on which Steps 1b and 1c hold.}

Combining Assumption~\ref{hyp:subset_exo_non-asymptotic_validity}~(ii) and Lemma~\ref{lem:concentr_sq_mat} with $A_i := \widetilde{X}_i \widetilde{X}'_i$, for $i=1,\dots,n$, we obtain
\begin{equation}
    \Probsousthetan\Bigg( \underbrace{\norme{\frac{1}{n} \sum_{i=1}^n \widetilde{X}_i \widetilde{X}_i' - \Id_p}
    \leq \sqrt{\frac{\majorantEXXprimeconcentration}{n\gamma}}}_{=: \, \eventA_1} \Bigg)
    \geq 1 - \gamma.
    \label{eq:event_A1}
\end{equation}
On the event~\(\eventA_1\), thanks to Lemma~\ref{lem:lower_bound_mineigen_sym_mat}, we have
\begin{equation*}
    \minvp(\widetilde{S})
    \geq 1 - \sqrt{\frac{\majorantEXXprimeconcentration}{n\gamma}}.
\end{equation*}
Thanks to the constraint on $\omega$, which ensures $n \omega > 2 \majorantEXXprimeconcentration / \alpha$, we get on~\(\eventA_1\)
\begin{equation}
    \minvp(\widetilde{S})
    \geq 
    1 - \sqrt{\frac{\majorantEXXprimeconcentration}{n\gamma}}
    >
    0,
    \label{eq:event_A1_ter}
\end{equation}
which implies that $\widetilde{S}$ is invertible. We also remark that $S$ can be rewritten as $\E[XX']^{1/2}\widetilde{S}\E[XX']^{1/2}$. This, invertibility of $\widetilde{S}$ and Assumption~\ref{hyp:subset_exo_non-asymptotic_validity}.(i) ensure $S$ is invertible as well on $\mathcal{A}_1$. The estimator $\widehat{\beta}$ is well-defined on the same event.

\medskip 

In parallel, Assumption~\ref{hyp:subset_exo_non-asymptotic_validity}~(iii) and  Lemma~\ref{lem:concentr_inf_function} provide the following inequality
\begin{equation}
    \Probsousthetan\bigg( \underbrace{\norme{\frac{1}{n}\sum_{i=1}^n \widetilde{X}_i\eps_i}
    \leq \sqrt{\frac{2}{n}}
    \left( \frac{\majorantEnormeXepsquatre}{\gamma} \right)^{1/4}}_{=: \, \eventA_2} \bigg)
    \geq 1-\gamma.
    \label{eq:event_A2}
\end{equation}

\medskip 

Finally, we define $\eventA := \eventA_1 \cap \eventA_2$, which satisfies $\Prob(\eventA) \geq 1 - 2 \gamma$ thanks to  Equations~\eqref{eq:event_A1} and~\eqref{eq:event_A2} and the law of total probability.
Note that on $\eventA$, the estimator $\widehat{\beta}$ is well-defined as well.

\paragraph{Step 1b. Linearization.}
The goal in this step is to formalize the approximation
$\uu' \widehat{\beta} \approx \uu' \beta_{0} + \MeanXi$ where $\MeanXi$ has the following expression
\begin{align*}
    \MeanXi 
    := \frac{1}{n} \sumiunn \xi_i
    = \frac{1}{n} \sum_{i=1}^n
    \uu' \E[XX']^{-1} X_i\eps_i.
\end{align*}
%
%
On the event $\eventA$, we wish to show that
\begin{equation}
    \sqrt{n} \, \big|\uu' \widehat{\beta} - \uu' \beta_{0} - \MeanXi \big|
    \leq 
    \frac{\sqrt{2}\|u\|\minlambdaEXXt^{-1/2}}{1-\sqrt{\majorantEXXprimeconcentration/(n\gamma)}}\sqrt{\frac{\majorantEXXprimeconcentration}{n\gamma}}\left(\frac{\majorantEnormeXepsquatre}{\gamma}\right)^{1/4}
    =: \RnLin(\gamma).
    \label{eq:linearization_error}
\end{equation}

\medskip

By definition of $\betahat$ and $\widetilde{X}_i$, we can write
\begin{equation*}
    \betahat - \betazero = S^{-1}\frac{1}{n}\sum_{i=1}^nX_i\eps_i = \E[XX']^{-1/2}\widetilde{S}^{-1}\frac{1}{n}\sum_{i=1}^n\widetilde{X}_i\eps_i.
\end{equation*}

This and the identity $A^{-1}-\Id_p = A^{-1}(\Id_p-A)$ yield
\begin{align}
    \norme{ \widehat{\beta} - \beta_{0} - \frac{1}{n} \sum_{i=1}^n
    \E[XX']^{-1} X_i\eps_i }
    & = \norme{\E[XX']^{-1/2}\left(\widetilde{S}^{-1}-\Id_p\right)\frac{1}{n}\sum_{i=1}^n\widetilde{X}_i\eps_i} \nonumber \\
    & = \norme{\E[XX']^{-1/2}\widetilde{S}^{-1}\left(\Id_p - \widetilde{S}\right)\frac{1}{n}\sum_{i=1}^n\widetilde{X}_i\eps_i}.
    \label{eq:bound_beta_lin_1}
\end{align}

From Equation~\eqref{eq:bound_beta_lin_1}, we obtain using the properties of the operator norm (in particular $||A^{-1}|| = \minvp(A)^{-1}$), Assumption~\ref{hyp:subset_exo_non-asymptotic_validity}(i) and the definition of the event $\mathcal{A}$, 
\begin{align}
    \norme{ \widehat{\beta} - \beta_{0} - \frac{1}{n} \sum_{i=1}^n
    \E[XX']^{-1} X_i\eps_i } & \leq \norme{\E[XX']^{-1/2}} \norme{\widetilde{S}^{-1}} \norme{ \widetilde{S} - \Id_p } \norme{ \frac{1}{n}\sum_{i=1}^n\widetilde{X}_i\eps_i } \nonumber \\
    & \leq \frac{\minlambdaEXXt^{-1/2}}{1-\sqrt{\majorantEXXprimeconcentration/(n\gamma)}}\sqrt{\frac{\majorantEXXprimeconcentration}{n\gamma}}\sqrt{\frac{2}{n}}\left(\frac{\majorantEnormeXepsquatre}{\gamma}\right)^{1/4}.
    \label{eq:bound_beta_lin_2}
\end{align}
To get~\eqref{eq:linearization_error}, we remark that by the Cauchy-Schwarz inequality
\begin{align*}
    \sqrt{n} \, \big|\uu' \widehat{\beta} - \uu' \beta_{0} - \MeanXi \big| \leq \sqrt{n} \|u\| \, \norme{ \widehat{\beta} - \beta_{0} - \frac{1}{n} \sum_{i=1}^n
    \E[XX']^{-1} X_i\eps_i }
\end{align*}
and apply~\eqref{eq:bound_beta_lin_2}.

\medskip

As a byproduct of~\eqref{eq:linearization_error}, we obtain a concentration inequality on $\norme{\betahat - \beta_0}$ valid on $\mathcal{A}$ which proves useful in Step 1.c. By the triangle and Cauchy-Schwarz's inequalities and the definition of $\mathcal{A}$, we get
\begin{align}
    \norme{\betahat - \beta_0} 
    & \leq \norme{ \widehat{\beta} - \beta_{0} - \frac{1}{n} \sum_{i=1}^n
    \E[XX']^{-1} X_i\eps_i } + \norme{ \E[XX']^{-1/2} \frac{1}{n} \sum_{i=1}^n
     \widetilde{X}_i\eps_i } \nonumber \\
     & \leq \norme{ \widehat{\beta} - \beta_{0} - \frac{1}{n} \sum_{i=1}^n
    \E[XX']^{-1} X_i\eps_i } + \norme{ \E[XX']^{-1/2} } \, \norme{ \frac{1}{n} \sum_{i=1}^n
     \widetilde{X}_i\eps_i } \nonumber \\
     & \leq \minlambdaEXXt^{-1/2} \left\{ \frac{\sqrt{\majorantEXXprimeconcentration/(n\gamma)}}{1-\sqrt{\majorantEXXprimeconcentration/(n\gamma)}} + 1 \right\} \sqrt{\frac{2}{n}}\left(\frac{\majorantEnormeXepsquatre}{\gamma}\right)^{1/4}.
     \label{eq:concentr_ineq_beta}
\end{align}
 


\paragraph{Step 1c. Bound on the distance to the oracle variance $\uu'\Voracle\uu$.}
In this step, we still reason on the event $\eventA$ on which we prove
\begin{align}
    |\uu'\Voracle\uu - \uu' \Vhat \uu| \leq  \|u\|^2 \RnVar(\gamma)
    \label{eq:bound_Vjoracle}
\end{align}
for some $\RnVar(\gamma)$ to be specified later and
\begin{align*}
    \Voracle := \E[XX']^{-1}\left(\frac{1}{n} \sum_{i=1}^nX_iX_i' \eps_i^2 \right)\E[XX']^{-1}.
\end{align*}
Note that 
$\uu' \Voracle \uu = \frac{1}{n}\sum_{i=1}^n\langle \uu , \E[XX']^{-1}X_i\eps_i \rangle^2.$
Adding and subtracting $\epshat_i$ and then expanding the square yields
$\uu' \Voracle \uu = V_1 + V_2 + V_3,$
where
\begin{align*}
    V_1 &:= \frac{1}{n}\sum_{i=1}^n\langle \uu , \E[XX']^{-1}X_i\epshat_i \rangle^2 \\
    V_2 &:= \frac{1}{n}\sum_{i=1}^n\langle \uu , \E[XX']^{-1}X_i(\eps_i-\epshat_i) \rangle^2 \\
    V_3 &:= \frac{2}{n}\sum_{i=1}^n\langle \uu , \E[XX']^{-1}X_i\epshat_i \rangle\langle \uu , \E[XX']^{-1}X_i(\eps_i-\epshat_i) \rangle.
\end{align*}
By Cauchy-Schwarz and using \(\eps_i - \epshat_i = X_i'(\betahat - \betazero)\), 
as well as Equation~\eqref{eq:concentr_ineq_beta} and Assumption~\ref{hyp:subset_exo_non-asymptotic_validity}~(i),
we have
\begin{align*}
    V_2 &\leq \|\uu\|^2 \|\E[XX']^{-1}\|^2 \|\widehat{\beta}-\beta_0\|^2
    \times 
    \frac{1}{n}\sum_{i=1}^n\|X_iX_i'\|^2 \\
    & \leq \|\uu\|^2 \frac{1}{\minlambdaEXXt^2} \left( \minlambdaEXXt^{-1/2} \left\{ \frac{\sqrt{\majorantEXXprimeconcentration/(n\gamma)}}{1-\sqrt{\majorantEXXprimeconcentration/(n\gamma)}} + 1 \right\} \sqrt{\frac{2}{n}}\left(\frac{\majorantEnormeXepsquatre}{\gamma}\right)^{1/4} \right)^2 
    \times 
    \frac{1}{n}\sum_{i=1}^n\|X_i\|^4 \\
    & \leq 
    \frac{2\|\uu\|^2}{n\minlambdaEXXt^3} \left( \frac{\sqrt{\majorantEXXprimeconcentration/(n\gamma)}}{1-\sqrt{\majorantEXXprimeconcentration/(n\gamma)}} + 1 \right)  ^2 \sqrt{\frac{\majorantEnormeXepsquatre}{\gamma}}
    \times 
    \frac{1}{n}\sum_{i=1}^n\|X_i\|^4
\end{align*}
and
\begin{align*}
    V_3 & \leq 2\|\uu\|^2 \|\E[XX']^{-1}\|^2 \|\widehat{\beta}-\beta_0\|
    \times 
    \frac{1}{n}\sum_{i=1}^n\|X_iX_i'\|\,\|X_i\epshat_i\| \\
    & \leq \frac{2\|\uu\|^2}{\minlambdaEXXt^{5/2}} 
    \left( \frac{\sqrt{\majorantEXXprimeconcentration/(n\gamma)}}{1-\sqrt{\majorantEXXprimeconcentration/(n\gamma)}} + 1 \right)\sqrt{\frac{2}{n}}\left( \frac{\majorantEnormeXepsquatre}{\gamma} \right)^{1/4}
    \times 
    \frac{1}{n}\sum_{i=1}^n\|X_i\|^2\,\|X_i\epshat_i\| \\
    & = 
    \frac{2\sqrt{2}\|\uu\|^2}{\minlambdaEXXt^{5/2} \sqrt{n}} \left( \frac{\sqrt{\majorantEXXprimeconcentration/(n\gamma)}}{1-\sqrt{\majorantEXXprimeconcentration/(n\gamma)}} + 1 \right)
    \left( \frac{\majorantEnormeXepsquatre}{\gamma} \right)^{1/4}
    \times 
    \frac{1}{n}\sum_{i=1}^n\|X_i\|^3\,|\epshat_i|.
\end{align*}

\medskip

Let us now focus on $V_1$.
We define $H_n := \E[XX']^{-1}-\left(\frac{1}{n}\sum_{i=1}^nX_iX_i'\right)^{-1}$.
Adding and substracting $(n^{-1} \sumiunn X_i X_i')^{-1}$ and expanding the square, we get
\begin{align*}
    V_1 &= \frac{1}{n}\sum_{i=1}^n\langle \uu , \E[XX']^{-1}X_i\epshat_i \rangle^2 = \uu'\Vhat\uu + V_4 + V_5,
\end{align*}
where
\begin{align*}
    V_4 &:= \frac{1}{n}\sum_{i=1}^n\langle \uu , H_n X_i\epshat_i \rangle^2, \\
    V_5 &:= \frac{2}{n}\sum_{i=1}^n\langle \uu , H_n X_i\epshat_i \rangle
    \bigg\langle \uu , \Big(\frac{1}{n}\sum_{j=1}^n X_j X_j'\Big)^{-1} X_i\epshat_i \bigg\rangle.
\end{align*}
We remark that 
\begin{equation*}
    H_n = \E[XX']^{-1/2}\left(\Id_p-\widetilde{S}^{-1}\right)\E[XX']^{-1/2} = \E[XX']^{-1/2}\left(\widetilde{S}-\Id_p\right)\widetilde{S}^{-1}\E[XX']^{-1/2},
\end{equation*}
so that by definition of $\eventA$ we get
\begin{align*}
    \|H_n\| \leq \frac{\minlambdaEXXt^{-1}\sqrt{\majorantEXXprimeconcentration/(n\gamma)}}{1-\sqrt{\majorantEXXprimeconcentration/(n\gamma)}}.
\end{align*}
Therefore,
\begin{align*}
    V_4 &\leq \|\uu\|^2\|H_n\|^2 \times
    \frac{1}{n}\sum_{i=1}^n\|X_i\epshat_i\|^2
    \leq \frac{\|\uu\|^2\minlambdaEXXt^{-2}\majorantEXXprimeconcentration/(n\gamma)}{(1-\sqrt{\majorantEXXprimeconcentration/(n\gamma)})^2} \times 
    \frac{1}{n}\sum_{i=1}^n\|X_i\epshat_i\|^2,
\end{align*}
and
\begin{align*}
    |V_5| & = \left|2\uu'H_n\frac{1}{n}\sum_{i=1}^nX_iX_i'\epshat_i^2S^{-1}\uu\right| \\
    & \leq 2\|\uu\|^2 \|H_n\| \times
    \norme{\frac{1}{n}\sum_{i=1}^nX_iX_i'S^{-1}\epshat_i^2} 
    \leq \frac{2\|\uu\|^2\minlambdaEXXt^{-1}\sqrt{\majorantEXXprimeconcentration/(n\gamma)}}{1-\sqrt{\majorantEXXprimeconcentration/(n\gamma)}} \times \norme{\frac{1}{n}\sum_{i=1}^nX_iX_i'S^{-1}\epshat_i^2}.
\end{align*}

\medskip

Finally, we obtain Equation~\eqref{eq:bound_Vjoracle}, for the choice of $\RnVar(\gamma)$ given by
\begin{align*}
    \RnVar(\gamma)
    := & \;\;
     \frac{2}{n\minlambdaEXXt^3} \left( \frac{\sqrt{\majorantEXXprimeconcentration/(n\gamma)}}{1-\sqrt{\majorantEXXprimeconcentration/(n\gamma)}} + 1 \right)  ^2 \sqrt{\frac{\majorantEnormeXepsquatre}{\gamma}}
    \times 
    \frac{1}{n}\sum_{i=1}^n\|X_i\|^4 \\
    & + 
    \frac{2\sqrt{2}}{\minlambdaEXXt^{5/2} \sqrt{n}} \left( \frac{\sqrt{\majorantEXXprimeconcentration/(n\gamma)}}{1-\sqrt{\majorantEXXprimeconcentration/(n\gamma)}} + 1 \right)
    \left( \frac{\majorantEnormeXepsquatre}{\gamma} \right)^{1/4}
    \times 
    \frac{1}{n}\sum_{i=1}^n\|X_i\|^3\,|\epshat_i| \\
    &+ \frac{\minlambdaEXXt^{-2}\majorantEXXprimeconcentration/(n\gamma)}{(1-\sqrt{\majorantEXXprimeconcentration/(n\gamma)})^2} \times 
    \frac{1}{n}\sum_{i=1}^n\|X_i\epshat_i\|^2 \\
    &+ \frac{2\minlambdaEXXt^{-1}\sqrt{\majorantEXXprimeconcentration/(n\gamma)}}{1-\sqrt{\majorantEXXprimeconcentration/(n\gamma)}} \times \norme{\frac{1}{n}\sum_{i=1}^nX_iX_i'S^{\dagger}\epshat_i^2}.
\end{align*}

\paragraph{Step 2. Control of the self-normalized sum $\MeanXi / \sqrt{\uu' \Voracle \uu}$.}

Applying Lemma~\ref{lem:edg_exp}~(ii) with $\delta_n \geq \sup_{\theta \in \ThetaBEE}
    \min\{ \DeltanE(\theta) , \DeltanB(\theta) \}$ by definition,
we have
\begin{align}
    \Probsousthetan\left( \underbrace{ |\MeanXi| > yn^{-1/2}\sqrt{\uu' \Voracle \uu} }_{=: \, \eventB} \right)
    \leq 2 \Bigg( \Phi \bigg(- \frac{y}{\sqrt{a}} \bigg) + \delta_n \Bigg)
    + \exp\!\left( -\frac{n(1-1/a)^2}{2 \majorantKurtxi} \right)
    \label{eq:control_self_norm}
\end{align}
for any given value $y>0$.

\paragraph{Step 3. Combining the previous results.}

Recall we want to show Equation~\eqref{eq:OLS_final_concentration_inequality}.
We denote by \(\eventC\) the corresponding event
\begin{equation*}
    \eventC
    := \left\{ |\uu' \widehat{\beta} - \uu' \beta_{0}| > \sqrt{u' \Vhat u + \|u\|^2 \RnVar(\gamma)}\QnEdg n^{-1/2}
    \right\}.
\end{equation*}

\medskip

We can use the fact that $\gamma = \omega\alpha/2$, the definition of $\eventA$ and $\eventB$ and the law of total probabilities  to write 
\begin{align}
    \Probsousthetan ( \eventC )
    & \leq \Probsousthetan ( \eventC \cap \eventA \cap \eventB )
    + \Probsousthetan ( \eventA^c ) + \Probsousthetan ( \eventB^c ) \nonumber\\
    & \leq \Probsousthetan ( \eventC \cap \eventA \cap \eventB )
    + \omega\alpha + 2 \Bigg( \Phi \bigg(- \frac{y}{\sqrt{a}} \bigg) + \delta_n \Bigg)
    + \exp\!\left( -\frac{n(1-1/a)^2}{2 \majorantKurtxi} \right).
    \label{eq:prob_control_event_C}
\end{align}
Then, combining Equations~\eqref{eq:linearization_error},~\eqref{eq:bound_Vjoracle} and~\eqref{eq:control_self_norm}, we get on $\eventC \cap \eventA \cap \eventB$,
\begin{align*}
    \sqrt{u' \Vhat u + \|u\|^2 \RnVar(\gamma)}\QnEdg / \sqrt{n} & \leq |\uu' \widehat{\beta} - \uu' \beta_{0}| \\
    & \leq |\MeanXi| + \RnLin(\gamma) / \sqrt{n} \\
    & \leq y\sqrt{\uu'\Voracle\uu}/\sqrt{n} + \RnLin(\gamma) / \sqrt{n} \\
    & \leq y\sqrt{\uu' \Vhat \uu + \|u\|^2 \RnVar(\gamma)}/\sqrt{n} + \RnLin(\gamma) / \sqrt{n}.   
\end{align*}
This and Equation~\eqref{eq:prob_control_event_C} ensure that
\begin{align}
    \Probsousthetan ( \eventC ) & \leq \Probsousthetan
    \left( \frac{\sqrt{u' \Vhat u + \|u\|^2 \RnVar(\gamma)}\QnEdg }{ \sqrt{n} } \leq \frac{ y\sqrt{\uu' \Vhat \uu + \|u\|^2 \RnVar(\gamma)} }{ \sqrt{n} } + \frac{ \RnLin(\gamma) }{ \sqrt{n} } \right) \nonumber \\
    & \quad + \omega\alpha + 2 \Bigg( \Phi \bigg(- \frac{y}{\sqrt{a}} \bigg) + \delta_n \Bigg)
    + \exp\!\left( -\frac{n(1-1/a)^2}{2 \majorantKurtxi} \right).
\end{align}

\noindent Our goal is to choose $y$ such that
\begin{align*}
    2 \Bigg( \Phi \bigg(- \frac{y}{\sqrt{a}} \bigg) + \delta_n \Bigg)
    + \exp\!\left( -\frac{n(1-1/a)^2}{2\majorantKurtxi} \right)
    = (1 - \omega)\alpha.
\end{align*}
Solving this equation, we find that the solution is
\begin{align*}
    y^* & = \sqrt{a} \, q_{\mathcal{N}(0,1)}
    \bigg(1 - \frac{(1-\omega)\alpha
    - \exp\!\big(-n(1-1/a)^2/(2\majorantKurtxi)\big)}{2} + \delta_n \bigg) \\
    & = \sqrt{a} \QuantileGaussAt{1 - \alpha/2 + \nuEdg}
\end{align*}
whenever \(\nuEdg < \alpha/2\).
Therefore, we get
\begin{align*}
    \Probsousthetan ( \eventC )
    &\leq \Probsousthetan
    \Bigg(y^*\sqrt{\uu' \Vhat \uu
    + \|u\|^2 \RnVar(\gamma)}/\sqrt{n} + \RnLin(\gamma)/\sqrt{n} > \QnEdg\sqrt{\uu' \Vhat \uu
    + \|u\|^2 \RnVar(\gamma)} /\sqrt{n} \Bigg) \\
    & \quad + \omega\alpha
    + (1 - \omega) \alpha \\
    & = \alpha,
\end{align*}
since the equality 
\begin{equation*}
    y^*\sqrt{\uu' \Vhat \uu
    + \|u\|^2 \RnVar(\gamma)}/\sqrt{n} + \RnLin(\gamma)/\sqrt{n} = \QnEdg\sqrt{\uu' \Vhat \uu
    + \|u\|^2 \RnVar(\gamma)} /\sqrt{n}
\end{equation*}
holds (so that the probability appearing on the right-hand side of the final inequality just above is zero). $\Box$



\subsection{Proof of Theorem~\ref{thm:exogenous_case_asymptotic_uniform_exactness} (non-asymptotic bound on the uniform exactness with rates)}

\subsubsection{Proof when \texorpdfstring{$\rho>0$}{rho > 0}}

By Theorem~\ref{thm:exogenous_case_non-asymptotic_validity}, we can write for every $n \geq 1$ and every $\theta \in \ThetaBEE$ (and thus every $\theta \in \ThetaBEEstrict$),
\begin{equation}
    \label{eq:target_unif_exact}
    \Probsousthetan\Big( \uu' \betazero \in \CIEdgu \Big) \geq 1-\alpha.
\end{equation}
By assumption, $\omega_n$ and $b_n$ are chosen independent from the true value $\theta$, and $\omega_n \to 0$, $n\omega_n^{3/2} \to +\infty$ and $b_n\sqrt{n} \to +\infty$. Recalling the definition of $n_0 := \max \{n \in \INTST: 
n \leq 2 \majorantEXXprimeconcentration / (\omega_n \alpha)
\text{ or } \nuEdg \geq \alpha / 2 \}$ (first given in Section~\ref{subsubsec:formal_definition_our_OLS}), we remark that $n_0$ is finite and independent from $\theta$. By definition of $\CIEdgu$, we can write for every $n > n_0$ and every $\theta \in \ThetaBEEstrict$ that the following holds $\Probsousthetan$-almost surely
\begin{equation*}
    \CIEdgu = \left[ \uu' \widehat{\beta}
        \pm \dfrac{\QnEdg}{\sqrt{n}} \sqrt{u' \Vhat u
        + \|u\|^2 \RnVar(\omega_n\alpha/2)} \right].
\end{equation*}

\medskip

\paragraph{Step 1. Construction of a suitable high-probability event.} We assume $n > n_0$ from now on. Recall the definition of $\mathcal{A}$ in the proof of Lemma~\ref{lem:OLS_Edg}. We have
\begin{align}
    \label{eq:intermed_unif_exact_1}
    \Probsousthetan\Big( \uu' \betazero \in & \,\CIEdgu \Big) 
    \\ & \leq \Probsousthetan\Big( \left\{\uu' \betazero \in \CIEdgu\right\} \cap \mathcal{A} \Big) + \omega_n\alpha \nonumber \\
    & = \Probsousthetan\left( \left\{ \left|u'(\betahat-\betazero)\right| \leq \frac{\QnEdg}{\sqrt{n}} \sqrt{u' \Vhat u
    + \|u\|^2 \RnVar(\omega_n\alpha/2)} \right\} \cap \mathcal{A} \right) + \omega_n\alpha.
\end{align}
Let $M_1 := 2(C_\rho \majorantEnormeXquatreunplusrho)^{1/(1+\rho)}$ and $M_2 := \majorantEnormeXquatreunplusrho^{1/{(1+\rho)}}\majorantEnormeXepsquatre$ with $C_\rho$ the constant given in Lemma~\ref{lem:marcinkiewicz_zygmund}. We define
\begin{align*}
    \mathcal{B} & := \left\{ \left| \frac{1}{n}\sum_{i=1}^n||X_i||^4 -  \E[||X||^4] \right| \leq \frac{M_1}{\ln(n)^{1/(1+\rho)}}\right\} \cap \left\{\left|\frac{1}{n}\sum_{i=1}^n||X_i\varepsilon_i||^2 - \E\left[||X\varepsilon||^2\right]\right|\leq \sqrt{\frac{M_2}{\ln (n)}}\right\} \\
    & \cap \left\{\left|\frac{1}{n}\sum_{i=1}^n\left(u'\E[XX']^{-1}X_i\varepsilon_i\right)^2-\E\left[\left(u'\E[XX']^{-1}X\varepsilon\right)^2\right]\right| \leq \frac{\|\uu\|^2\minlambdaEXXt^{-1}\sqrt{\majorantEnormeXepsquatre}}{\sqrt{n\omega_n\alpha}}\right\}.
\end{align*}
By the law of total probabilities, Markov's inequality, Assumption~\ref{hyp:subset_exo_unif_asymptotic_exactness} and Lemmas~\ref{lem:up_bound_moments} and~\ref{lem:marcinkiewicz_zygmund}, we can write
\begin{align}
    \label{eq:intermed_unif_exact_2}
    & \Probsousthetan\left( \left\{ \left|u'(\betahat-\betazero)\right| \leq \frac{\QnEdg}{\sqrt{n}} \sqrt{u' \Vhat u
    + \|u\|^2 \RnVar(\omega_n\alpha/2)} \right\} \cap \mathcal{A} \right) \nonumber \\
    & \leq \Probsousthetan\left( \left\{ \frac{\sqrt{n}\left|u'(\betahat-\betazero)\right|}{\sqrt{u'Vu}} \leq \QnEdg\sqrt{\frac{u' \Vhat u
    + \|u\|^2 \RnVar(\omega_n\alpha/2)}{u'Vu}} \right\} \cap \mathcal{A} \cap \mathcal{B} \right) + \Probsousthetan(\mathcal{B}^c)
    \nonumber \displaybreak[0] \\
    & \leq \Probsousthetan\left( \left\{ \frac{\sqrt{n}\left|u'(\betahat-\betazero)\right|}{\sqrt{u'Vu}} \leq \QnEdg\sqrt{\frac{u' \Vhat u
    + \|u\|^2 \RnVar(\omega_n\alpha/2)}{u'Vu}} \right\} \cap \mathcal{A} \cap \mathcal{B} \right) \nonumber \\
    & + \Probsousthetan\left( \left| \frac{1}{n}\sum_{i=1}^n||X_i||^4 -  \E[||X||^4] \right| > \frac{M_1}{\ln(n)^{1/(1+\rho)}} \right) \nonumber \\
    & + \Probsousthetan\left( \left|\frac{1}{n}\sum_{i=1}^n||X_i\varepsilon_i||^2 - \E\left[||X\varepsilon||^2\right]\right| > \sqrt{\frac{M_2}{\ln (n)}} \right) \nonumber \\
    & + \Probsousthetan\left( \left|\frac{1}{n}\sum_{i=1}^n\left(u'\E[XX']^{-1}X_i\varepsilon_i\right)^2-\E\left[\left(u'\E[XX']^{-1}X\varepsilon\right)^2\right]\right| \leq \frac{\|\uu\|^2\minlambdaEXXt^{-1}\sqrt{\majorantEnormeXepsquatre}}{\sqrt{n\omega_n\alpha}} \right) \nonumber \displaybreak[0] \\
    & \leq \Probsousthetan\left( \left\{ \frac{\sqrt{n}\left|u'(\betahat-\betazero)\right|}{\sqrt{u'Vu}} \leq \QnEdg\sqrt{\frac{u' \Vhat u
    + \|u\|^2 \RnVar(\omega_n\alpha/2)}{u'Vu}} \right\} \cap \mathcal{A} \cap \mathcal{B} \right) \nonumber \\
    & + \ln(n)\left( n^{-\rho}\Indicator\{\rho \leq 1\} + n^{-(1+\rho)/2}\Indicator\{\rho > 1\} + n^{-1} \right) + \omega_n\alpha.
\end{align}

\paragraph{Step 2. Control of the distance to the influence function.} Let
\begin{align*}
    \mathcal{C}
    := \left\{ \frac{\sqrt{n}\left|u'(\betahat-\betazero)\right|}{\sqrt{u'Vu}}
    \leq \QnEdg\sqrt{\frac{u' \Vhat u
    + \|u\|^2 \RnVar(\omega_n\alpha/2)}{u'Vu}} \right\} \cap \mathcal{A} \cap \mathcal{B}.
\end{align*}
Our goal is now to prove there exists a decreasing positive deterministic sequence $\mu_n$ (given below) such that the following holds on $\mathcal{C}$
\begin{equation}
    \label{eq:intermed_unif_exact_3}
    \frac{\sqrt{n}\left|\MeanXi\right|}{\sqrt{u'Vu}} \leq \QuantileGaussAt{1-\alpha/2} + \mu_n.
\end{equation}
We start by lower bounding $\sqrt{n}\left|u'(\betahat-\betazero)\right|/\sqrt{u'Vu}$.
By Equation~\eqref{eq:linearization_error}, we have that, on $\mathcal{A}$,
\begin{align*}
    \sqrt{n}\left|u'(\betahat-\betazero) - \MeanXi\right| \leq \RnLin(\omega_n\alpha/2).
\end{align*}
This and the reverse triangle inequality ensure
\begin{equation*}
    \frac{\sqrt{n}\left|u'(\betahat-\betazero)\right|}{\sqrt{u'Vu}} \geq \frac{\sqrt{n}\left|\MeanXi\right|}{\sqrt{u'Vu}} - \frac{\RnLin(\omega_n\alpha/2)}{\sqrt{u'Vu}}.
\end{equation*}
Assumptions~\ref{hyp:subset_exo_non-asymptotic_validity} and~\ref{hyp:subset_exo_unif_asymptotic_exactness} allow us to write that
\begin{equation*}
    \maxvp\left(\E[XX']\right) = ||\E[XX']|| \leq \E\left[||X||^{4(1+\rho)}\right]^{\frac{1}{2(1+\rho)}} \leq \majorantEnormeXquatreunplusrho^{\frac{1}{2(1+\rho)}},    
\end{equation*}
and $u'Vu \geq ||u||^2\minlambdaEXXt \majorantEnormeXquatreunplusrho^{-\frac{1}{1+\rho}}$, so that
\begin{equation}
    \label{eq:intermed_unif_exact_4}
    \frac{\sqrt{n}\left|u'(\betahat-\betazero)\right|}{\sqrt{u'Vu}} \geq \frac{\sqrt{n}\left|\MeanXi\right|}{\sqrt{u'Vu}} - \frac{\majorantEnormeXquatreunplusrho^{\frac{1}{2(1+\rho)}}\RnLin(\omega_n\alpha/2)}{||u||\minlambdaEXXt^{1/2}}.
\end{equation} 
We also need to exhibit an upper bound on $\QnEdg \sqrt{\frac{u' \Vhat u
+ \|u\|^2 \RnVar(\omega_n\alpha/2)}{u'Vu}}$.
We can write
\begin{align}
    & \QnEdg \sqrt{\frac{u' \Vhat u
    + \|u\|^2 \RnVar(\omega_n\alpha/2)}{u'Vu}} \nonumber \\
    & \leq \left\{ \QuantileGaussAt{1-\alpha/2} + \left|\sqrt{a_n} \, q_{\mathcal{N}(0,1)}
    \big(1 - \alpha/2 + \nuEdg \big) - \QuantileGaussAt{1-\alpha/2} \right| \right\}
    \nonumber \\
    & \hspace{5cm} \times \sqrt{\frac{u' \Vhat u
    + \|u\|^2 \RnVar(\omega_n\alpha/2)}{u'Vu}} 
    +  \frac{\RnLin(\omega_n\alpha/2)}{\sqrt{u'Vu}} \nonumber \displaybreak[0] \\
    & \leq \left\{ \QuantileGaussAt{1-\alpha/2} + \left|\sqrt{a_n} \, q_{\mathcal{N}(0,1)}
    \big(1 - \alpha/2 + \nuEdg \big) - \QuantileGaussAt{1-\alpha/2} \right| \right\}  \nonumber \\
    & \hspace{5cm} \times \sqrt{\frac{u' \Vhat u
    + \|u\|^2 \RnVar(\omega_n\alpha/2)}{u'Vu}}
    + \frac{\majorantEnormeXquatreunplusrho^{\frac{1}{2(1+\rho)}} \RnLin(\omega_n\alpha/2)}{||u||\minlambdaEXXt^{1/2}}.
\label{eq:intermed_unif_exact_5}
\end{align}
Turning to
$\sqrt{(u' \Vhat u
+ \|u\|^2 \RnVar(\omega_n\alpha/2))/(u'Vu)}$,
we note that on $\mathcal{C}$,
\begin{align}
    \sqrt{\frac{u' \Vhat u
    + \|u\|^2 \RnVar(\omega_n\alpha/2)}{u'Vu}} & \leq \sqrt{\frac{|u'( \Vhat - \Voracle) u| + u'\Voracle u
    + \|u\|^2 \RnVar(\omega_n\alpha/2)}{u'Vu}} \nonumber \\
    & \leq \sqrt{1 + (u'Vu)^{-1}\left(\frac{\|\uu\|^2\minlambdaEXXt^{-1}\sqrt{\majorantEnormeXepsquatre}}{\sqrt{n\omega_n\alpha}} + 2\|u\|^2 \RnVar(\omega_n\alpha/2) \right)} \nonumber \\   
    & \leq \sqrt{1 + \frac{\majorantEnormeXquatreunplusrho^{1/(1+\rho)}}{||u||^2\minlambdaEXXt }\left(\frac{\|\uu\|^2\minlambdaEXXt^{-1}\sqrt{\majorantEnormeXepsquatre}}{\sqrt{n\omega_n\alpha}} + 2\|u\|^2 \RnVar(\omega_n\alpha/2) \right)}.
    \label{eq:up_bound_ratio_var}
\end{align}
The term $\RnVar(\omega_n\alpha/2)$ is stochastic and has to be upper bounded deterministically. Given the definition of $\RnVar(\omega_n\alpha/2)$, we can write $\RnVar(\omega_n\alpha/2) = D_1 + D_2 + D_3 + D_4$, with
\begin{align*}
    & D_1 := \frac{2}{n\minlambdaEXXt^3} \left( \frac{\sqrt{2\majorantEXXprimeconcentration/(n\omega_n\alpha)}}{1-\sqrt{2\majorantEXXprimeconcentration/(n\omega_n\alpha)}} + 1 \right)  ^2 \sqrt{\frac{2\majorantEnormeXepsquatre}{\omega_n\alpha}}
    \times 
    \frac{1}{n}\sum_{i=1}^n\|X_i\|^4 \\
    & D_2 :=  \frac{2\sqrt{2}}{\minlambdaEXXt^{5/2} \sqrt{n}} \left( \frac{\sqrt{2\majorantEXXprimeconcentration/(n\omega_n\alpha)}}{1-\sqrt{2\majorantEXXprimeconcentration/(n\omega_n\alpha)}} + 1 \right)
    \left( \frac{2\majorantEnormeXepsquatre}{\omega_n\alpha} \right)^{1/4}
    \times 
    \frac{1}{n}\sum_{i=1}^n\|X_i\|^3\,|\epshat_i| \\
    & D_3 := \frac{2\minlambdaEXXt^{-2}\majorantEXXprimeconcentration/(n\omega_n\alpha)}{(1-\sqrt{2\majorantEXXprimeconcentration/(n\omega_n\alpha)})^2} \times 
    \frac{1}{n}\sum_{i=1}^n\|X_i\epshat_i\|^2 \\
    & D_4 :=\frac{2\minlambdaEXXt^{-1}\sqrt{2\majorantEXXprimeconcentration/(n\omega_n\alpha)}}{1-\sqrt{2\majorantEXXprimeconcentration/(n\omega_n\alpha)}} \times \norme{\frac{1}{n}\sum_{i=1}^nX_iX_i'S^{-1}\epshat_i^2}.
\end{align*}
On $\mathcal{C}$, it is straightforward to write (using the fact that $\E[||X||^4] \leq \majorantEnormeXquatreunplusrho^{1/(1+\rho)}$ under Assumption~\ref{hyp:subset_exo_unif_asymptotic_exactness}.(ii))
\begin{equation}
    \label{eq:up_bound_D1}
    D_1 \leq \overline{D}_1 := \frac{2}{n\minlambdaEXXt^3} \left( \frac{\sqrt{2\majorantEXXprimeconcentration/(n\omega_n\alpha)}}{1-\sqrt{2\majorantEXXprimeconcentration/(n\omega_n\alpha)}} + 1 \right)  ^2 \sqrt{\frac{2\majorantEnormeXepsquatre}{\omega_n\alpha}}
    \times 
    \left(\frac{M_1}{\ln(n)^{1/(1+\rho)}} + \majorantEnormeXquatreunplusrho^{1/(1+\rho)} \right).
\end{equation}
To control $D_2$, we use the fact that $\epshat_i = X_i'(\betahat-\beta_0)+\varepsilon_i$. Combined with the triangle and Cauchy-Schwarz inequalities, we obtain
\begin{align}
    \label{eq:up_bound_D2_1}
    D_2 \leq & \frac{2\sqrt{2}}{\minlambdaEXXt^{5/2} \sqrt{n}} \left( \frac{\sqrt{2\majorantEXXprimeconcentration/(n\omega_n\alpha)}}{1-\sqrt{2\majorantEXXprimeconcentration/(n\omega_n\alpha)}} + 1 \right)
    \left( \frac{2\majorantEnormeXepsquatre}{\omega_n\alpha} \right)^{1/4}  \nonumber \\
    & \times \left\{ \norme{\betahat-\beta_0}\frac{1}{n}\sum_{i=1}^n\|X_i\|^4 + \sqrt{\frac{1}{n}\sum_{i=1}^n\|X_i\|^4}\sqrt{\frac{1}{n}\sum_{i=1}^n\|X_i\varepsilon_i\|^2} \right\}.
\end{align}
On $\mathcal{C}$, the proof of Lemma~\ref{lem:OLS_Edg} ensures
\begin{equation}
    \label{eq:up_bound_D2_2}
    \norme{\betahat-\beta_0} \leq \minlambdaEXXt^{-1/2} \left\{ \frac{\sqrt{2\majorantEXXprimeconcentration/(n\omega_n\alpha)}}{1-\sqrt{2\majorantEXXprimeconcentration/(n\omega_n\alpha)}} + 1 \right\} \sqrt{\frac{2}{n}}\left(\frac{2\majorantEnormeXepsquatre}{\omega_n\alpha}\right)^{1/4}.
\end{equation}
Equations~\eqref{eq:up_bound_D2_1} and~\eqref{eq:up_bound_D2_2}, $\E[||X\varepsilon||^2] \leq \sqrt{\E[||X\varepsilon||^4]}$, Lemma~\ref{lem:up_bound_moments} and the definition of $\mathcal{C}$ yield
\begin{align}
    \label{eq:up_bound_D2}
    D_2 \leq \overline{D}_2 & := \frac{2\sqrt{2}}{\minlambdaEXXt^{5/2} \sqrt{n}} \left( \frac{\sqrt{2\majorantEXXprimeconcentration/(n\omega_n\alpha)}}{1-\sqrt{2\majorantEXXprimeconcentration/(n\omega_n\alpha)}} + 1 \right)
    \left( \frac{2\majorantEnormeXepsquatre}{\omega_n\alpha} \right)^{1/4} \nonumber \\
    & \times \Bigg\{ \minlambdaEXXt^{-1/2} \left\{ \frac{\sqrt{2\majorantEXXprimeconcentration/(n\omega_n\alpha)}}{1-\sqrt{2\majorantEXXprimeconcentration/(n\omega_n\alpha)}} + 1 \right\} \sqrt{\frac{2}{n}}\left(\frac{2\majorantEnormeXepsquatre}{\omega_n\alpha}\right)^{1/4}\left( \frac{M_1}{\ln(n)^{1/(1+\rho)}} + \majorantEnormeXquatreunplusrho^{1/(1+\rho)}\right) \nonumber \\
    & \qquad + M_2^{1/4}\sqrt{1+\ln(n)^{-1/2}}\sqrt{ \frac{M_1}{\ln(n)^{1/(1+\rho)}} + \majorantEnormeXquatreunplusrho^{1/(1+\rho)}} \Bigg\}.
\end{align}
As for $D_3$, a convexity argument and the definition of $\mathcal{C}$ and~\eqref{eq:up_bound_D2_2} allow us to write
\begin{align}
    \label{eq:up_bound_D3}
    D_3
    &\leq \frac{4\minlambdaEXXt^{-2}\majorantEXXprimeconcentration/(n\omega_n\alpha)}{(1-\sqrt{2\majorantEXXprimeconcentration/(n\omega_n\alpha)})^2} \times 
    \left\{ \norme{\betahat-\beta_0}^2\frac{1}{n}\sum_{i=1}^n\|X_i\|^4 + \frac{1}{n}\sum_{i=1}^n\|X_i\varepsilon_i\|^2 \right\} \nonumber \\
    & \leq \frac{4\minlambdaEXXt^{-2}\majorantEXXprimeconcentration/(n\omega_n\alpha)}{(1-\sqrt{2\majorantEXXprimeconcentration/(n\omega_n\alpha)})^2} \nonumber \\
    & \qquad \times 
    \left\{ \minlambdaEXXt^{-1} \left\{ \frac{\sqrt{2\majorantEXXprimeconcentration/(n\omega_n\alpha)}}{1-\sqrt{2\majorantEXXprimeconcentration/(n\omega_n\alpha)}} + 1 \right\}^2\frac{2}{n}\left(\frac{2\majorantEnormeXepsquatre}{\omega_n\alpha}\right)^{1/2}\frac{1}{n}\sum_{i=1}^n\|X_i\|^4 + \frac{1}{n}\sum_{i=1}^n\|X_i\varepsilon_i\|^2 \right\} \nonumber \\
    &\leq \overline{D}_3
    := \frac{4\minlambdaEXXt^{-2} \majorantEXXprimeconcentration
    / (n\omega_n\alpha)}{ (1-\sqrt{2\majorantEXXprimeconcentration
    /(n\omega_n\alpha)} )^2 } \nonumber \\
    & \qquad \times 
    \Bigg\{ \minlambdaEXXt^{-1} \left\{ \frac{\sqrt{2\majorantEXXprimeconcentration/(n\omega_n\alpha)}}{1-\sqrt{2\majorantEXXprimeconcentration/(n\omega_n\alpha)}} + 1 \right\}^2\frac{2}{n}\left(\frac{2\majorantEnormeXepsquatre}{\omega_n\alpha}\right)^{1/2}\left( \frac{M_1}{\ln(n)^{1/(1+\rho)}} + \majorantEnormeXquatreunplusrho^{1/(1+\rho)}\right) \nonumber \\
    & \qquad \qquad + \sqrt{M_2}(1+\ln(n)^{-1/2}) \Bigg\}.
\end{align}
The term $D_4$ satisfies
\begin{align}
    \label{eq:up_bound_D4_1}
    D_4 & \leq \frac{2\minlambdaEXXt^{-1}\sqrt{2\majorantEXXprimeconcentration/(n\omega_n\alpha)}}{1-\sqrt{2\majorantEXXprimeconcentration/(n\omega_n\alpha)}} \times \norme{S^{-1}}\frac{1}{n}\sum_{i=1}^n\norme{X_i\epshat_i}^2 \nonumber \\
    & \leq \frac{4\minlambdaEXXt^{-1}\sqrt{2\majorantEXXprimeconcentration/(n\omega_n\alpha)}}{1-\sqrt{2\majorantEXXprimeconcentration/(n\omega_n\alpha)}} \times \norme{S^{-1}}\left\{\norme{\betahat-\beta_0}^2\frac{1}{n}\sum_{i=1}^n\norme{X_i}^4 + \frac{1}{n}\sum_{i=1}^n\norme{X_i\varepsilon_i}^2 \right\}.
\end{align}
Recall the definition of $\widetilde{S}$ in the proof of Lemma~\ref{lem:OLS_Edg}: $\widetilde{S} = n^{-1}\sum_{i=1}^n\widetilde{X}_i\widetilde{X}_i'$.
We can write 
\begin{equation}
    \label{eq:up_bound_D4_2}
    \norme{S^{-1}} = \norme{(\E[XX']^{1/2}\widetilde{S}\E[XX']^{1/2})^{-1}} \leq \minlambdaEXXt^{-1}\norme{\widetilde{S}^{-1}} = \minlambdaEXXt^{-1}\minvp(\widetilde{S})^{-1}.
\end{equation}
In the proof of Lemma~\ref{lem:OLS_Edg}, it is also shown that on $\mathcal{A}$ (and therefore on $\mathcal{C}$) whenever $n \omega > 2 \majorantEXXprimeconcentration / \alpha$, we have
\begin{equation*}
    \minvp(\widetilde{S})
    \geq 1 - \sqrt{\frac{2\majorantEXXprimeconcentration}{n\omega_n\alpha}}
    > 0.
\end{equation*}
By definition, $n \omega > 2 \majorantEXXprimeconcentration / \alpha$ for every $n \geq n_0$. Since the quantity $2\majorantEXXprimeconcentration/(n\omega_n\alpha)$ goes to 0 as $n$ goes to infinity (because of the rate restrictions on $\omega_n$), we can further claim there exists $n_1 \geq n_2$ such that for every $n \geq n_1$ and every $\theta \in \ThetaBEEstrict$, 
\begin{equation}
    \label{eq:up_bound_D4_3}
    \minvp(\widetilde{S}) \geq \frac{1}{2}
\end{equation}
on $\mathcal{C}$. Combining~\eqref{eq:up_bound_D2_2},~\eqref{eq:up_bound_D4_1},~\eqref{eq:up_bound_D4_2} and~\eqref{eq:up_bound_D4_3}, we can claim
\begin{align}
    \label{eq:up_bound_D4}
    D_4 & \leq \overline{D}_4 := \frac{8\minlambdaEXXt^{-2}\sqrt{2\majorantEXXprimeconcentration/(n\omega_n\alpha)}}{1-\sqrt{2\majorantEXXprimeconcentration/(n\omega_n\alpha)}} \nonumber \\
    & \qquad \times 
    \Bigg\{ \minlambdaEXXt^{-1} \left\{ \frac{\sqrt{2\majorantEXXprimeconcentration/(n\omega_n\alpha)}}{1-\sqrt{2\majorantEXXprimeconcentration/(n\omega_n\alpha)}} + 1 \right\}^2\frac{2}{n}\left(\frac{2\majorantEnormeXepsquatre}{\omega_n\alpha}\right)^{1/2}\left( \frac{M_1}{\ln(n)^{1/(1+\rho)}} + \majorantEnormeXquatreunplusrho^{1/(1+\rho)}\right) \nonumber \\
    & \qquad \qquad + \sqrt{M_2}(1+\ln(n)^{-1/2}) \Bigg\}.
\end{align}
We can now collect~\eqref{eq:up_bound_D1},~\eqref{eq:up_bound_D2},~\eqref{eq:up_bound_D3} and~\eqref{eq:up_bound_D4} and claim
\begin{equation*}
    \RnVar(\omega_n\alpha/2) \leq \overline{D}_1 + \overline{D}_2 + \overline{D}_3 + \overline{D}_4.
\end{equation*}
Under the rate restrictions $n\omega_n^{3/2} \to +\infty$ and $\omega_n \to 0$ as $n$ goes to infinity (which imply $n\omega_n \to +\infty$ in particular), we can exhibit a constant $C_1$ that depends only on $\alpha$ and $\ThetaBEEstrict$ and $n_2 \geq n_1$ such that for every $n \geq n_2$ on $\mathcal{C}$
\begin{equation*}
    \RnVar(\omega_n\alpha/2) \leq C_1(n\omega_n)^{-1/2}.
\end{equation*}
Starting from~\eqref{eq:up_bound_ratio_var}, we can also exhibit a constant $C_2 \geq C_1$ that depends on the same quantities as $C_1$ such that for every $n \geq n_2$ on $\mathcal{C}$
\begin{align}
    \label{eq:intermed_unif_exact_6}
    \sqrt{\frac{u' \Vhat u
    + \|u\|^2 \RnVar(\omega_n\alpha/2)}{u'Vu}} & \leq \sqrt{1 + \frac{\majorantEnormeXquatreunplusrho^{1/(1+\rho)}}{||u||^2\minlambdaEXXt }\left(\frac{\|\uu\|^2\minlambdaEXXt^{-1}\sqrt{\majorantEnormeXepsquatre}}{\sqrt{n\omega_n\alpha}} + 2\|u\|^2 \RnVar(\omega_n\alpha/2) \right) } \nonumber \\
    & \leq \sqrt{1 + C_2(n\omega_n)^{-1/2} }.
\end{align}

\medskip

With~\eqref{eq:intermed_unif_exact_4},~\eqref{eq:intermed_unif_exact_5} and~\eqref{eq:intermed_unif_exact_6} in hand, we conclude that for every $n \geq n_2$ and every $\theta \in \ThetaBEEstrict$, \eqref{eq:intermed_unif_exact_3} holds on $\mathcal{C}$ with
\begin{align}
    \label{eq:value_mu_n}
    \mu_n & = \QuantileGaussAt{1-\alpha/2}(n\omega_n)^{-1/4}\sqrt{C_2} + \frac{2\majorantEnormeXquatreunplusrho^{1/2(1+\rho)}\RnLin(\omega_n\alpha/2)}{\sqrt{||u||^2\minlambdaEXXt }} \nonumber \\
    & + \left|\sqrt{a_n} \, q_{\mathcal{N}(0,1)}
    \big(1 - \alpha/2 + \nuEdg \big) - \QuantileGaussAt{1-\alpha/2} \right| \times \sqrt{1+C_2(n\omega_n)^{-1/2}}.
\end{align}
To shed light on the rate at which $\mu_n$ goes to zero, we want to expand $$\left|\sqrt{a_n} \, q_{\mathcal{N}(0,1)}
    \big(1 - \alpha/2 + \nuEdg \big) - \QuantileGaussAt{1-\alpha/2} \right|.$$
There exists $n_3 \geq n_2$ such that for every $n\geq n_3$, $\nuEdg \leq \alpha/4$. On the interval $[0,\alpha/4]$, the map $x \mapsto \QuantileGaussAt{1-\alpha/2+x}$ is continuously differentiable with derivative bounded by $K_\alpha := \sup_{x\in [0,\alpha/4]}\varphi^{-1}\left(\QuantileGaussAt{1-\alpha/2+x}\right)$. We thus have that $\mu_n$ is upper bounded by
\begin{align*}
    \widetilde{\mu}_n & := \QuantileGaussAt{1-\alpha/2}(n\omega_n)^{-1/4}\sqrt{C_2} + \frac{2\majorantEnormeXquatreunplusrho^{1/2(1+\rho)}\RnLin(\omega_n\alpha/2)}{\sqrt{||u||^2\minlambdaEXXt }} \\
    & + \left((\sqrt{a_n}-1)\QuantileGaussAt{1-\alpha/2} + \sqrt{a_n}K_\alpha\nuEdg\right) \times \sqrt{1+C_2(n\omega_n)^{-1/2}}.
\end{align*}
We can exhibit $C_3$ and $n_4$ such that for every $n \geq n_4$ (we remark that $\RnLin(\omega_n\alpha/2) = O(n^{-1/2}\omega_n^{-3/4})$),
\begin{equation}
    \label{eq:upper_bound_mu_n}
    \widetilde{\mu}_n \leq C_3\left(\sqrt{b_n} + \nuEdg + \frac{1}{(n\omega_n)^{1/4}} + \frac{1}{\sqrt{n}\omega_n^{3/4}}\right).
\end{equation}

\paragraph{Step 3. Control of the distance to the asymptotic normality and conclusion.} Equations~\eqref{eq:intermed_unif_exact_1},~\eqref{eq:intermed_unif_exact_2},~\eqref{eq:value_mu_n} and~\eqref{eq:upper_bound_mu_n} enable us to write for every $n \geq n_4$ and every $\theta \in \ThetaBEEstrict$
\begin{align}
    \label{eq:intermed_unif_exact_7}
    & \Probsousthetan\Big( \uu' \betazero \in \CIEdgu \Big) \nonumber \\
    & \leq \Probsousthetan\left( 
    \frac{\sqrt{n}\left|\MeanXi\right|}{\sqrt{u'Vu}} \leq \QuantileGaussAt{1-\alpha/2} + \widetilde{\mu}_n \right) + \ln(n)\left( n^{-\rho}\Indicator\{\rho \leq 1\} + n^{-(1+\rho)/2}\Indicator\{\rho > 1\} + 2n^{-1} \right) \nonumber \\
    & \leq \left| \Probsousthetan\left( 
    \frac{\sqrt{n}\left|\MeanXi\right|}{\sqrt{u'Vu}} > \QuantileGaussAt{1-\alpha/2} + \widetilde{\mu}_n \right) - \alpha \right| \nonumber \\
    & + 1 - \alpha + \ln(n)\left( n^{-\rho}\Indicator\{\rho \leq 1\} + n^{-(1+\rho)/2}\Indicator\{\rho > 1\} + n^{-1} \right) + 2\omega_n\alpha.
\end{align}
We apply Lemma~\ref{lemma:bound_cdf_alpha_Phi}
with $A = \sqrt{n}\MeanXi / \sqrt{u' V u}$
and $x = \QuantileGaussAt{1 - \dfrac{\alpha}{2} + \delta_n} + \widetilde{\mu}_n$.
Therefore, for any $\theta \in \ThetaBEEstrict$,
\begin{align}
    \label{eq:intermed_unif_exact_8}
    \left| \Probsousthetan\left( 
    \frac{\sqrt{n}\left|\MeanXi\right|}{\sqrt{u'Vu}} > \QuantileGaussAt{1-\alpha/2} + \widetilde{\mu}_n \right) - \alpha \right|
    & \leq  2 \min\{\DeltanB,\DeltanE\} \nonumber \\
    & + \left| 2 \Phi\left(-\QuantileGaussAt{1-\alpha/2} - \widetilde{\mu}_n \right) - \alpha \right|.
\end{align}
Combining~\eqref{eq:target_unif_exact},~\eqref{eq:intermed_unif_exact_1},~\eqref{eq:intermed_unif_exact_7} and~\eqref{eq:intermed_unif_exact_8}, we can write for every $n \geq n_4$ and every $\theta \in \ThetaBEEstrict$
\begin{align*}
    1 - \alpha \leq \Probsousthetan\Big( \uu' \betazero \in \CIEdgu \Big)
    & \leq 1 - \alpha \\
    & + 2\left(\omega_n\alpha + \min\{\DeltanB,\DeltanE\}\right) \\
    & + \left| 2 \Phi\left(-\QuantileGaussAt{1-\alpha/2} - \widetilde{\mu}_n \right) - \alpha \right| \\
    & + \ln(n)\left( n^{-\rho}\Indicator\{\rho \leq 1\} + n^{-(1+\rho)/2}\Indicator\{\rho > 1\} + n^{-1} \right).
\end{align*}
Therefore,
\begin{align*}
    & \left| \Probsousthetan\Big( \uu' \betazero \in \CIEdgu \Big) - (1-\alpha) \right| \\
    & \leq 2\left(\omega_n\alpha + \min\{\DeltanB,\DeltanE\}\right) 
    + \left| 2 \Phi\left(-\QuantileGaussAt{1-\alpha/2} - \widetilde{\mu}_n \right) - \alpha \right| \\
    & + \ln(n)\left( n^{-\rho}\Indicator\{\rho \leq 1\} + n^{-(1+\rho)/2}\Indicator\{\rho > 1\} + n^{-1} \right).
\end{align*}
The term $\min\{\DeltanB,\DeltanE\}$ is bounded by $\delta_n$ by assumption. Using a first-order Taylor expansion and the boundedness of the density of the $\mathcal{N}(0,1)$ distribution (denoted by $\varphi$), the third term on the right-hand side is upper-bounded by 
\begin{equation*}
  2\widetilde{\mu}_n||\varphi||_\infty \leq 2C_3||\varphi||_\infty\left(\sqrt{b_n} + \nuEdg + (n\omega_n)^{-1/4} + n^{-1/2}\omega_n^{-3/4}\right).  
\end{equation*}

All in all, we get for every $n \geq n_4$ and $\theta\in\ThetaBEEstrict$
\begin{align*}
    & \left| \Probsousthetan\Big( \uu' \betazero \in \CIEdgu \Big) - (1-\alpha) \right| \\
    & \leq 2\left(\omega_n\alpha + \delta_n\right)
    + 2C_3||\varphi||_\infty\left(\sqrt{b_n} + \nuEdg + \frac{1}{(n\omega_n)^{1/4}} + \frac{1}{\sqrt{n}\omega_n^{3/4}}\right) \\
    & + \ln(n)\left( n^{-\rho}\Indicator\{\rho \leq 1\} + n^{-(1+\rho)/2}\Indicator\{\rho > 1\} + n^{-1} \right).    
\end{align*}
By definition of $\nuEdg$ (and because $\omega_n\to 0$), there exists $C_4$ such that
\begin{align*}
    & 2\left(\omega_n\alpha + \delta_n\right)
    + 2C_3||\varphi||_\infty\left(\sqrt{b_n} + \nuEdg + \frac{1}{(n\omega_n)^{1/4}} + \frac{1}{\sqrt{n}\omega_n^{3/4}}\right) + \frac{\ln(n)}{n} \\
    &\leq C_4\left\{\sqrt{b_n} + \nuEdg + \frac{1}{(n\omega_n)^{1/4}} + \frac{1}{\sqrt{n}\omega_n^{3/4}}\right\}.
\end{align*}
Since the bound with $\rho=0$ (proved just below) is still valid when $\rho > 0$, we can take the minimum of the two bounds as our final control on 
\begin{equation*}
    \left| \Probsousthetan\Big( \uu' \betazero \in \CIEdgu \Big) - (1-\alpha) \right|
\end{equation*}
when $\rho >0$. This concludes the proof, setting $n^*$ equal to $n_4$. $\Box$

\subsubsection{Proof when \texorpdfstring{$\rho=0$}{rho = 0}} The only change is in the definition of the probabilistic event $\mathcal{B}$ which becomes
\begin{align*}
    \mathcal{B} & := \left\{ \frac{1}{n}\sum_{i=1}^n||X_i||^4 \leq \frac{\majorantEnormeXquatreunplusrho}{\omega_n\alpha}\right\} \cap \left\{\left|\frac{1}{n}\sum_{i=1}^n||X_i\varepsilon_i||^2 - \E\left[||X\varepsilon||^2\right]\right|\leq \sqrt{\frac{M_2}{\ln (n)}}\right\} \\
    & \cap \left\{\left|\frac{1}{n}\sum_{i=1}^n\left(u'\E[XX']^{-1}X_i\varepsilon_i\right)^2-\E\left[\left(u'\E[XX']^{-1}X\varepsilon\right)^2\right]\right| \leq \frac{\|\uu\|^2\minlambdaEXXt^{-1}\sqrt{\majorantEnormeXepsquatre}}{\sqrt{n\omega_n\alpha}}\right\}.    
\end{align*}
This modification implies several changes further down the proof. For instance, Equation~\eqref{eq:intermed_unif_exact_2} is replaced with
\begin{align*}
    & \Probsousthetan\left( \left\{ \left|u'(\betahat-\betazero)\right| \leq \frac{\QnEdg}{\sqrt{n}} \sqrt{u' \Vhat u
    + \|u\|^2 \RnVar(\omega_n\alpha/2)} \right\} \cap \mathcal{A} \right) \nonumber \\
    & \leq \Probsousthetan\left( \left\{ \frac{\sqrt{n}\left|u'(\betahat-\betazero)\right|}{\sqrt{u'Vu}} \leq \QnEdg\sqrt{\frac{u' \Vhat u
    + \|u\|^2 \RnVar(\omega_n\alpha/2)}{u'Vu}} \right\} \cap \mathcal{A} \cap \mathcal{B} \right) \nonumber \\
    & + \frac{\ln(n)}{n} + 2\omega_n\alpha.
\end{align*}
When upper-bounding $\RnVar(\omega_n\alpha/2)$, we get
\begin{equation*}
    \RnVar(\omega_n\alpha/2) \leq C_1 n^{-1/2}\omega_n^{-3/4} 
\end{equation*}
for some positive constant $C_1$ that need not coincide with $C_1$ in the proof with $\rho>0$. The condition $n\omega^{3/2} \to +\infty$ ensures $\RnVar(\omega_n\alpha/2) = o(1)$. 
We then change Equation~\eqref{eq:upper_bound_mu_n} to
\begin{equation*}
    \widetilde{\mu}_n \leq C_3\left(\sqrt{b_n} + \nuEdg + n^{-1/4}\omega_n^{-3/8}\right).
\end{equation*}
We also have for every $\theta \in \ThetaBEEstrict$
\begin{align*}
    & \left| \Probsousthetan\Big( \uu' \betazero \in \CIEdgu \Big) - (1-\alpha) \right| \\
    & \leq 3\omega_n\alpha + 2\delta_n
    + 2C_3||\varphi||_\infty\left(\sqrt{b_n} + \nuEdg + n^{-1/4}\omega_n^{-3/8}\right) + \frac{\ln(n)}{n}.  
\end{align*}
Finally we can write that there exists $C_4>0$ such that 
\begin{align*}
    & 3\omega_n\alpha + 2\delta_n
    + 2C_3||\varphi||_\infty\left(\sqrt{b_n} + \nuEdg + n^{-1/4}\omega_n^{-3/8}\right) + \frac{\ln(n)}{n} \\
    & \leq 
    C_4\left\{\sqrt{b_n} + \nuEdg + n^{-1/4}\omega_n^{-3/8}\right\},
\end{align*}
which concludes the proof. $\Box$

\subsubsection{Proof when \texorpdfstring{$\rho=+\infty$}{rho = infty}}

There are two notable changes in the proof (compared to the case $\rho>0$). First, the definition of the probabilistic event $\mathcal{B}$ is modified and becomes
\begin{align*}
    \mathcal{B} & := \left\{\left|\frac{1}{n}\sum_{i=1}^n||X_i\varepsilon_i||^2 - \E\left[||X\varepsilon||^2\right]\right|\leq \sqrt{\frac{M_2}{\ln (n)}}\right\} \\
    & \cap \left\{\left|\frac{1}{n}\sum_{i=1}^n\left(u'\E[XX']^{-1}X_i\varepsilon_i\right)^2-\E\left[\left(u'\E[XX']^{-1}X\varepsilon\right)^2\right]\right| \leq \frac{\|\uu\|^2\minlambdaEXXt^{-1}\sqrt{\majorantEnormeXepsquatre}}{\sqrt{n\omega_n\alpha}}\right\}.    
\end{align*}

With this modification, Equation~\eqref{eq:intermed_unif_exact_2} is replaced with
\begin{align*}
    & \Probsousthetan\left( \left\{ \left|u'(\betahat-\betazero)\right| \leq \frac{\QnEdg}{\sqrt{n}} \sqrt{u' \Vhat u
    + \|u\|^2 \RnVar(\omega_n\alpha/2)} \right\} \cap \mathcal{A} \right) \nonumber \\
    & \leq \Probsousthetan\left( \left\{ \frac{\sqrt{n}\left|u'(\betahat-\betazero)\right|}{\sqrt{u'Vu}} \leq \QnEdg\sqrt{\frac{u' \Vhat u
    + \|u\|^2 \RnVar(\omega_n\alpha/2)}{u'Vu}} \right\} \cap \mathcal{A} \cap \mathcal{B} \right) \nonumber \\
    & + \frac{\ln(n)}{n} + \omega_n\alpha.
\end{align*}

\medskip

Second, the upper bound on $\maxvp\left(\E[XX']\right)$ is modified as well and we have
\begin{equation*}
    \maxvp\left(\E[XX']\right) = ||\E[XX']|| \leq \E\left[||X||^2\right] \leq \majorantEnormeXquatreunplusrho^2. 
\end{equation*}
On the other hand, the bounds on $\RnVar(\omega_n\alpha/2)$ and $\widetilde{\mu}_n$ remain of the same order so that we get
\begin{equation*}
    \RnVar(\omega_n\alpha/2) \leq C_1 (n\omega_n)^{-1/2},
\end{equation*}
and
\begin{equation*}
    \widetilde{\mu}_n \leq C_3\left(\sqrt{b_n} + \nuEdg + \frac{1}{(n\omega_n)^{1/4}} + \frac{1}{\sqrt{n}\omega_n^{3/4}}\right),
\end{equation*}
for some positive constants $C_1$ and $C_3$ that need not coincide with $C_1$ and $C_3$ in the proof with $\rho>0$. The condition $n\omega^{3/2} \to +\infty$ ensures $\RnVar(\omega_n\alpha/2) = o(1)$. 

\medskip

We also have for every $\theta \in \ThetaBEEstrict$
\begin{align*}
    & \left| \Probsousthetan\Big( \uu' \betazero \in \CIEdgu \Big) - (1-\alpha) \right| \\
    & \leq 2(\omega_n\alpha + \delta_n)
    + 2C_3||\varphi||_\infty\left(\sqrt{b_n} + \nuEdg + (n\omega_n)^{-1/4} + n^{-1/2}\omega_n^{-3/4}\right) + \frac{\ln(n)}{n}.  
\end{align*}
Finally we can write that there exists $C_4>0$ such that 
\begin{align*}
    & 2(\omega_n\alpha + \delta_n)
    + 2C_3||\varphi||_\infty\left(\sqrt{b_n} + \nuEdg + (n\omega_n)^{-1/4} + n^{-1/2}\omega_n^{-3/4}\right) + \frac{\ln(n)}{n} \\
    & \leq
    C_4\left\{\sqrt{b_n} + \nuEdg + (n\omega_n)^{-1/4} + n^{-1/2}\omega_n^{-3/4}\right\},
\end{align*}
which concludes the proof. $\Box$

\subsection{Proof of Proposition~\ref{prop:exogenous_case_exact_rates}}
\label{proof:prop:exogenous_case_exact_rates}

By Theorem~\ref{thm:exogenous_case_asymptotic_uniform_exactness}, we can write
\begin{equation*}
    \sup_{\theta \in \ThetaBEEstrict}
        \Probsousthetan \Big( \uu' \betazero \in \CIEdgu \Big)
        \leq 1 - \alpha
        + C\left\{\sqrt{b_n} + \nuEdg + \mathfrak{e}_n \right\}.
\end{equation*}
Remember that we consider tuning parameters of the form
$\omega_n = n^{-a}$ and $b_n = n^{-b}$ for some positive reals $a$ and $b$ to be chosen in a smart way.
Indeed, our objective is the following: given~\(\rho\) (a parameter of the regularity of our class of distributions), we seek \(a\) and \(b\) to obtain the optimal (i.e. tending to~0 as fast as possible) rate for the quantity \(\sqrt{b_n} + \nuEdg + \mathfrak{e}_n\).
To do so, we obtain asymptotic equivalents of those terms as negative powers of~\(n\), and then optimize in~\(a\) and~\(b\).

\paragraph{Constraints on \(a\) and \(b\).}

Note that the rate restrictions imposed on $\omega_n$ and $b_n$ have an impact on the range of admissible values for $a$ and $b$.
The constraints on $a$ and $b$ are the following.
First, $\omega_n$ and $b_n$ both tend to $0$, so $a$ and $b$ must be positive.
Second, \(n^{1 - 3a/2} = n \omega_n^{3/2} \to +\infty\), so $1 - 3a/2 > 0$,
meaning that
$a < 2/3$.
Third, \(n^{1/2 - b} = b_n \sqrt{n} \to +\infty\), so $1/2 - b > 0$,
meaning that $b < 1/2$.

\paragraph{Bound on \(\sqrt{b_n} + \nuEdg + \mathfrak{e}_n\) decomposed between terms dependent on~\(a\) and those on~\(b\).}

By Remark \ref{rem:valid_delta_n}, we can choose $\delta_n \lesssim 1/\sqrt{n}$. Therefore
\begin{align*}
    \nuEdg
    &\lesssim \omega_n
    + \exp\!\big(-n(1-1/a_n)^2/(2 \majorantKurtxi)\big)
    + 1/\sqrt{n} \\
    &\lesssim n^{-a}
    + \exp\!\big(- n (2 \majorantKurtxi)^{-1}
    (1-1/(1+b_n))^2\big)
    + 1/\sqrt{n} \\
    &\lesssim n^{-a}
    + \exp\!\bigg(-n (2 \majorantKurtxi)^{-1}
    \Big(1 - \frac{1}{1+n^{-b}} \Big)^2 \bigg)
    + 1/\sqrt{n} \\
    &\lesssim n^{-a}
    + \exp\!\Big(-n (2 \majorantKurtxi)^{-1}
    (n^{-2b} + O(n^{-4b}) \Big)
    + 1/\sqrt{n}.
\end{align*}

In the mean time, $\mathfrak{e}_n$ depends on $\omega_n$ (hence on $a$) but not on $b_n$ (hence not on $b$). Thus, we obtain as a bound for our quantity of interest
\begin{equation*}
    \sqrt{b_n} + \nuEdg + \mathfrak{e}_n
    \lesssim
    1 / \sqrt{n}
    +
    \underbrace{n^{-a} + \mathfrak{e}_n}_{
    \displaystyle =: \widetilde{\mathfrak{a}}_n}
    + 
    \underbrace{\exp\!\Big(-n (2 \majorantKurtxi)^{-1}
    (n^{-2b} + O(n^{-4b}) \Big)
    + n^{-b/2}}_{\displaystyle =: \widetilde{\mathfrak{b}}_n},
\end{equation*}
where \(\widetilde{\mathfrak{a}}_n\) depends on the tuning parameter~\(a\) but not on b, and the reverse for \(\widetilde{\mathfrak{b}}_n\).
We thus analyze those terms separately.

\paragraph{Study of $\widetilde{\mathfrak{a}}_n
:= n^{-a} + \mathfrak{e}_n$ and choice of~\(a\).}

As stated in Theorem~\ref{thm:exogenous_case_asymptotic_uniform_exactness}, the expression of \(\mathfrak{e}_n\) depends on the value of~\(\rho\).
Therefore, we distinguish the following three cases.

\medskip

\textbf{Case \(\rho = 0\).}
The rate of convergence of $\widetilde{\mathfrak{a}}_n$ is then
$n^{-a} + n^{-1/4 + 3a/8}$.
In that situation, the best choice for \(a\) is such that 
$-a = -1/4 + 3a/8$
that is,
$a = 2/11$.
For that choice, \(\widetilde{\mathfrak{a}} \lesssim n^{-2/11}\).

\medskip

\textbf{Case \(\rho = \infty\).}
When \(\rho = +\infty\), the rate of convergence of \(\widetilde{\mathfrak{a}}_n\) is \(n^{-a} + n^{(a - 1)/4} + n^{3a/4 - 1/2}\).
As in the case \(\rho = 0\), we want to choose the tuning parameter \(a\) to obtain the fastest rate of convergence.
For that, it is enough to consider the situations where we equalize two out of the three exponents, namely
\begin{itemize}

\item 
\(-a = (a-1)/4 \iff a = 1/5\), which yields 
\(\widetilde{\mathfrak{a}}_n = n^{-1/5} + n^{-1/5} + n^{-7/20} \lesssim n^{-1/5}\).

\item
\(-a = 3a/4 - 1/2 \iff a = 2/7\), which yields
\(\widetilde{\mathfrak{a}}_n = n^{-2/7} + n^{-5/28} + n^{-2/7} \lesssim n^{-5/28}
\)

\item
\((a-1)/4 = 3a/4 - 1/2 \iff a = 1/2\), which yields
\(\widetilde{\mathfrak{a}}_n = n^{-1/2} + n^{-1/8} + n^{-1/8} \lesssim n^{-1/8}
\).

\end{itemize}
Overall, the fastest rate is attained at \(a = 1/5\) and, in that case, we have
\(\widetilde{\mathfrak{a}}_n \lesssim n^{-1/5}\).

\medskip

\textbf{Case \(\rho \in (0,\infty)\).}
Remember that, in this case, we have 
\begin{equation*}
    \mathfrak{e}_n
    =
    \min\left\{ n^{-1/4}\omega_n^{-3/8} \, , \, 
            \frac{1}{(n\omega_n)^{1/4}}
            + \frac{1}{\sqrt{n}\omega_n^{3/4}}
            + \ln(n)\left(
            \frac{\Indicator\{\rho \leq 1\}}{n^{\rho}}
            + \frac{\Indicator\{\rho > 1\}}{n^{(1+\rho)/2}} 
            \right) \right\}    
\end{equation*}
We introduce the shortcut notation $\widetilde{\rho} := 
\rho \Indicator\{\rho \leq 1\} 
+ \Indicator\{\rho > 1\} (1+\rho)/2$.
Using that notation and~\(\omega_n = n^{-b}\), we can write
\begin{align*}
    \widetilde{\mathfrak{a}}_n
    &\asymp \min \{
    n^{-a} + n^{-1/4 + 3a/8} \, , \,
    n^{-a} + n^{-1/4 + a/4} + n^{-1/2 + 3a/4}
    + \ln(n) n^{- \widetilde{\rho}} \}.
\end{align*}

As in the previous two cases, we want to choose~\(a\) so as obtain the fastest rate.
The additional layer of complexity compared to the case \(\rho = \infty\) is the minimum, but we follow the same logic: for both quantities within the minimum, we compute the intersection points of the lines of the exponents of~\(n\) that depends on~\(a\).
Doing so, we obtain:
\begin{itemize}
    \item $-a < -1/4 + 3a/8$
    if and only if $1/4 < a(1 + 3/8) = 11 a/8$    
    if and only if $a > 2/11$;
    
    \item $-a < -1/4 + a/4$
    if and only if $1/4 < 5a/4$
    if and only if $a > 1/5$;

    \item $-1/4 + a/4 < -1/2 + 3a/4$
    if and only if $1/4 < 2a/4$
    if and only if $a > 1/2$.
\end{itemize}

Therefore,
\begin{align*}
    \widetilde{\mathfrak{a}}_n \asymp \left\{
    \begin{array}{ll}
        \min \{
        n^{-a} \, , \,
        n^{-a}
        + \ln(n) n^{- \widetilde{\rho}} \}
        & \mbox{if } a \in (0, 2/11], \\ 
        \min \{
        n^{-1/4 + 3a/8} \, , \,
        n^{-a}
        + \ln(n) n^{- \widetilde{\rho}} \}
        & \mbox{if } a \in [2/11, 1/5], \\ 
        \min \{
        n^{-1/4 + 3a/8} \, , \,
        n^{-1/4 + a/4}
        + \ln(n) n^{- \widetilde{\rho}} \}
        & \mbox{if } a \in [1/5, 1/2], \\ 
        \min \{
        n^{-1/4 + 3a/8} \, , \,
        n^{-1/2 + 3a/4}
        + \ln(n) n^{- \widetilde{\rho}} \}
        & \mbox{if } a \in [1/2, 2/3).
    \end{array}
    \right.
\end{align*}

Then, for each possibility, we distinguish two cases depending on whether the term in $\widetilde{\rho}$ is dominant or not in the sum.
We can thus write
\begin{align*}
    \widetilde{\mathfrak{a}}_n \asymp \left\{
    \begin{array}{ll}
        \min \{
        n^{-a} \, , \,
        n^{-a} \}
        & \mbox{if } a \in (0, 2/11]
        \mbox{ and } \widetilde{\rho} > a, \\ 
        \min \{
        n^{-a} \, , \,
        \ln(n) n^{- \widetilde{\rho}} \}
        & \mbox{if } a \in (0, 2/11]
        \mbox{ and } \widetilde{\rho} \leq a, \\ 
        \min \{
        n^{-1/4 + 3a/8} \, , \,
        n^{-a} \}
        & \mbox{if } a \in [2/11, 1/5]
        \mbox{ and } \widetilde{\rho} > a, \\ 
        \min \{
        n^{-1/4 + 3a/8} \, , \,
        \ln(n) n^{- \widetilde{\rho}} \}
        & \mbox{if } a \in [2/11, 1/5]
        \mbox{ and } \widetilde{\rho} \leq a, \\ 
        \min \{
        n^{-1/4 + 3a/8} \, , \,
        n^{-1/4 + a/4} \}
        & \mbox{if } a \in [1/5, 1/2]
        \mbox{ and } \widetilde{\rho} > 1/4 - a/4, \\ 
        \min \{
        n^{-1/4 + 3a/8} \, , \,
        \ln(n) n^{- \widetilde{\rho}} \}
        & \mbox{if } a \in [1/5, 1/2]
        \mbox{ and } \widetilde{\rho} \leq 1/4 - a/4, \\ 
        \min \{
        n^{-1/4 + 3a/8} \, , \,
        n^{-1/2 + 3a/4} \}
        & \mbox{if } a \in [1/2, 2/3)
        \mbox{ and } \widetilde{\rho} > 1/2 - 3a/4, \\ 
        \min \{
        n^{-1/4 + 3a/8} \, , \,
        \ln(n) n^{- \widetilde{\rho}} \}
        & \mbox{if } a \in [1/2, 2/3)
        \mbox{ and } \widetilde{\rho} \leq 1/2 - 3a/4. \\ 
    \end{array}
    \right.
\end{align*}

We now simplify these cases, that is, in each case, we determine the minimum as a function of \(a\) and \(\widetilde{\rho}\).
\begin{itemize}
    \item  In the first case, it is straightforward to see that the minimum is $n^{-a}$.
    \item  In the second case, the minimum is $n^{-a}$ too.
    \item  In the third case, $a > 2/11$, so the minimum is $n^{-a}$.
    \item  For the fourth case, we have
    $-1/4 + 3a/8 \leq - \widetilde{\rho}$
    if and only if
    $\widetilde{\rho} \leq 1/4 - 3a/8$.
    Therefore, we divide this case depending on the form of the main term.
    \item  For the fifth case, we have
    $-1/4 + 3a/8 > -1/4 + a/4$
    if and only if
    $3a/8 > a/4$, which is always satisfied.
    \item  For the sixth case, we have
    $-1/4 + 3a/8 \leq - \widetilde{\rho}$
    if and only if
    $\widetilde{\rho} \leq 1/4 - 3a/8$.
    Therefore, we divide this case depending on the form of the main term.
    \item  For the seventh term, we have
    $-1/4 + 3a/8 < -1/2 + 3a/4$
    if and only if
    $1/4 < 3a/8$
    if and only if
    $2/3 < a$.
    \item For the eighth term, we have
    $-1/4 + 3a/8 < - \widetilde{\rho}$
    if and only if
    $\widetilde{\rho} < 1/4 - 3a/8$.
    Note that, for $a < 2/3$, we always have
    $1/4 - 3a/8 < 1/2 - 3a/4$.
\end{itemize}

In the end, we thus obtain:  
\begin{align*}
    \!\!\!\!\!\!\!\!\!\!\!\!\!\!\!\!\!\!\!\!\!
    \widetilde{\mathfrak{a}}_n
    \asymp
    f(\widetilde{\rho}, a) :=
    \left\{
    \begin{array}{ll}
        n^{-a}
        & \mbox{if } a \in (0, 2/11], \\ 
        n^{-a}
        & \mbox{if } a \in [2/11, 1/5]
        \mbox{ and } \widetilde{\rho} > a, \\ 
        \ln(n) n^{- \widetilde{\rho}}
        & \mbox{if } a \in [2/11, 1/5]
        \mbox{ and } 1/4 - 3a/8 < \widetilde{\rho} \leq a 
        \textcolor{black}{
        \mbox{ i.e. }
        a > \max(2/3 - 8\widetilde{\rho}/3 , \widetilde{\rho})
        }
        \\ 
        n^{-1/4 + 3a/8}
        & \mbox{if } a \in [2/11, 1/5]
        %
        \mbox{ and }
        \widetilde{\rho} \leq 1/4 - 3a/8
        \textcolor{black}{
        \mbox{ i.e. }
        a \leq 2/3 - 8\widetilde{\rho}/3 
        }
        \\ 
        n^{-1/4 + a/4}
        & \mbox{if } a \in [1/5, 1/2]
        \mbox{ and } \widetilde{\rho} > 1/4 - a/4, 
        \textcolor{black}{
        \mbox{ i.e. }
        a > 1 - 4\widetilde{\rho}
        }
        \\  
        \ln(n) n^{- \widetilde{\rho}}
        & \mbox{if } a \in [1/5, 1/2]
        \mbox{ and } 1/4 - 3a/8 < \widetilde{\rho} \leq 1/4 - a/4
        \textcolor{black}{
        \mbox{ i.e. }
        2/3 - 8\widetilde{\rho}/3 
        <
        a
        \leq 
        1 - 4\widetilde{\rho}
        }
        \\  
        n^{-1/4 + 3a/8} 
        & \mbox{if } a \in [1/5, 1/2]
        \mbox{ and } \widetilde{\rho} \leq 1/4 - 3a/8
        \textcolor{black}{
        \mbox{ i.e. }
        a \leq 2/3 - 8\widetilde{\rho}/3
        }
        \\ 
        n^{-1/4 + 3a/8}
        & \mbox{if } a \in [1/2, 2/3] 
        \mbox{ and } \widetilde{\rho} > 1/2 - 3a/4, \textcolor{black}{\mbox{ i.e. } a > 2/3 - 4\widetilde{\rho}/3}\\ 
        %
        %
        \ln(n) n^{- \widetilde{\rho}}
        & \mbox{if } a \in [1/2, 2/3)
        \mbox{ and } 1/4 - 3a/8 \leq \widetilde{\rho} \leq 1/2 - 3a/4, \textcolor{black}{\mbox{ i.e. } 2/3 - 8\widetilde{\rho}/3 \leq a \leq 2/3 -4\widetilde{\rho}/3 
        }
        \\ 
        n^{-1/4 + 3a/8}
        & \mbox{if } a \in [1/2, 2/3)
        \mbox{ and } 
        \widetilde{\rho} < 1/4 - 3a/8
        \textcolor{black}{
        \mbox{ i.e. }
        a < 2/3 - 8\widetilde{\rho}/3
        }
    \end{array}
    \right.
\end{align*}
Those regimes are summarized in Figure~\ref{fig:error_regimes} below.
\begin{figure}[H]
    \centering
    \includegraphics[width=0.7\linewidth]{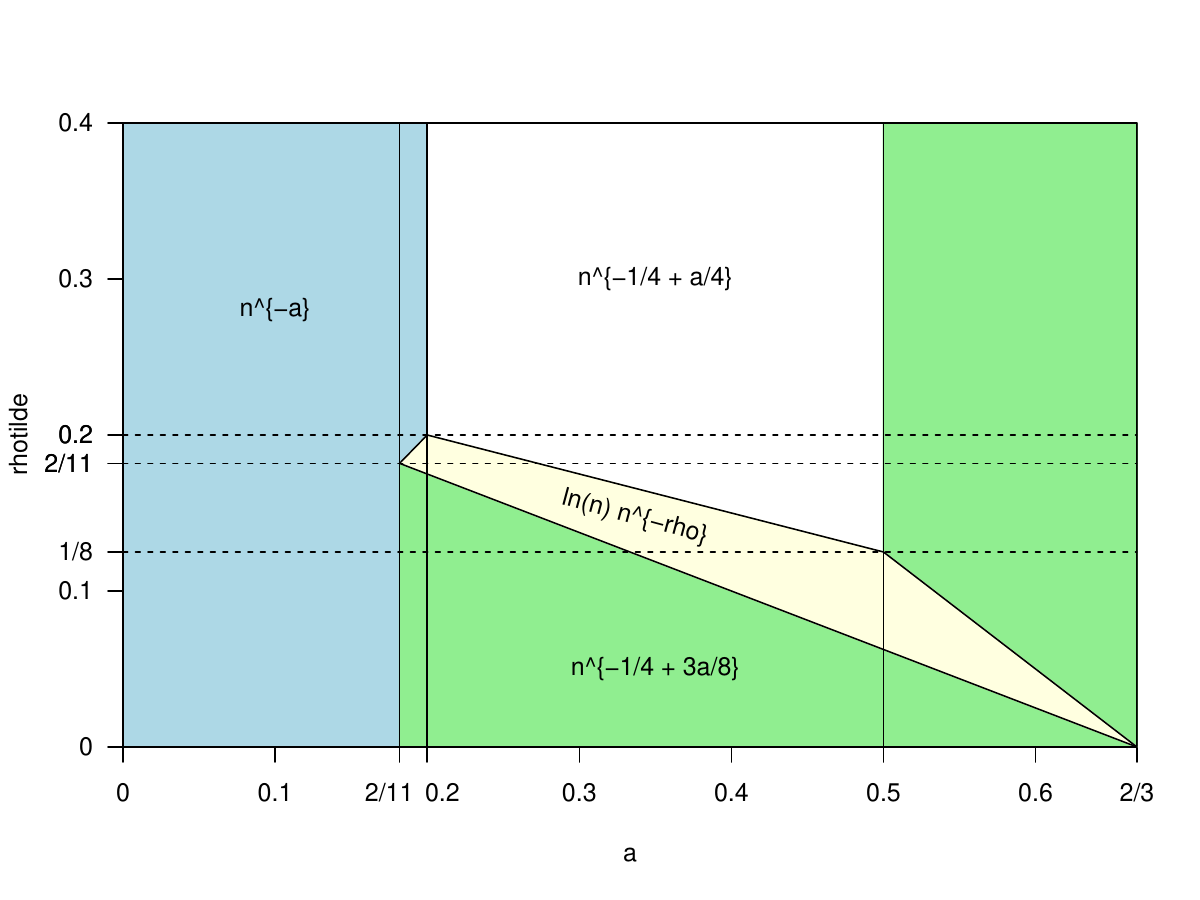}
    \caption{Regimes of $\widetilde{\mathfrak{a}}_n$ as a function of \(a\) and \(\widetilde{\rho}\). 
    Regions of the same color indicate the same expression for the rate.}
    \label{fig:error_regimes}
\end{figure}

For each regime, it remains to pick the value of~\(a\) that yields the fastest rate, that is, given \(\widetilde{\rho}\), to minimize the above-defined function \(a \mapsto f(\widetilde{\rho}, a)\).
We distinguish different cases depending on the value of \(\widetilde{\rho}\).

\begin{itemize}

\item 
For $\widetilde{\rho} > 1/5$, the function does not depend on $\widetilde{\rho}$ anymore, and is just equal to
\begin{align*}
    f(\widetilde{\rho},a) = 
    \left\{
    \begin{array}{ll}
        n^{-a} & \mbox{if } 0 < a \leq 0.2 \\
        n^{-1/4 + a/4} & \mbox{if } 0.2 \leq a \leq 0.5 \\
        n^{-1/4 + 3a/8} & \mbox{if } 0.5 \leq a \leq 2/3.
    \end{array}
    \right.
\end{align*}
This is equivalent to maximizing in $a$ the function
\begin{align*}
    a \mapsto 
    \left\{
    \begin{array}{ll}
        a & \mbox{if } 0 < a \leq 0.2 \\
        1/4 - a/4 & \mbox{if } 0.2 \leq a \leq 0.5 \\
        1/4 - 3a/8 & \mbox{if } 0.5 \leq a \leq 2/3.
    \end{array}
    \right.
\end{align*}
We study separately the three intervals.
The maximums are respectively
$0.2$,
$1/4 - (1/5)/4 = 1/5 = 0.2$
and $1/4 - (3/2)/8 = 1/4 - 3/16 = 1/16 = 0.0625$.
Therefore the optimal value is $a = 0.2$, giving us the rate $n^{-1/5}$.

\medskip

We now consider the other cases when \(\widetilde{\rho} \leq 1/5\).
For these cases, remark that \(\rho = \widetilde{\rho}\) since, by the definition of \(\widetilde{\rho}\), we have \(\widetilde{\rho} = \rho\) for any~\(\rho \leq 1\).
In what follows, we could thus replace \(\widetilde{\rho}\) with \(\rho\).

\item 
In the case $\widetilde{\rho} \in [2/11, 1/5]$ (equivalently, \(\rho \in [2/11, 1/5]\)),
the function to be minimized is:
\begin{align*}
    f(\widetilde{\rho},a) =
    \left\{
    \begin{array}{ll}
        n^{-a} & \mbox{if } 0 < a \leq \widetilde{\rho} \\
        \ln(n) n^{- \widetilde{\rho}}
        & \mbox{if } \widetilde{\rho} \leq a
        \leq 1 - 4\widetilde{\rho} \\
        n^{-1/4 + a/4}
        & \mbox{if }
        1 - 4\widetilde{\rho} \leq a \leq 0.5 \\
        n^{-1/4 + 3a/8} & \mbox{if } 0.5 \leq a \leq 2/3.
    \end{array}
    \right.
\end{align*}
On each interval, the smallest rates are respectively
$n^{- \widetilde{\rho}}$,
$\ln(n) n^{- \widetilde{\rho}}$,
$n^{-1/4 + ( 1 - 4\widetilde{\rho})/4}
= n^{ - \widetilde{\rho}}$
and
$n^{-1/4 + 3 \times (1/2)/8}
= n^{-1/4 + 3/16}
= n^{-1/16}$.
Since, by hypothesis in this case, $\widetilde{\rho} \geq 2/11 > 1/16$,
the optimal rate is then $n^{- \widetilde{\rho}}$ and is obtained for \(a = \widetilde{\rho} = \rho\).

\item  
In the case $\widetilde{\rho} \in [1/8, 2/11]$ (equivalently, \(\rho \in [1/8, 2/11]\)),
the function to be minimized is:
\begin{align*}
    f(\widetilde{\rho},a) =
    \left\{
    \begin{array}{ll}
        n^{-a} & \mbox{if } 0 < a \leq 2/11 \\
        n^{-1/4 + 3a/8}
        & \mbox{if } 2/11 \leq a \leq 2/3 - 8\widetilde{\rho}/3 \\
        \ln(n) n^{- \widetilde{\rho}}
        & \mbox{if } 2/3 - 8\widetilde{\rho}/3 \leq a
        \leq 1 - 4\widetilde{\rho} \\
        n^{-1/4 + a/4}
        & \mbox{if }
        1 - 4\widetilde{\rho} \leq a \leq 0.5 \\
        n^{-1/4 + 3a/8} & \mbox{if } 0.5 \leq a \leq 2/3.
    \end{array}
    \right.
\end{align*}
Again, we study the interval separately.
On each interval, the smallest rates are respectively
$n^{-2/11}$,
$n^{-1/4 + 3 \times (2 / 11) / 8}
= n^{-1/4 + 3/44}
= n^{-11/44 + 3/44}
= n^{-8/44}
= n^{-2/11}$,
$\ln(n) n^{- \widetilde{\rho}}$,
$n^{-1/4 + (1 - 4\widetilde{\rho})/4}
= n^{- \widetilde{\rho}}$,
and
$n^{-1/4 + 3 (1/2)/8}
= n^{-1/4 + 3/16}
= n^{-4/16 + 3/16}
= n^{-1/16}$.
Consequently, the fastest rate is $n^{-2/11}$ obtained when \(a = 2/11\).

\item 

In the case $\widetilde{\rho} \in [0, 1/8]$ (equivalently \(\rho \in [0, 1/8]\)),
the function to be minimized is:
\begin{align*}
    f(\widetilde{\rho},a) =
    \left\{
    \begin{array}{ll}
        n^{-a} & \mbox{if } 0 < a \leq 2/11 \\
        n^{-1/4 + 3a/8}
        & \mbox{if } 2/11 \leq a \leq 2/3 - 8\widetilde{\rho}/3 \\
        \ln(n) n^{- \widetilde{\rho}}
        & \mbox{if } 2/3 - 8\widetilde{\rho}/3 \leq a
        \leq 2/3 - 4\widetilde{\rho}/3 \\
        n^{-1/4 + 3a/8}
        & \mbox{if } 2/3 - 4\widetilde{\rho}/3 \leq a \leq 2/3.
    \end{array}
    \right.
\end{align*}
On each interval, the smallest rates are respectively
$n^{-2/11}$,
$n^{-1/4 + 3 \times (2 / 11) / 8}
= n^{-1/4 + 3 / 44}
= n^{-11/44 + 3 / 44}
= n^{-8/44}
= n^{-2/11}$,
$\ln(n) n^{- \widetilde{\rho}}$,
$n^{-1/4 + 3 \times (2/3 - 4\widetilde{\rho} / 3) /8}
= n^{-1/4 + (2 - 4\widetilde{\rho}) /8}
= n^{-1/4 + 1/4 - \widetilde{\rho}/2}
= n^{- \widetilde{\rho} / 2}
$.
Since $\widetilde{\rho} \leq 1/8$, we obtain
$\widetilde{\rho} / 2 \leq 1/16 < 2/11$.
So the optimal rate is $n^{-2/11}$.
As in the previous case (\(\rho \in [1/8, 2/11]\)), it is obtained for \(a = 2/11\).

\end{itemize}

All in all, the choice of \(a\) yielding the fastest rate for \(\widetilde{\mathfrak{a}}\) as a function of \(\rho\) is thus: 
\begin{align*}
    a
    =
    r(\rho)
    =
    \left\{
    \begin{array}{ll}
        1/5 & \mbox{if } \rho > 1/5 \\    
        \rho & \mbox{if } 2/11 \leq \rho \leq 1/5\\
        2/11
        & \mbox{if } 0 \leq \rho < 2/11.
    \end{array}
    \right.
    =
    \frac{2}{11} \Indicator\{\rho < 2/11\}
    + \rho \Indicator\{2/11 \leq \rho \leq 1/5\}
    + \frac{1}{5} \Indicator\{1/5 < \rho\},
\end{align*}
and the corresponding rate for \(\widetilde{\mathfrak{a}}\) happens to be precisely \(\widetilde{\mathfrak{a}} \lesssim n^{-r(\rho)}\) for any~\(\rho \in (0, +\infty)\).

\medskip

Let us conclude this paragraph devoted to the study of \(\widetilde{\mathfrak{a}}\) by a recap.
Depending on the value of~\(\rho\), the choice of~\(a\) yielding the fastest rate are, respectively
\begin{itemize}

\item if \(\rho = 0\), \(a = 2/11\) yielding \(\widetilde{\mathfrak{a}} \lesssim n^{-2/11}\);

\item if \(\rho \in (0, +\infty)\), \(a = r(\rho)\) yielding 
\(\widetilde{\mathfrak{a}} \lesssim n^{-r(\rho)}\);

\item if \(\rho = +\infty\), \(a = 1/5\) yielding 
\(\widetilde{\mathfrak{a}} \lesssim n^{-1/5}\).

\end{itemize}

Remark that \(r(0) = 2/11\) and \(r(+\infty) = 1/5\).
Therefore, we can summarize the three cases in one expression: for any~\(\rho \in [0, +\infty]\), the choice of~\(a\) that yields the fastest rate is \(a = r(\rho)\) yielding \(\widetilde{\mathfrak{a}} \lesssim n^{-r(\rho)}\).

\paragraph{Study of $\widetilde{\mathfrak{b}}_n$ and choice of~\(b\).}


We now turn to the analysis of the term in the bound of \(\sqrt{b_n} + \nuEdg + \mathfrak{e}_n\) that depends only on the tuning parameter~\(b\) (and not on~\(a\)).
Let us first recall its expression:
\begin{equation*}
    \widetilde{\mathfrak{b}}_n
    :=
    \exp\!\Big(-n (2 \majorantKurtxi)^{-1}
    (n^{-2b} + O(n^{-4b}) \Big)
    + n^{-b/2},
\end{equation*}
and the bound on \(\sqrt{b_n} + \nuEdg + \mathfrak{e}_n\):
\begin{equation*}
    \sqrt{b_n} + \nuEdg + \mathfrak{e}_n
    \lesssim
    n^{-1/2}
    +
    \widetilde{\mathfrak{a}}_n
    +
    \widetilde{\mathfrak{b}}_n.
\end{equation*}

The larger~\(b\) (with the constraint that \(b < 1/2\)), the fastest the rate of \(n^{-b/2}\): a priori, we would like to pick~\(b\) as large as possible.

However, the previous paragraph shows that the rate of
\(\widetilde{\mathfrak{a}}_n\)
is never faster than $n^{-1/5}$.
Hence, similarly to the method followed to choose~\(a\), we want to choose \(b\) so as to equalize the fastest possible rate of \(\widetilde{\mathfrak{a}}_n\) with the rate of \(n^{-b/2}\).
In other words, we solve
\(-1/5 = -b/2\), that is,
\(b = 2/5\).
For \(b = 2/5\)
(and more generally, for any \(b \in (2/5 , 1/2)\)), we have
\begin{equation*}
    \widetilde{\mathfrak{b}}_n
    \lesssim
    n^{-1/5}.
\end{equation*}
Combined with the choice~\(a = r(\rho)\),
we obtain, for \(b = 2/5\),
\begin{align*}
    \sqrt{b_n} + \nuEdg + \mathfrak{e}_n
    & \lesssim
    n^{-1/2}
    + \widetilde{\mathfrak{a}}_n
    + \widetilde{\mathfrak{b}}_n 
    \\ 
    & \lesssim
    n^{-1/2}
    + n^{-1/5}
    + n^{-1/5}
    \\
    & \lesssim
    n^{-1/5}.
\end{align*}

Finally, to conclude the proof, we remark the following.
Fix any~\(\rho \in [0, +\infty]\).
For any choice of \(b \in [2/5, 1/2)\), the term \(n^{-b/2}\) has a faster rate of decrease to zero than the term \(\widetilde{\mathfrak{a}}_n\) when choosing \(a = r(\rho)\) since the rate of \(\widetilde{\mathfrak{a}}_n\) is \(r(\rho) \leq 1/5\).
The latter term is thus dominant.
As a consequence, for any~\(b \in [2/5, 1/2)\), if \(a = r(\rho)\),
\begin{align*}
    \sqrt{b_n} + \nuEdg + \mathfrak{e}_n
    & \lesssim
    n^{-1/2}
    +
    \widetilde{\mathfrak{a}}_n
    +
    \widetilde{\mathfrak{b}}_n 
    \\ 
    & \lesssim
    n^{-1/2}
    +
    n^{-b/2}
    +
    n^{-r(\rho)}
    \\
    & \lesssim
    n^{-r(\rho)}.
\end{align*}

\section{Additional lemmas}
\label{appendix:sec:additional_lemmas}

\begin{lemma}
    \label{lem:concentr_sq_mat}
    Let $n \geq 1$ and $(A_i)_{i=1}^n$ be an \iid{} sequence of random square matrices of dimension $p$ with finite second moment. For every $\gamma\in(0,1)$,
    \begin{align*}
        & \Prob\left( \left|\left| \frac{1}{n}\sum_{i=1}^nA_i - \E[A] \right|\right| > \sqrt{\frac{B}{n\gamma}} \right) < \gamma,
    \end{align*}
    where $B := \E\!\left[ \left|\left| \operatorvec(A-\E[A]) \right|\right|^2 \right].$
\end{lemma}

\medskip

\noindent
{\it Proof of
Lemma~\ref{lem:concentr_sq_mat}: }
By Markov's inequality and the inequality $||M||^2 \leq \sum_{1 \leq l,l' \leq p}(M^{(l,l')})^2$ valid for every square matrix $M,$ we get for every $t>0$
\begin{align*}
    \Prob\left( \left|\left|
    \frac{1}{n}\sum_{i=1}^n A_i - \E[A] \right|\right| > t \right)
    & \leq \frac{1}{t^2}\E\!\left[ \left|\left| \frac{1}{n}\sum_{i=1}^nA_i - \E[A] \right|\right|^2 \right] \\
    & \leq \frac{1}{t^2}\E\!\left[ \sum_{1 \leq l,l' \leq p}\left( \frac{1}{n}\sum_{i=1}^nA_i^{(l,l')} - \E\!\left[A^{(l,l')}\right] \right)^2 \right].
\end{align*}
Since $(A_i)_{i=1}^n$ is an \iid{} sequence, we can write
\begin{align*}
    \E\!\left[ \sum_{1 \leq l,l' \leq p}\left( \frac{1}{n}\sum_{i=1}^nA_i^{(l,l')} - \E\!\left[A^{(l,l')}\right] \right)^2 \right] & = \frac{1}{n}\sum_{1 \leq l,l' \leq p}\E\!\left[\left(A^{(l,l')}-\E\!\left[A^{(l,l')}\right]\right)^2\right] \\
    & = \frac{1}{n}\E\!\left[ \left|\left| \operatorvec(A-\E[A]) \right|\right|^2 \right].
\end{align*}

Choosing $t = \sqrt{\frac{B}{n\gamma}}$ concludes the proof. $\Box$

\medskip

For reader's convenience, we recall the following eigenvalue stability lemma, which is a corollary of Weyl's inequality $\minvp(A+B) \leq \minvp(A) + \maxvp(B)$ for real symmetric matrices, see for example \cite[Section 8.2]{hogben2006handbook}.
\begin{lemma}
    \label{lem:lower_bound_mineigen_sym_mat}
    For $A_1$ and $A_2$ be two symmetric matrices of dimension $d$, we have
    $|\minvp(A_1) - \minvp(A_2)|
    \leq ||A_1-A_2||$.
\end{lemma}

\medskip

\noindent
{\it Proof of
Lemma~\ref{lem:lower_bound_mineigen_sym_mat}: }
If $\minvp(A_1) = \minvp(A_2)$, the results follows directly. 
Without loss of generality, $\minvp(A_1) > \minvp(A_2)$.
Applying Weyl's inequality with $A:= A_2$ and $B:= A_1-A_2$
we get $0 < \minvp(A_1) - \minvp(A_2) \leq \maxvp(A_1 - A_2) \leq ||A_1-A_2||$. $\Box$

%

\begin{lemma}
\label{lem:concentr_inf_function}
Let $n \geq 1$ and \((X_i)_{i=1}^n\) be an \iid{} sequence of random vectors of dimension~$p$ and 
\((\eps_i)_{i=1}^n\) an \iid{} sequence of random variables such that \(\E[X\eps] = 0\) and 
\(\E \big[\norme{X\eps}^4 \big] < +\infty\).
For every $\gamma\in(0,1)$,
\begin{align*}
    & \Prob\left( \bigg|\bigg| \frac{1}{n}\sum_{i=1}^nX_i\eps_i \bigg|\bigg|
    \leq \sqrt{\frac{2}{n}}\left( \frac{\E \big [\norme{X\eps}^4 \big]}{\gamma} \right)^{1/4} \right) \geq 1-\gamma.
\end{align*}
\end{lemma}

\medskip

\noindent
{\it Proof of
Lemma~\ref{lem:concentr_inf_function}: }
By Markov's inequality, we can write
\begin{align}
\label{eq:concentr_inf_1}
    & \Prob\left( \left|\left| \frac{1}{n}\sum_{i=1}^nX_i\eps_i \right|\right| > t \right) \leq \frac{1}{(tn)^4}\E\!\left[ \left|\left| \sum_{i=1}^nX_i\eps_i \right|\right|^4 \right].
\end{align}

We now focus on $\E\!\left[ \left|\left| \sum_{i=1}^nX_i\eps_i \right|\right|^4 \right]$. We have
\begin{align}
\label{eq:concentr_inf_2}
    & \E\!\left[ \left|\left| \sum_{i=1}^nX_i\eps_i \right|\right|^4 \right] = \sum_{1 \leq l,l' \leq d}\sum_{1 \leq i,j,i',j' \leq n} \E\!\left[ X_{il}\eps_iX_{jl}\eps_jX_{i'l'}\eps_{i'}X_{j'l'}\eps_{j'} \right]
\end{align}

The sequence $(X_i\eps_i)_{i=1}^n$ is \iid{} and centered. Consequently,
\begin{align*}
    \E\big[ X_{il} &\eps_i X_{jl} \eps_j X_{i'l'} \eps_{i'} X_{j'l'}\eps_{j'} \big] \\
    &= \E\!\left[ X_l^2X_{l'}^2\eps^4 \right]\mathbbm{1}_{\left\{ i=j=i'=j' \right\}} + \E\!\left[ X_l^2\eps^2 \right]\E\!\left[ X_{l'}^2\eps^2 \right]\mathbbm{1}_{\left\{ i=j, i'=j', i\neq i' \right\}} \\
    & + \E\!\left[ X_lX_{l'}\eps^2 \right]^2\left\{ \mathbbm{1}_{\left\{ i=i', j=j', i\neq j \right\}} + \mathbbm{1}_{\left\{ i=j', j=i', i\neq j \right\}} \right\}.
\end{align*}

Combining this with~\eqref{eq:concentr_inf_2}, we get (using Jensen's inequality on the last line)
\begin{align}
\label{eq:concentr_inf_3}
    & \E\!\left[ \left|\left| \sum_{i=1}^nX_i\eps_i \right|\right|^4 \right] \nonumber \\
    &= \sum_{1 \leq l,l' \leq d} \left\{ n\E\!\left[ X_l^2X_{l'}^2\eps^4 \right] + n(n-1)\E\!\left[ X_l^2\eps^2 \right]\E\!\left[ X_{l'}^2\eps^2 \right] + 2n(n-1)\E\!\left[ X_lX_{l'}\eps^2 \right]^2 \right\} \nonumber \\
    &= n\E\!\left[ ||X\eps||^4 \right] + n(n-1)\E\!\left[||X\eps||^2\right]^2 + 2n(n-1)\sum_{1 \leq l,l' \leq d}\E\!\left[ X_lX_{l'}\eps^2 \right]^2 \nonumber \\
    &\leq 4n^2\E\!\left[ ||X\eps||^4 \right].
\end{align}

We plug~\eqref{eq:concentr_inf_3} back in~\eqref{eq:concentr_inf_1} and choose $t = \sqrt{\frac{2}{n}}\left( \frac{\E[||X\eps||^4]}{\gamma} \right)^{1/4}$ to conclude. $\Box$

\begin{lemma}
    \label{lem:up_bound_moments}
    Under Assumptions~\ref{hyp:subset_exo_non-asymptotic_validity}(i),~\ref{hyp:subset_exo_non-asymptotic_validity}(iii) and~\ref{hyp:subset_exo_unif_asymptotic_exactness}(ii), we have
    \begin{align*}
        \E \big[||X\varepsilon||^4 \big]
        \leq \majorantEnormeXquatreunplusrho^{\frac{1}{1+\rho}}
        \majorantEnormeXepsquatre \text{ and }
        \E\left[ \left(u'\E[XX']^{-1}X \varepsilon\right)^4 \right]
        \leq \frac{||\uu||^4 \majorantEnormeXepsquatre}{\minlambdaEXXt^2}.
    \end{align*}
\end{lemma}

\medskip

\noindent
{\it Proof of Lemma~\ref{lem:up_bound_moments}:} We can write
\begin{equation*}
    \E[\|X\varepsilon\|^4] \leq \|\E[XX']^{1/2}\|^4\E\left[\big|\big|\widetilde{X}\varepsilon\big|\big|^4\right] \leq \maxvp(\E[XX'])^2\majorantEnormeXepsquatre.
\end{equation*}
Using standard properties of the trace of a matrix, we obtain
\begin{equation*}
    \maxvp(\E[XX'])
    \leq \trace(\E[XX']) = \E[\|X\|^2]
    \leq \E\left[\|X\|^{4(1+\rho)}\right]^{\frac{1}{2(1+\rho)}} \leq \majorantEnormeXquatreunplusrho^{\frac{1}{2(1+\rho)}},
\end{equation*}
which is enough to get the first result.

\medskip

For the second result, we first remark that
\begin{equation*}
    \E\left[\left(u'\E[XX']^{-1}X\varepsilon\right)^4\right] \leq ||\uu||^4\big|\big|\E[XX']^{-1/2}\big|\big|^4\E\big[\|\widetilde{X}\varepsilon\|^4\big] \leq \|\uu\|^4\big|\big|\E[XX']^{-1/2}\big|\big|^4\majorantEnormeXepsquatre.
\end{equation*}
To conclude, we use the fact that $\big|\big|\E[XX']^{-1/2}\big|\big|^4 = \minvp\big(\E[XX']\big)^{-2} \leq \minlambdaEXXt^{-2}$. $\Box$

\begin{lemma}
    \label{lem:marcinkiewicz_zygmund}
    Let $n \geq 1$ and \((X_i)_{i=1}^n\) be an \iid{} sequence of random vectors of dimension~$p$ such that \(\E[||X||^{4(1+\rho)}] \leq K\) for some $\rho>0$. There exists a constant $C(\rho)$, depending only on $\rho$ such that
    \begin{equation*}
        \E\left[ \left|\left|
        \frac{1}{n}\sum_{i=1}^n||X_i||^4
        - \E \big[||X||^4 \big]  \right|\right|^{1+\rho} \right]
        \leq 2^{1+\rho} C(\rho) 
        K\left\{ n^{-\rho}
        \Indicator\{\rho \leq 1\}
        + n^{-(1+\rho)/2}
        \Indicator\{\rho > 1\} \right\}
    \end{equation*}
\end{lemma}

\noindent
{\it Proof of Lemma~\ref{lem:marcinkiewicz_zygmund}:} By a direct application of Corollary 8.2 in~\cite{Gut2005} (with $||X||^4-\E[||X||^4]$ instead of $X$ and $1+\rho$ instead of $p$)
\begin{align*}
    & \E\left[ \left|\left| \frac{1}{n}\sum_{i=1}^n||X_i||^4 -\E[||X||^4]  \right|\right|^{1+\rho} \right] \\
    &\qquad \leq C(\rho)\E\left[\left| \, ||X||^4 - \E[||X||^4] \, \right|^{1+\rho}\right]\left\{n^{-\rho}\Indicator\{\rho \leq 1\} + n^{-(1+\rho)/2}\Indicator\{\rho > 1\}\right\}.
\end{align*}
We also use the triangle and Jensen inequality to write
\begin{equation*}
    \E\left[\left| \, ||X||^4 - \E[||X||^4] \, \right|^{1+\rho}\right] \leq 2^\rho\left( \E[||X||^{4(1+\rho)}] + \E[||X||^4]^{1+\rho} \right) \leq 2^{1+\rho}K. \qquad\qquad\qquad\qquad \Box
\end{equation*}

\begin{lemma}[Edgeworth expansion]
    \label{lem:edg_exp}
    Let \(n \geq 1\) 
    and $(\xi_i)_{i=1}^n$ be an \iid{} sequence of random variables
    with $\E[\xi] = 0,$ $\V(\xi) = \sigma^2 > 0$ and $\E\!\left[\xi^4\right]/\sigma^4 \leq \majorantKurtxi$.
    Let \(\overline{\xi}_n := n^{-1} \sumiunn \xi_i\), $\widehat{\sigma}_0^2 := n^{-1}\sum_{i=1}^n \xi_i^2$ and $\sigmanhat^2 := n^{-1}\sum_{i=1}^n\left(\xi_i-\overline{\xi}_n\right)^2$ and recall the definition of $\DeltanB$ and $\DeltanE$ in Section~\ref{sec:expectation_example}, respectively in Equations~\eqref{eq:definition_DeltanB} and~\eqref{eq:definition_DeltanE}. For every $x>0$ and $a>1$
    \begin{align*}
        (i) & \quad \Prob\left( \sqrt{n}\left| \overline{\xi}_n \right|/\sigma > x \right)
        \leq 2\Big\{ \Phi(-x) + \DeltanE \wedge \DeltanB \Big\} \\
        (ii) & \quad \Prob\left( \sqrt{n}\left| \overline{\xi}_n \right|
        > \sigmazeronhat x \right) 
        \leq 2 \Big\{ \Phi(-x/\sqrt{a}) + \DeltanE \wedge \DeltanB \Big\} + \exp\left( -\frac{n(1-1/a)^2}{2\majorantKurtxi} \right) \\
        (iii) & \quad \Prob\left( \sqrt{n}\left| \overline{\xi}_n \right| > \sigmanhat x \right) \leq 2\Big\{ \Phi\big(-x/\sqrt{a(1+x^2/n)}\big) + \DeltanE \wedge \DeltanB \Big\} + \exp\left( -\frac{n(1-1/a)^2}{2\majorantKurtxi} \right).
    \end{align*}
\end{lemma}


Note that we consider the events of (ii) and (iii) without dividing \(\sqrt{n}\left| \overline{\xi}_n \right|\) by \(\sigmazeronhat\) or by \(\sigmanhat\) as the latter random variables can take the value~0 (in that case, the fraction is not well defined).

\medskip

\noindent
{\it Proof of Lemma~\ref{lem:edg_exp}:}
We define $\lambda_3 := \E[\xi^3]/\sigma^3,$
$K_p := \E[|\xi|^p]/\sigma^p$ for any \(p \in \INTST\), and
$\Edg(x) := \dfrac{\lambda_3}{6\sqrt{n}}(1-x^2)\varphi(x)$ for any real number~\(x\).
We obtain
\begin{align}
\label{ch5:eq:edg_exp_1}
    &\Prob\left( \frac{\sqrt{n}\left|\overline{\xi}_n\right|}{\sigma} > x \right)
    \leq 1 - \Prob\left( \frac{\sqrt{n} \, \overline{\xi}_n}{\sigma} \leq x \right) + \Prob\left( \frac{\sqrt{n} \, \overline{\xi}_n}{\sigma} \leq -x \right) \nonumber \\
    &\leq 1 - \left( \Prob\left( \frac{\sqrt{n} \, \overline{\xi}_n}{\sigma} \leq x \right) - \Phi(x) - \Edg(x) \right)
    + \left( \Prob\left( \frac{\sqrt{n} \, \overline{\xi}_n}{\sigma} \leq -x \right) - \Phi(-x) - \Edg(-x) \right) \nonumber \\
    & \hspace{7.5cm} - \Phi(x) - \Edg(x) + \Phi(-x) + \Edg(-x) \nonumber \\
    &\leq 2\Phi(-x) + \left| \Prob\left( \frac{\sqrt{n} \, \overline{\xi}_n}{\sigma} \leq x \right) - \Phi(x) - \Edg(x) \right|
    + \left| \Prob\left( \frac{\sqrt{n} \, \overline{\xi}_n}{\sigma} \leq -x \right) - \Phi(-x) - \Edg(-x) \right| \nonumber \\
    & \leq 2\left\{\Phi(-x)+\DeltanE\right\},
\end{align}
where in the before-last line, we combine parity of $\Edg$ with the fact that $\Phi(x) = 1-\Phi(-x)$. In a similar fashion, we can write
\begin{align}
\label{ch5:eq:be_exp_1}
    &\Prob\left( \frac{\sqrt{n}\left|\overline{\xi}_n\right|}{\sigma} > x \right)
    \leq 1 - \Prob\left( \frac{\sqrt{n} \, \overline{\xi}_n}{\sigma} \leq x \right) + \Prob\left( \frac{\sqrt{n} \, \overline{\xi}_n}{\sigma} \leq -x \right) \nonumber \\
    &\leq 1 - \Prob\left( \frac{\sqrt{n} \, \overline{\xi}_n}{\sigma} \leq x \right) - (1-\Phi(x))
    + \Prob\left( \frac{\sqrt{n} \, \overline{\xi}_n}{\sigma} \leq -x \right) - \Phi(-x) + (1-\Phi(x)) + \Phi(-x) \nonumber \\
    & \leq 2\left\{\Phi(-x)+\DeltanB\right\},
\end{align}
Combining~\eqref{ch5:eq:edg_exp_1} and~\eqref{ch5:eq:be_exp_1} yields the first result of the lemma.

\medskip
 
For the second result, we use Theorem 2.19 in~\cite{pena2008self} which allows us to write for every $a > 1$
\begin{align}
\label{eq:left_tail_exp_concentration}
    & \Prob\left( \frac{\widehat{\sigma}_0^2}{\sigma^2} < \frac{1}{a} \right)
    \leq \exp\left( -\frac{n(1-1/a)^2}
    {2\majorantKurtxi} \right).
\end{align}
Combining Lemma~\ref{lem:edg_exp}(i) and~\eqref{eq:left_tail_exp_concentration}, we can claim for every $x>0$ and $a>1$
\begin{align*}
    \Prob\left( \sqrt{n}\left| \overline{\xi}_n \right| > \widehat{\sigma}_0 x \right)
    & \leq \Prob\left( \sqrt{n}\left| \overline{\xi}_n \right|/\sigma > x/\sqrt{a} \right) 
    + \Prob\left( \frac{\widehat{\sigma}_0^2}{\sigma^2}<\frac{1}{a} \right) \\
    & \leq 2\Big\{ \Phi(-x/\sqrt{a}) + \DeltanE \wedge \DeltanB \Big\}
    + \exp\left( - \frac{n(1-1/a)^2}{2\majorantKurtxi} \right),
\end{align*}
which concludes the proof of result (ii).

\medskip

The final result of the lemma can be proved as follows. We can write
\begin{align*}
    \Prob\left( \sqrt{n}\left| \overline{\xi}_n \right| > \sigmanhat x \right) \leq \Prob\left( \underbrace{ \left\{ \sqrt{n}\left| \overline{\xi}_n \right| > \sigmanhat x\right\} \cap \left\{\frac{\sigmazeronhat^2}{\sigma^2} \geq \frac{1}{a}\right\} }_{=: \mathcal{A}} \right) + \Prob\left( \frac{\sigmazeronhat^2}{\sigma^2} < \frac{1}{a} \right).
\end{align*}
On $\mathcal{A}$, we have $\sigmazeronhat^2/\sigma^2 \geq 1/a$ which implies that $\sigmazeronhat^2>0$
(remember that $\sigma^2>0$ by assumption)
and $T_n := \sqrt{n}\left| \overline{\xi}_n \right|/\sigmazeronhat$ is well-defined. Lemma~\ref{lem:deterministic_equality} therefore ensures on $\mathcal{A}$ that
\begin{align*}
    \left\{ \sqrt{n}\left| \overline{\xi}_n \right| > \sigmanhat x\right\} & = \left\{ \sqrt{n}\left| \overline{\xi}_n \right|/\sigmazeronhat > (\sigmanhat/\sigmazeronhat) x\right\} \\
    & = \left\{ \left| T_n \right| > \sqrt{1-\frac{T_n^2}{n}} x\right\} \\
    & = \left\{ |T_n| > x\left(1+\frac{x^2}{n}\right)^{-1/2} \right\}.
\end{align*}
Using again that $\sigmazeronhat^2/\sigma^2 \geq 1/a$ on $\mathcal{A}$ together with~\eqref{eq:left_tail_exp_concentration} and Lemma~\ref{lem:edg_exp}(i), we can write
\begin{align*}
    \Prob\left( \sqrt{n}\left| \overline{\xi}_n \right| > \sigmanhat x \right) & \leq \Prob\left( \sqrt{n}\left| \overline{\xi}_n \right| / \sigma > x/\sqrt{a(1+x^2/n)} \right) + \Prob\left( \frac{\sigmazeronhat^2}{\sigma^2} < \frac{1}{a} \right) \\
    & \leq 2\Big\{ \Phi\big(-x/\sqrt{a(1+x^2/n)}\big) + \DeltanE \wedge \DeltanB \Big\} + \exp\left( -\frac{n(1-1/a)^2}{2\majorantKurtxi} \right),
\end{align*}
which is result (iii). $\Box$

\begin{lemma}
    \label{lem:deterministic_equality}
    Let \(n \geq 1\) and $(\xi_i)_{i=1}^n$ be an \iid{} sequence of random variables
    with $\E[\xi] = 0$. Let \(\overline{\xi}_n := n^{-1} \sumiunn \xi_i\), $\sigmazeronhat^2 := n^{-1}\sum_{i=1}^n \xi_i^2$ and $\sigmanhat^2 := n^{-1}\sum_{i=1}^n \left(\xi_i-\overline{\xi}_n\right)^2$. If $\sigmazeronhat^2>0$, then \(\sigmanhat^2/\sigmazeronhat^2 = 1 - T_n^2/n\),
    with $T_n := \sqrt{n}\overline{\xi}_n/\sigmazeronhat$.
\end{lemma}

\medskip

\noindent
{\it Proof of Lemma~\ref{lem:deterministic_equality}:}
We first note that the following is always true: $\sigmanhat^2 = \sigmazeronhat^2-\overline{\xi}_n^2$. Dividing on both side by $\sigmazeronhat^2$ which is positive and  using the definition of $T_n$ yields the result. $\Box$

\medskip

\begin{lemma}
    Let \(\Phi\) denote the c.d.f of the standard Gaussian \(\Normale(0,1)\) distribution.
    Let $A$ be a real-valued random variable, and $F$ be its c.d.f.
    Then for any $x, \alpha \in \Rb$,
    we have
    \begin{enumerate}
        \item[(i)] $\left| \Prob\left( |A| > x \right) - \alpha \right|
        \leq 2 \|F - \Phi \|_\infty
        + \left| 2 \Phi(-x) - \alpha \right|,$
        
        \item[(ii)] $\left| \Prob\left( |A| > x \right) - \alpha \right|
        \leq 2 \left\|F - (\Phi + \Edg) \right\|_\infty
        + \left| 2 \Phi(-x) - \alpha \right|,$
    \end{enumerate}
    where $\Edg(x) := \dfrac{\lambda}{6\sqrt{n}}(1-x^2)\varphi(x)$
    for $x \in \Rb$, $n \in \INTST$, $\lambda\in\Rb$.
    \label{lemma:bound_cdf_alpha_Phi}
\end{lemma}

A careful examination of the following proof shows that Statement (ii) in Lemma~\ref{lemma:bound_cdf_alpha_Phi} still holds as long as $\Edg$ is replaced by any function which is even and continuous.

\begin{proof}
We first prove (i).
Using the triangle inequality, we obtain
\begin{align}
\label{eq:detail_univ_conv_sigma_known_1}
    \left| \Prob\left( |A| > x \right) - \alpha \right| 
    & \leq \Big| \Prob( A > x) 
    + \Prob( A < -x) - \alpha \Big| \nonumber \\
    &\leq \left| \Prob( A \leq x) - \Phi(x) \right| 
    + \left| \Prob( A < -x ) - \Phi(-x) \right| 
    + \left| 2\Phi(-x) - \alpha \right| \nonumber \\
    &\leq \|F - \Phi \|_\infty
    + \left| \Prob( A < -x) - \Phi(-x) \right| 
    +  \left| 2 \Phi(-x) - \alpha\right|.
\end{align}
We remark that
\begin{align*}
     \Prob( A < -x) - \Phi(-x)
     \leq F(-x) - \Phi(-x)
     \leq \|F - \Phi \|_\infty.
\end{align*}
Moreover, for every $\eps > 0$, we have
\begin{align*}
    \Phi(-x) - \Prob( A < -x)
    &\leq \Phi(-x) - \Prob( A \leq -x - \eps) \\
    &= \Phi(-x) - \Phi(-x - \eps)
    + \Phi(-x - \eps) - \Prob( A \leq -x - \eps) \\
    &\leq \Phi(-x) - \Phi(-x - \eps)
    + \|F - \Phi \|_\infty.
\end{align*}
By continuity of $\Phi$, in the limit when $\eps \to 0$, we obtain
\begin{align*}
    \Phi(-x) - \Prob( A < -x)
    &\leq \|F - \Phi \|_\infty.
\end{align*}
This shows that
$\left| \Prob( A < -x) - \Phi(-x) \right|
\leq  \|F - \Phi \|_\infty$.
Combining this inequality with \eqref{eq:detail_univ_conv_sigma_known_1} finishes the proof of (i).

\medskip

We now prove (ii).
Using the triangle inequality and the parity of the function $\Edg$, we obtain
\begin{align}
    \label{eq:detail_univ_conv_sigma_known_2}
    & \left| \Prob\left( |A| > x \right) - \alpha \right| 
    \leq \Big| \Prob( A > x) 
    + \Prob( A < -x) - \alpha \Big| \nonumber \\
    &\leq \left| \Prob( A \leq x) - \Phi(x) - \Edg(x) \right| 
    + \left| \Prob( A < -x ) - \Phi(-x) - \Edg(-x) \right| 
    + \left| 2\Phi(-x) - \alpha \right| \nonumber \\
    &\leq \|F - \Phi - \Edg \|_\infty
    + \left| \Prob( A < -x) - \Phi(-x) - \Edg(-x) \right| 
    +  \left| 2 \Phi(-x) - \alpha \right|.
\end{align}
We remark that
\begin{align*}
     \Prob( A < -x) - \Phi(-x) - \Edg(-x)
     \leq F(-x) - \Phi(-x) - \Edg(-x)
     \leq \|F - \Phi - \Edg \|_\infty.
\end{align*}
Let $\widetilde{\Phi} := \Edg + \Phi$.
Moreover, for every $\eps > 0$, we have
\begin{align*}
    \widetilde{\Phi}(-x) - \Prob( A < -x)
    &\leq \widetilde{\Phi}(-x) - \Prob( A \leq -x - \eps) \\
    &= \widetilde{\Phi}(-x) - \widetilde{\Phi}(-x - \eps)
    + \widetilde{\Phi}(-x - \eps) - \Prob( A \leq -x - \eps) \\
    &\leq \widetilde{\Phi}(-x) - \widetilde{\Phi}(-x - \eps)
    + \|F - \widetilde{\Phi} \|_\infty.
\end{align*}
By continuity of $\widetilde{\Phi}$, in the limit when $\eps \to 0$, we obtain
\begin{align*}
    \widetilde{\Phi}(-x) - \Prob( A < -x)
    &\leq \|F - \widetilde{\Phi} \|_\infty.
\end{align*}
This show that
$\left| \Prob( A < -x) - \Phi(-x) - \Edg(- x) \right|
\leq  \|F - \Edg - \Phi\|_\infty$.
Combining this inequality with \eqref{eq:detail_univ_conv_sigma_known_2} finishes the proof of (ii).
\end{proof}

\end{document}